%% file: oseen-eigenvalue.tex
\documentclass[onefignum,onetabnum]{siamart190516}

\input{ex_shared}
\ifpdf
\hypersetup{
  pdftitle={Finite Element Analysis of the Oseen eigenvalue problem},
  pdfauthor={Felipe Lepe, Gonzalo Rivera and Jesus Vellojin}
}
\fi
\def\CE{\mathcal{E}}

\def\CT{{\mathcal T}}

\begin{document}

\maketitle

\begin{abstract}
We propose and analyze a finite element method for the Oseen eigenvalue problem. This problem is an extension of the Stokes eigenvalue problem, where the presence of the convective term leads to a non-symmetric problem and hence, to complex eigenvalues and eigenfunctions.  With the aid of the compact operators theory, we prove that for inf-sup stable finite elements the convergence holds and hence, error estimates for the eigenvalues and eigenfunctions are derived. We also propose an a posteriori error estimator which results to be reliable and efficient. We report a series of numerical tests in two and three dimension in order to assess the performance of the method and the proposed estimator.

\end{abstract}



\section{Introduction}\label{sec:intro}
In fluid mechanics, the knowledge of the eigenvalues and eigenfunctions associated to spectral problems are important since they are the core in the analysis of stability of certain systems.  Since several partial differential equations are difficult to solve analytically, the development of numerical techniques to approximate accurately the solutions is such equations play an important role in mathematics and engineering. In the case of the Stokes eigenvalue problem, the literature is abundant, where several formulations and numerical methods have emerged. On this subject we mention \cite{MR3197278,MR3095260,MR4071826,MR3864690,MR4022421,MR4077220,LRVSISC,MR2473688,MR4229296,MR3335223} and the references therein.

The Oseen problem appears as a linearization of the Navier-Stokes equations. For the best of the author's knowledge, the eigenvalue problem associated to the Ossen system has not been analyzed, at least from a numerical point of view. This problem, contrary to the Stokes eigenvalue problem, results to be non-selfadjoint due to the presence of the convective term. Hence, the eigenfunctions and eigenvalues must be found naturally on  complex spaces, which clearly is a difference with the Stokes eigensystem. This is of course the theoretical implication since in practice, particularly on the computational tests, it is possible to recover real solutions instead of complex eigenvalues and eigenfunctions. The relation between the viscosity and the convective velocity of the system promotes the appearance of complex eigenvalues. However, for real applications of Stokes and Oseen, the Reynolds number is assumed to be small, and therefore it is more feasible to recover real eigenvalues \cite{John2016}.

The non-symmetric eigenvalue problems need a different treatment for analysis compared with the symmetric problems, particularly in the incorporation of the dual problems. This is needed for convergence in norm and error analysis for numerical methods as is shown in \cite{MR3452773,MR4660149, MR4320859,MR3895875,MR4497827}. Clearly the research on non-symmetric eigenvalue problems is in ongoing progress where the results for different partial differential equations and methods to solve the spectral problems are emerging.

Our contribution is a conforming finite element method to approximate the eigenvalues and eigenfunctions of the Oseen eigenvalue problem. The Oseen problem  is clearly different from the Stokes problem due to the presence of the convective term. Hence, the natural question is if the convective term goes to zero on the eigenvalue problem, the spectrum of Oseen approaches to the spectrum of the Stokes eigensystem as it happen in other eigenvalue problems like in   \cite{MR4570534}, where the mixed elasticity eigenvalue problem converge to the Stokes problem when the corresponding Lam\'e constant explodes. In the case of the Oseen eigenvalue problem, we focus when the convective vector goes to zero in norm and the convergence in the limit to the Stokes spectrum. This must be reflected on the computational experiments.

Related to the numerical method, we focus on inf-sup stable conforming finite elements for the Stokes problem. In particular, we consider in our method two families: the mini element and Taylor-Hood finite elements. Both families are suitable choices for the approximation and our intention is to compare their performance on our problem. With these finite elements we derive a priori error estimates  which we derive according to the compact operators theory presented in  \cite{MR1115235,MR2652780}. 


The outline of our manuscript is as follows: In Section \ref{sec:model_problem} we present the Oseen eigenvalue problem and its variational formulation written in terms of the velocity and pressure. We recall key  properties of the problem such as stability and regularity of the eigenfunctions, leading to the introduction of the corresponding solution operator. Since our problem is non-selfadjoint, we also introduce the adjoint problem and its properties. For completeness, we present a result that relates the spectrum of the Oseen eigenvalue problem with the Stokes one. In Section \ref{sec:fem}  we present the finite element method in which our analysis is supported. This numerical scheme consists in the inf-sup stable families for the mini element and Taylor-Hood elements. We report convergence and a priori error estimates for the eigenvalues and eigenfunctions. In Section \ref{sec:apost} we introduce an a posteriori error estimator for the primal and adjoint eigenvalue problems that result to be reliable and efficient. Finally, in Section \ref{sec:numerics} we report a series of numerical tests to analyze the performance and accuracy of the methods in two and three dimensions. We also perform several tests in order to analyze the behavior of the a posteriori estimator, where two and three dimensional domains with singularities are considered.


\subsection{Notations and preliminaries}
Throughout this work, we will use notations that will allow a smoother reading of the content. Let us set these notations. Given $d \in\{2,3\}$, we denote $\mathbb{C}^{d \times d}$ as  the space of square real matrices of order $d$, where $\mathbb{I}:=\left(\delta_{i j}\right) \in \mathbb{C}^{d \times d}$ denotes the identity matrix. Given $\boldsymbol{A}:=\left(A_{i j}\right), \boldsymbol{B}:=\left(B_{i j}\right) \in \mathbb{C}^{n \times n}$, we define 
$$
 \boldsymbol{A}: \boldsymbol{B}:=\sum_{i, j=1}^d A_{i j} \overline{B}_{i j},
$$
as the tensorial product between $\boldsymbol{A}$ and $\boldsymbol{B}$. The entry $\overline{B}_{i j}$ represent the complex conjugate of $B_{i j}$. Similarly, given two vectors $\bs:=(s_i), \br:=(r_i) \in  \mathbb{C}^d$, we define the products
$$
\bs\cdot\br:= \sum_{i=1}^d s_i \overline{r}_i, \qquad \bs \otimes \br : = \bs \overline{\br}^\texttt{t}  = \sum_{i=1}^d\sum_{j=1}^d s_i \overline{r}_j,
$$
as the dot and dyadic product in $\mathbb{C}$, respectively, where $(\cdot)^\texttt{t}$ denotes the transpose operator.

In turn, in what follows we will resort to a standard simplified terminology for Sobolev spaces and norms. Let $\mathcal{O}$ be a subset of $\mathbb{R}^d$ with Lipschitz boundary $\partial\mathcal{O}$.   For $r \geq 0$ and $p \in[1, \infty]$, we denote by $\mathrm{L}^p(\mathcal{O},\mathbb{C})$ and $\mathrm{W}^{r, p}(\mathcal{O},\mathbb{C})$ the usual Lebesgue and Sobolev spaces of maps from $\mathcal{O}$ to $\mathbb{C}$, and endowed with the norms $\|\cdot\|_{\mathrm{L}^p(\mathcal{O})}$ and $\|\cdot\|_{\mathrm{W}^{r, p}(\mathcal{O})}$, respectively, where  $\mathrm{W}^{0, p}(\mathcal{O},\mathbb{C})=\mathrm{L}^p(\mathcal{O},\mathbb{C})$. If $p=2$, we write $\mathrm{H}^r(\mathcal{O},\mathbb{C})$ instead of $\mathrm{W}^{r, 2}(\mathcal{O},\mathbb{C})$, and denote the corresponding Lebesgue and Sobolev norms by $\|\cdot\|_{0, \mathcal{O}}$ and $\|\cdot\|_{r, \mathcal{O}}$, respectively. As usual, we write $|\cdot|_{r, \mathcal{O}}$ to denote the seminorm. In particular, for $p=\infty$, we will denote $\Vert\cdot\Vert_{\infty,\mathcal{O}}$ as the induced norm over the space $\mathrm{W}^{1,\infty}(\mathcal{O},\mathbb{C})$.  We define $\mathrm{H}_0^1(\mathcal{O},\mathbb{C})$ as the space of functions in $\mathrm{H}^1(\mathcal{O},\mathbb{C})$ with vanishing trace on $\partial \mathcal{O}$, and $\mathrm{L}_0^2(\mathcal{O},\mathbb{C})$ as the space of $\mathrm{L}^2(\mathcal{O},\mathbb{C})$ functions with vanishing mean value over $\mathcal{O}$.

\section{The model problem}
\label{sec:model_problem}

Let  $\O\subset\mathbb{R}^d$,  with $d\in\{2,3\}$,  be an open bounded polygonal/polyhedral domain with Lipschitz boundary $\partial\O$. 
The problem of interest in this study is based on the Oseen equations, whose main characteristic is to model incompressible  viscous fluids at small Reynolds number. The corresponding strong form of the equations are given as:
\begin{equation}\label{def:oseen-eigenvalue}
\left\{
\begin{array}{rcll}
-\nu\Delta \bu + (\boldsymbol{\beta}\cdot\nabla)\bu + \nabla p&=&\lambda\bu,&\text{in}\,\O,\\
\div \bu&=&0,&\text{in}\,\O,\\
\displaystyle\int_{\O} p &=&0, &\text{in}\,\O,\\
\bu &=&\boldsymbol{0},&\text{in}\,\partial\O,
\end{array}
\right.
\end{equation}
where $\bu$ is the displacement, $p$ is the pressure and $\boldsymbol{\beta}$ is a given vector field, representing a \textit{steady flow velocity} such that $\boldsymbol{\beta}\in \mathrm{W}^{1,\infty}(\O,\mathbb{C})^d$ is divergence free and $\nu>0$ is the kinematic viscosity. Over this last  parameter, we assume that there exists two positive  numbers $\nu^+$ and $\nu^{-}$ such that $\nu^{-}< \nu< \nu^{+}$. 

The standard assumptions on the coefficients are the following (see \cite{John2016})
\begin{itemize}
\item $\|\boldsymbol{\beta}\|_{\infty,\O}\sim 1$ if $\nu\leq \|\boldsymbol{\beta}\|_{\infty,\O}$,
\item $\nu\sim 1$ if $\|\boldsymbol{\beta}\|_{\infty,\O}<\nu$,
\end{itemize}
where the first point is the case more close to the real applications. Now we need some regularity assumptions on the convective coefficient. In two dimensions, le us assume that there exists $\varepsilon_1>0$ such that $\boldsymbol{\beta}\in\L^{2+\varepsilon_1}(\O,,\mathbb{C})^d$ and in three dimensions $\boldsymbol{\beta}\in\L^3(\O,\mathbb{C})^d$. With these regularity assumptions at hand, we have the skew-symmetry property of the convective term (see \cite[Remark 5.6]{John2016}) which claims that for all $\bv\in\H_0^1(\O,\mathbb{C})^d$, there holds
\begin{equation}
\label{eq:skew}
\int_{\O}(\boldsymbol{\beta}\cdot\nabla)\bv\cdot\bv=0\quad\forall\bv\in\H_0^1(\O,\mathbb{C})^d.
\end{equation}

Now we introduce a variational formulation for \eqref{def:oseen-eigenvalue}. To simplify the presentation fo the material, let us define the spaces $\mathcal{X}:=\H_0^1(\O,\mathbb{C})^d\times \L_0^2(\O,\mathbb{C})$ and its dual space $\mathcal{X}'$ together with the space space $\mathcal{M}:=\H_0^1(\O,\mathbb{C})^d\times\H_0^1(\O,\mathbb{C})^d$ and its dual $\mathcal{M}'$. For the space $\mathcal{X}$  we define the norm $\|\cdot\|_{\mathcal{X}}^2:=\|\cdot\|_{1,\O}^2+\|\cdot\|_{0,\O}^2$ whereas for $\mathcal{M}$ the norm will be $\|(\bv,\bw)\|_{\mathcal{M}}^2=\|\bv\|_{1,\O}^2+\|\bw\|_{1,\O}^2$, for all $(\bv,\bw)\in\mathcal{M}$. 

We follow the usual approach as in the Stokes model and we multiply such a system with suitable tests functions, integrate by parts, and use the boundary conditions in order to obtain the following bilinear forms $a:\mathcal{M}\rightarrow\mathbb{C}$
and $b:\mathcal{X}\rightarrow\mathbb{C}$ defined by 
\begin{equation*}
a(\bu,\bv):=\int_{\Omega}\nu\nabla\bu :\nabla\bv + \int_{\Omega}(\boldsymbol{\beta}\cdot\nabla)\bu\cdot\bv\quad
\text{and}
\quad
b(\bv,q):=-\int_{\O}q\,\div\bv.
\end{equation*}
Observe that the resulting eigenvalue problem will be non-symmetric since the bilinear form $a(\cdot,\cdot)$ is precisely a non-symmetric bilinear form and, rigorously speaking, the eigenvalues are complex.  
With these sesquilinear forms at hand, we write the following weak formulation for problem \eqref{def:oseen-eigenvalue}: Find $\lambda\in\mathbb{C}$ and $(\boldsymbol{0},0)\neq(\bu,p)\in \mathcal{X}$ such that 
\begin{equation}\label{def:oseen_system_weak}
	\left\{
	\begin{array}{rcll}
a(\bu,\bv) + b(\bv,p)&=&\lambda(\bu,\bv)&\forall \bv\in \H_0^1(\O,\mathbb{C})^d,\\
b(\bu,q)&=&0&\forall q\in \L_0^2(\O,\mathbb{C}).
\end{array}
	\right.
\end{equation}

Let us define the kernel $\mathcal{K}$ of $b(\cdot,\cdot)$ as follows
\begin{equation*}
\mathcal{K}:=\{\bv\in\H_0^1(\O,\mathbb{C})^d\,:\,  b(\bv, q)=0\,\,\,\,\forall q\in\L_0^2(\O,\mathbb{C})\}.
\end{equation*}
With this space at hand, is easy to check with the aid of \eqref{eq:skew} that  $a(\cdot,\cdot)$ is $\mathcal{K}$-coercive. Moreover, the bilinear form $b(\cdot,\cdot)$ satisfies the following inf-sup condition
\begin{equation}
\label{ec:inf-sup_cont}
\displaystyle\sup_{\btau\in\H_0^1(\O,\mathbb{C})^d}\frac{b(\btau,q)}{\|\btau_h\|_{1,\O}}\geq\beta\|q\|_{0,\O}\quad\forall q\in\L^2_0(\O,\mathbb{C}).
\end{equation}
Hence, we introduce the so-called solution operator, which we denote by $\bT$ and is defined as follows
\begin{equation}\label{eq:operador_solucion_u}
	\bT:\L^2(\O,\mathbb{C})^d\rightarrow \H^1(\O,\mathbb{C})^d,\qquad \boldsymbol{f}\mapsto \bT\boldsymbol{f}:=\widehat{\bu}, 
\end{equation}
where the pair  $(\widehat{\bu}, \widehat{p})\in\mathcal{X}$ is the solution of the following well posed source problem
\begin{equation}\label{def:oseen_system_weak_source}
	\left\{
	\begin{array}{rcll}
a(\widehat{\bu}, \bv)+b(\bv,\widehat{p})&=&(\boldsymbol{f},\bv)&\forall \bv\in \H_0^1(\O,\mathbb{C})^d,\\
b(\widehat{\bu},q)&=&0&\forall q\in \L_0^2(\O,\mathbb{C}),
\end{array}
	\right.
\end{equation}
implying  that $\bT$ is well defined due to the Babu\^{s}ka-Brezzi theory. Moreover, 
we have the following estimates for the velocity and pressure \cite[Lemma 5.8]{John2016}
\begin{equation}\label{eq:estimatefuente_velocity}
\ds \nu\|\nabla\widehat{\bu}\|_{0,\O}\leq C_{p}\|\boldsymbol{f}\|_{0,\O},
\end{equation}
where $C_{p}>0$ represents the constant of the  Poincar\'e-Friedrichs inequality.
whereas for the pressure we have
\begin{equation*}
\label{eq:estimatefuente_pressure}
\ds \|\widehat{p}\|_{0,\O}\leq \frac{1}{\beta}\left( C_{p}\|\boldsymbol{f}\|_{0,\O}+\nu^{1/2}
\|\nabla\widehat{\bu}\|_{0,\O}\left(\nu^{1/2}+\dfrac{C_p\|\boldsymbol{\beta}\|_{\infty,\O}}{\nu^{1/2}}\right)\right),
\end{equation*}
where $\beta$ is the constant associated with the inf-sup condition \eqref{ec:inf-sup_cont}.

We observe that $(\lambda , (\bu , p)) \in  \mathbb{C}  \times \mathcal{X}$  solve \eqref{def:oseen_system_weak} if and only if $(\kappa , \bu )$ is an eigenpair of $\bT , i.e., \bT \bu  = \kappa \bu$  with $\kappa  := 1/\lambda$.

Using the fact that the convective term is well defined and taking advantage of the well known Stokes regularity results (see \cite{MR975121,MR1600081} for instance), we have  the following additional regularity result for the solution of the source problem \eqref{def:oseen_system_weak_source} and consequently, for the generalized eigenfunctions of $\bT$.
\begin{theorem}\label{th:regularidadfuente}
There exists $s>0$  that for all $\boldsymbol{f} \in \L^2(\O,\mathbb{C})^d$, the solution $(\widehat{\bu},\widehat{p})\in\mathcal{X}$ of problem \eqref{def:oseen_system_weak_source}, satisfies for the velocity $\widehat{\bu}\in  \H^{1+s}(\Omega,\mathbb{C})^d$, for the pressure $\widehat{p}\in \H^s(\Omega,\mathbb{C})$, and
 \begin{equation*}
\|\widehat{\bu}\|_{1+s,\O}+\|\widehat{p}\|_{s,\O}\leq C \|\boldsymbol{f}\|_{0,\O},
 \end{equation*}
 where $C:=\dfrac{C_p}{\beta}\max\left\{1,\dfrac{C_p\|\boldsymbol{\beta}\|_{\infty,\O}}{\nu}\right\}$.
\end{theorem}
Hence, because of the compact inclusion $\H^{1+s}(\O,\mathbb{C})^d\hookrightarrow \L^2(\O,\mathbb{C})^d$, we can
conclude that $\bT$ is a compact operator and we obtain the following  spectral characterization of $\bT$ holds.
\begin{lemma}(Spectral Characterization of $\bT$).
The spectrum of $\bT$ is such that $\sp(\bT)=\{0\}\cup\{\kappa_{k}\}_{k\in{N}}$ where $\{\kappa_{k}\}_{k\in\mathbf{N}}$ is a sequence of complex eigenvalues that converge to zero, according to their respective multiplicities. 
\end{lemma}
\subsection{Relation between the Oseen and the Stokes eigenvalue problems}
From \eqref{def:oseen-eigenvalue} we observe that if $\boldsymbol{\beta}=\boldsymbol{0}$ we obtain the Stokes eigenvalue problem. Hence, a natural question to answer is related to the relation between the spectrums of the Oseen and the Stokes eigenvalue problems. More precisely, the convergence of the solution operator \eqref{eq:operador_solucion_u} and the solution operator associated to the Stokes spectral problem. Let us begin by introducing the Stokes spectral problem: Find $\lambda^S\in\mathbb{R}$ and  $(\boldsymbol{0},0)\neq(\bu^S,p^S)\in\mathcal{X}$ such that 
\begin{equation}\label{def:stokes_system_weak}
	\left\{
	\begin{array}{rcll}
a^S(\bu^S,\bv) + b^S(\bv,p^S)&=&\lambda^S(\bu^S,\bv)&\forall \bv\in \H_0^1(\O,\mathbb{C})^d,\\
b^S(\bu^S,q)&=&0&\forall q\in \L_0^2(\O,\mathbb{C}),
\end{array}
	\right.
\end{equation}
where
\begin{equation*}
a^S(\bu^S,\bv):=\int_{\Omega}\nu\nabla\bu^S :\nabla\bv\quad
\text{and}
\quad
b^S(\bv,p^S):=-\int_{\O}p^S\,\div\bv.
\end{equation*}
Let us  denote by $\bT^S$ the solution operator associated to \eqref{def:stokes_system_weak},  defined by
\begin{equation}\label{eq:operador_solucion_u_stokes}
	\bT^S:\L^2(\O,\mathbb{C})^d\rightarrow \H_0^1(\O,\mathbb{C})^d,\qquad \boldsymbol{f}\mapsto \bT^S\boldsymbol{f}:=\widehat{\bu}^S, 
\end{equation}
where the pair  $(\widehat{\bu}^S, \widehat{p}^S)\in\mathcal{X}$ is the solution of the following well posed source problem
\begin{equation}\label{def:stokes_system_weak_source}
	\left\{
	\begin{array}{rcll}
a^S(\widehat{\bu}^S, \bv)+b^S(\bv,\widehat{p}^S)&=&(\boldsymbol{f},\bv)&\forall \bv\in \H_0^1(\O,\mathbb{C})^d,\\
b^S(\widehat{\bu}^S,q)&=&0&\forall q\in \L_0^2(\O,\mathbb{C}),
\end{array}
	\right.
\end{equation}

The following result establish that the operator  $\bT$ converge to  $\bT^S$ as $\|\boldsymbol{\beta}\|_{\infty,\O}$ tends to zero.
\begin{lemma}
Let $\bT$ be the operator defined in \eqref{eq:operador_solucion_u} and $\bT^S$ the operator defined in \eqref{eq:operador_solucion_u_stokes}. Then, for any $\boldsymbol{f}\in\L^2(\O,\mathbb{C})^d$ the following estimate holds
\begin{equation*}
\|(\bT-\bT^S)\boldsymbol{f}\|_{1,\O}\leq C_\nu\|\boldsymbol{\beta}\|_{\infty,\O}\|\boldsymbol{f}\|_{0,\O}.
\end{equation*} 
\end{lemma}
\begin{proof}
Let $\boldsymbol{f}\in\L^2(\O,\mathbb{C})^d$. Subtracting \eqref{def:stokes_system_weak_source} from \eqref{def:oseen_system_weak_source} we obtain
\begin{align*}
\displaystyle\int_{\O}\nu\nabla(\widehat{\bu}-\widehat{\bu}^S):\nabla\bv+\int_{\O}(\boldsymbol{\beta}\cdot\nabla)\widehat{\bu}\cdot\bv-\int_{\O}(\widehat{p}-\widehat{p}^S)\div\bv&=0\quad\forall \bv\in \H_0^1(\O,\mathbb{C})^d\\
-\int_{\O}\div(\widehat{\bu}-\widehat{\bu}^S)q&=0\quad\forall q\in \L_0^2(\O,\mathbb{C}).
\end{align*}
Testing the above system with $\bv=\widehat{\bu}-\widehat{\bu}^S$ and $q=\widehat{p}-\widehat{p}^S$, and using the incompressibility conditions that the velocity fields of the Oseen and Stokes source problem satisfy, we obtain 
\begin{equation}
\label{eq:bound1}
\nu\|\nabla(\widehat{\bu}-\widehat{\bu}^S)\|_{0,\O}^2\leq \|\boldsymbol{\beta}\|_{\infty,\O}\|\nabla\widehat{\bu}\|_{0,\O}\|\widehat{\bu}-\widehat{\bu}^S\|_{0,\O}.
\end{equation}

On the other hand, from Poincar\'e inequality there exists a constant $C_p>0$ such that $\|\widehat{\bu}-\widehat{\bu}^S\|_{0,\O}\leq C_p\|\nabla(\widehat{\bu}-\widehat{\bu}^S)\|_{0,\O}$. Replacing this in \eqref{eq:bound1} we obtain 
\begin{equation*}
\label{eq:bound2}
\|\nabla(\widehat{\bu}-\widehat{\bu}^S)\|_{0,\O}\leq C_p\frac{\|\boldsymbol{\beta}\|_{\infty,\O}}{\nu}\|\nabla\widehat{\bu}\|_{0,\O}\leq C_p \frac{\|\boldsymbol{\beta}\|_{\infty,\O}}{\nu}\|\boldsymbol{f}\|_{0,\O},
\end{equation*}
where the last estimate is consequence of \eqref{eq:estimatefuente_velocity}. This concludes the proof.
\end{proof}
We conclude this section by redefining the spectral problem \eqref{def:oseen_system_weak} in order to  simplify the notations for the forthcoming analysis. With this in mind, let us introduce the sesquilinear form $A:\mathcal{X}\times\mathcal{X}\rightarrow\mathbb{C}$ defined by 
\begin{equation*}
A((\bu,p);(\bv,q)):= a(\bu,\bv)+ b(\bv,p)- b(\bu,q),\quad\forall (\bv,q)\in\mathcal{X}, 
\end{equation*}
which allows us to rewrite problem \eqref{def:oseen_system_weak} as follows: Find $\lambda\in\mathbb{C}$ and $(\boldsymbol{0},0)\neq(\bu,p)\in\mathcal{X}$ such that
\begin{equation}
\label{eq:eigenA}
A((\bu,p);(\bv,q))=\lambda (\bu,\bv)_{0,\O}\quad\forall (\bv,q)\in\mathcal{X}.
\end{equation}

Since a part of $a(\cdot,\cdot)$ is elliptic,   $b(\cdot,\cdot)$ satisfies an inf-sup condition and $\boldsymbol{\beta}$ is  divergence free, we prove the following stability result holds which will be useful for the a posteriori error analysis.
	\begin{lemma}\label{lem:inf-sup-A}
		The sesquilinear form $A(\cdot,\cdot)$ satisfies the following inf-sup conditions
		\begin{align*}
			\nonumber &\ds\inf_{(\boldsymbol{0},0)\neq(\bw,r)\in\mathcal{X}}\sup_{(\boldsymbol{0},0)\neq(\bv,q)\in\mathcal{X}}\frac{A((\bv,q);(\bw,r))}{\|(\bv,q)\|\|(\bw,r)\|}=\gamma,\\
			&\inf_{(\boldsymbol{0},0)\neq(\bv,q)\in\mathcal{X}}\sup_{(\boldsymbol{0},0)\neq(\bw,r)\in\mathcal{X}}\frac{A((\bv,q);(\bw,r))}{\|(\bv,q)\|\|(\bw,r)\|}=\gamma,
			\label{eq:beta_infsup}
		\end{align*}
		where $\gamma$ is a positive  constant,  uniform with respect to $\nu$. Consequently, given $(\bv,q)\in\mathcal{X}$, there exists a pair  $(\bw,r)\in\mathcal{X}$ and a constant $C>0$ such that
			\begin{equation}
				\label{lem:stability1}
				\begin{aligned}
					&\Vert \bw\Vert_{1,\O} + \Vert r\Vert_{0,\O}\leq C,\\
					&\Vert \bv\Vert_{1,\O} + \Vert q \Vert_{0,\O} \leq A((\bv,q);(\bw,r)).
				\end{aligned}
			\end{equation}
	\end{lemma}
	
	\begin{proof}
		Let $(\boldsymbol{0},0)\neq(\bv,q)\in\mathcal{X}$.  With this pair at hand we define
		\begin{equation*}
		\mathbb{S}:=\sup_{(\boldsymbol{0},0)\neq(\bw,r)\in\mathcal{X}}\frac{A((\bv,q);(\bw,r))}{\|(\bw,r)\|}.
		\end{equation*}
	Since $\boldsymbol{\beta}$ is divergence free, we have that $(\boldsymbol{\beta}\cdot\nabla)\bv,\bv)_{0,\O} =0$, which together with the Poincar\'e inequality allows us  to have
		\begin{equation}
			\label{lem:eq001}
			C_p\nu\Vert\bv\Vert_{1,\O}^2\leq a(\bv,\bv)= A((\bv,q);(\bv,q)) \leq \mathbb{S} \Vert (\bv,q)\Vert.
		\end{equation}
		In turn, we have that $b(\cdot,\cdot)$ satisfies an inf-sup condition. Hence, there exists $\widetilde{\bv}\in \H_0^1(\O,\mathbb{C})^d$ such that $\div \widetilde{\bv} =-q$ together with a constant $C>0$ satisfying  $\Vert\widetilde{\bv}\Vert_{1,\O}\leq C\Vert q\Vert_{0,\O}$.  From this we obtain
		$$
			\begin{aligned}
			\Vert q\Vert_{0,\O}^2 = - \int_{\O} q\div &\widetilde{\bv}  = - \int_{\O} q\, \div \widetilde{\bv}  - a(\bv,\widetilde{\bv}) + a(\bv,\widetilde{\bv})\\
			&= - A((\bv,q);(\widetilde{\bv},0)) + a(\bv,\widetilde{\bv}) \\
			&\leq \mathbb{S}\Vert\widetilde{\bv}\Vert_{1,\O} +\nu\Vert\bv\Vert_{1,\O}\Vert\widetilde{\bv}\Vert_{1,\O} + \Vert \boldsymbol{\beta}\Vert_{\infty,\Omega} \Vert \bv\Vert_{1,\O} \Vert\widetilde{\bv}\Vert_{1,\O}\\
			&\leq  \mathbb{S}\Vert\widetilde{\bv}\Vert_{1,\O} + \mathbb{S}^{1/2}C_p^{-1/2}\left(\nu^{1/2} + \nu^{-1/2}\Vert\boldsymbol{\beta}\Vert_{\infty,\Omega}\right)\Vert (\bv,q)\Vert^{1/2}\Vert\widetilde{\bv}\Vert_{1,\O},
			\end{aligned}
		$$
		where we have used the continuity of $a(\cdot,\cdot)$ and \eqref{lem:eq001}. Exploiting the bound for $\widetilde{\bv}$ and using Young's inequality we obtain
		\begin{equation}
			\label{lem:eq003}
			\Vert q\Vert_{0,\O}^2  \leq 2C \left[\mathbb{S}^2 + \mathbb{S}C_p(\nu+\nu^{-1}\Vert \boldsymbol{\beta}\Vert_{\infty,\Omega}) ^2\Vert (\bv,q)\Vert\right].
		\end{equation}
		Gathering \eqref{lem:eq001} and \eqref{lem:eq003} allows to have
		$$
		\Vert (\bv,q)\Vert^2 \leq \nu^{-1}C_p^{-1}\mathbb{S} \Vert (\bv,q)\Vert + 2C \left[\mathbb{S}^2 + \mathbb{S}C_p(\nu+\nu^{-1}\Vert \boldsymbol{\beta}\Vert_{\infty,\Omega}) ^2\Vert (\bv,q)\Vert\right].
		$$
		Using now the following version of the Young's inequality $ab\leq \frac{a^2}{4} + b^2$  we obtain
		$$
		\Vert (\bv,q)\Vert^2 \leq \nu^{-2}C_p^{-2}\mathbb{S}^2 + \frac{1}{4}\Vert (\bv,q)\Vert^2 + 2C\mathbb{S}^2 +2C \mathbb{S}^2C_p^2(\nu+\nu^{-1}\Vert \boldsymbol{\beta}\Vert_{\infty,\Omega}) ^4 + \frac{1}{4}\Vert (\bv,q)\Vert^2.
		$$
		The resulting bound is then
		$$
		\gamma \Vert (\bv,q)\Vert \leq \mathbb{S},
		$$
		with
		$$
		\gamma :=\frac{1}{2\left\{(\nu)^{-2}C_p^{-2}+2C\left[1 + C_p^2(\nu + \nu^{-1}\Vert\boldsymbol{\beta}\Vert_{\infty,\Omega})^4\right]\right\}^{1/2}}.
		$$
		Finally the bounds \eqref{lem:stability1} follows by noting that 
		$$
		\frac{1}{\sqrt{2}}(\Vert \bv\Vert_{1,\O} + \Vert q \Vert_{0,\O})\leq (\Vert\bv\Vert_{1,\O}^2 + 	\Vert q\Vert_{0,\O}^2)^{1/2}.
		$$
		This concludes the proof.	\end{proof}


\subsection{The adjoint eigenvalue problem}
An important aspect of the Ossen eigenvalue problem is the lack of symmetry due to the presence of the convective term. This is an important fact that leads to a solution operator that is non-selfadjoint. Hence, the dual eigenvalue problem must be considered in order to have complete information of the convergence of the finite element approximation of the spectrum. 

The strong form of the dual equations are given as:
\begin{equation}\label{def:oseen-eigenvalue-dual}
	\left\{
	\begin{array}{rcll}
		-\nu\Delta \bu^* - \div(\bu^*\otimes\boldsymbol{\beta}) - \nabla p&=&\lambda^*\bu^*,&\text{in}\,\O,\\
		-\div \bu^*&=&0,&\text{in}\,\O,\\
		\displaystyle\int_{\O} p^* &=&0, &\text{in}\,\O,\\
		\bu^* &=&\boldsymbol{0},&\text{in}\,\partial\O,
	\end{array}
	\right.
\end{equation}

The dual weak variational formulation from \eqref{def:oseen-eigenvalue-dual} reads as follows: Find $\lambda^*\in\mathbb{C}$ and the pair $(\boldsymbol{0},0)\neq(\bu^*,p^*)\in\mathcal{X}'$ such that  
\begin{equation}\label{def:oseen_system_weak_dual_eigen}
	\left\{
	\begin{array}{rcll}
\widehat{a}(\bv,\bu^*)-b(p^*,\bv)&=&\lambda^*(\bv,\bu^*)&\forall \bv\in \H^{-1}(\O,\mathbb{C})^d,\\
-b(q, \bu^*)&=&0&\forall q\in \L_0^2(\O,\mathbb{C}).
\end{array}
	\right.
\end{equation}
where $\widehat{a}(\bv,\bu^*)=\displaystyle \int_{\Omega}\nu\nabla\bv:\nabla \bu^*+\int_{\Omega}(\bu^* \otimes \boldsymbol{\beta}):\nabla \bv$. From \eqref{def:oseen_system_weak}, \eqref{def:oseen_system_weak_dual_eigen} and integration by parts we obtain
\begin{equation*}
a(\bu,\bu^*) - \widehat{a}(\bu,\bu^*)=\int_{\O}(\boldsymbol{\beta}\cdot\nabla)\bu\cdot\bu^*  -  \int_{\O}(\boldsymbol{\beta}\cdot\nabla)\bu\cdot\bu^* = 0,
\end{equation*}
hence, we readily find that $\lambda = \lambda^*$.

Now we introduce the adjoint of \eqref{eq:operador_solucion_u} defined  by 
\begin{equation}\label{eq:operador_adjunto_solucion_u}
	\bT^*:\L^{2}(\O,\mathbb{C})^d\rightarrow \H^{-1}(\O,\mathbb{C})^d,\qquad \boldsymbol{f}\mapsto \bT\boldsymbol{f}:=\widehat{\bu}^*, 
\end{equation} 
where $\widehat{\bu}^*\in\H^{-1}(\O,\mathbb{C})$ is the adjoint velocity of $\widehat{\bu}$ and solves the following adjoint source  problem: Find  
$(\widehat{\bu}^*, \widehat{p}^*)\in\mathcal{X}'$ such that 
\begin{equation}\label{def:oseen_system_weak_dual_source}
	\left\{
	\begin{array}{rcll}
\widehat{a}(\bv,\widehat{\bu}^*)-b(\widehat{p}^*,\bv)&=&(\bv,\boldsymbol{f})&\forall \bv\in \H^{-1}(\O,\mathbb{C})^d,\\
-b(q, \widehat{\bu}^*)&=&0&\forall q\in \L_0^2(\O,\mathbb{C}).
\end{array}
	\right.
\end{equation}
Similar to Theorem \ref{th:regularidadfuente}, we have that the dual source and eigenvalue problems are such that the following estimate holds.

\begin{lemma}
There exist $s^*>0$ such that for all $\boldsymbol{f} \in \L^2(\O,\mathbb{C})^d$, the solution $(\widehat{\bu}^*,\widehat{p}^*)$ of problem \eqref{def:oseen_system_weak_dual_source}, satisfies $\widehat{\bu}^*\in  \H^{1+s^*}(\Omega,\mathbb{C} )^d$ and $\widehat{p}^*\in \H^{s^*}(\Omega,\mathbb{C} )$, and
 \begin{equation*}
\|\widehat{\bu}^*\|_{1+s^*,\O}+\|\widehat{p}^*\|_{s^*,\O}\leq C \|\boldsymbol{f}\|_{0,\O},
\end{equation*}
where the constant $C$ is the same as in Theorem \ref{th:regularidadfuente}.
\end{lemma}
Hence, with this regularity result at hand, the spectral characterization of $\bT^*$ is given as follows.
\begin{lemma}(Spectral Characterization of $\bT^*$).
The spectrum of $\bT^*$ is such that $\sp(\bT^*)=\{0\}\cup\{\kappa_{k}^*\}_{k\in{N}}$ where $\{\kappa_{k}*\}_{k\in\mathbf{N}}$ is a sequence of complex eigenvalues that converge to zero, according to their respective multiplicities. 
\end{lemma}
It is easy to prove that if $\kappa$ is an eigenvalue of $\bT$ with multiplicity $m$, $\bar{\kappa}^*$ is an eigenvalue of $\bT^*$ with the same multiplicity $m$.

Let us define the sesquilinear form $\widehat{A}:\mathcal{X}'\times\mathcal{X}'\rightarrow\mathbb{C}$ by
\begin{equation*}
\widehat{A}((\bv,q);(\bu^*,p^*)):=\widehat{a}(\bv,\bu^*)-b(p^*,\bv)+b(q, \bu^*), 
\end{equation*}
which allows us to rewrite the dual eigenvalue problem  \eqref{def:oseen_system_weak_dual_eigen} as follows:  Find $\lambda^*\in\mathbb{C}$ and the pair $(\boldsymbol{0},0)\neq(\bu^*,p^*)\in\mathcal{X}'$ such that  
\begin{equation*}
\label{eq:eigen_A_dual}
\widehat{A}((\bv,q);(\bu^*,p^*))=\lambda^*(\bv,\bu^*)\quad\forall (\bv,q)\in\mathcal{X}'.
\end{equation*}

The dual counterpart of Lemma \ref{lem:inf-sup-A} is given below,  where we have that $\widehat{A}$ is stable.

\begin{lemma}\label{lem:inf-sup-A*}
	The form $\widehat{A}(\cdot,\cdot)$ satisfies the inf-sup conditions
	\begin{align*}
		\nonumber &\ds\inf_{(\boldsymbol{0},0)\neq(\bw,r)\in\mathcal{X}'}\sup_{(\boldsymbol{0},0)\neq(\bv,q)\in\mathcal{X}'}\frac{\widehat{A}((\bv,q);(\bw,r))}{\|(\bv,q)\|\|(\bw,r)\|}=\gamma^*,\\
		&\inf_{(\boldsymbol{0},0)\neq(\bv,q)\in\mathcal{X}'}\sup_{(\boldsymbol{0},0)\neq(\bw,r)\in\mathcal{X}'}\frac{\widehat{A}((\bv,q);(\bw,r))}{\|(\bv,q)\|\|(\bw,r)\|}=\gamma^*.
		\label{eq:beta_infsup*}
	\end{align*}
	where $\gamma^*$ is a positive  constant,  uniform with respect to $\nu$. Consequently, given $(\bv,q)\in\mathcal{X}'$, there exists $(\bw,r)\in\mathcal{X}'$ such that
	\begin{equation*}
		\label{lem:stability1*}
		\begin{aligned}
			&\Vert \bw\Vert_{1,\O} + \Vert r\Vert_{0,\O}\leq C,\\
			&\Vert \bv\Vert_{1,\O} + \Vert q \Vert_{0,\O} \leq \widehat{A}((\bv,q);(\bw,r)).
		\end{aligned}
	\end{equation*}
\end{lemma}

In the forthcoming we will introduce the finite element discretization of \eqref{def:oseen_system_weak} and \eqref{def:oseen_system_weak_dual_source}, and hence, of $\bT$ and $\bT^*$.
\section{The finite element method}
\label{sec:fem}
In this section our aim is to describe a finite element discretization of problem \eqref{def:oseen_system_weak}. To do this task, we will introduce two families of inf-sup stable finite elements for the Oseen load problem whose properties also hold for the eigenvalue problem. We begin by introducing some preliminary definitions and notations to perform the analysis. 

\subsection{Inf--sup stable finite element spaces}
\label{sec:infsup_fem}

Let $\mathcal{T}_h=\{T\}$ be a conforming partition  of $\overline{\O}$ into closed simplices $T$ with size $h_T=\text{diam}(T)$. Define $h:=\max_{T\in\mathcal{T}_h}h_T$. Given a mesh $\mathcal{T}_h\in\mathbb{T}$, we denote by $\boldsymbol{V}_h$ and $\mathcal{P}_h$ the finite element spaces that approximate the velocity field and the pressure, respectively. In particular, our study is focused in the following two elections:
\begin{itemize}
\item[(a)] The mini element \cite[Section 4.2.4]{MR2050138}: Here,
\begin{align*}
&\boldsymbol{V}_h=\{\bv_{h}\in\boldsymbol{C}(\overline{\O})\ :\ \bv_{h}|_T\in[\mathbb{P}_1(T)\oplus\mathbb{B}(T)]^{d} \ \forall \ T\in\mathcal{T}_h\}\cap\H_0^{1}(\O,\mathbb{C})^d,\\
&\mathcal{P}_h=\{ q_{h}\in C(\overline{\O})\ :\ q_{h}|_T\in\mathbb{P}_1(T) \ \forall \ T\in\mathcal{T}_h \}\cap \L_0^{2}(\O,\mathbb{C}),
\end{align*}
where $\mathbb{B}(T)$ denotes the space spanned by local bubble functions.
\item[(b)] The lowest order Taylor--Hood element \cite[Section 4.2.5]{MR2050138}: In this case,
\begin{align*}
&\boldsymbol{V}_h=\{\bv_{h}\in\boldsymbol{C}(\overline{\O})\ :\ \bv_h|_T\in[\mathbb{P}_2(T)]^{d} \ \forall \ T\in\mathcal{T}_h\}\cap\H_0^{1}(\O,\mathbb{C})^d,\\
&\mathcal{P}_h=\{ q_{h}\in C(\overline{\O})\ :\ q_{h}|_T\in\mathbb{P}_1(T) \ \forall\ T\in\mathcal{T} \}\cap \L_0^{2}(\O,\mathbb{C}).
\end{align*}
\end{itemize}
The discrete analysis will be performed in a general manner, where the both families of finite elements, namely Taylor-Hood and mini element, are considered with no difference. If some difference must be claimed, we will point it out when is necessary. To make  more simple the presentation of the material, let us define the space $\mathcal{X}_h:=\boldsymbol{V}_h\times\mathcal{P}_h$.
\subsection{The discrete eigenvalue problem}
With our finite element spaces at hand, we are in position to introduce the FEM discretization of problem \eqref{def:oseen_system_weak} which reads as follows:  Find $\lambda_h\in\mathbb{C}$ and $(\boldsymbol{0},0)\neq(\bu_h,p_h)\in\mathcal{X}_h$ such that 
\begin{equation}\label{def:oseen_system_weak_disc}
	\left\{
	\begin{array}{rcll}
a(\bu_h,\bv_h) + b(\bv_h,p_h)&=&\lambda_h(\bu_h,\bv_h)&\forall \bv_h\in\boldsymbol{V}_h,\\
b(\bu_h,q_h)&=&0&\forall q_h\in\mathcal{P}_h.
\end{array}
	\right.
\end{equation}
We introduce the discrete solution operator $\bT_h$ defined as follows
\begin{equation*}\label{eq:operador_solucion_u_h}
	\bT_h:\L^2(\O,\mathbb{C})^d\rightarrow \boldsymbol{V}_h,\qquad \boldsymbol{f}\mapsto \bT_h\boldsymbol{f}:=\widehat{\bu}_h, 
\end{equation*}
where the pair  $(\widehat{\bu}_h, \widehat{p}_h)\in\mathcal{X}_h$ is the solution of the following well posed source discrete problem
\begin{equation}\label{def:oseen_system_weak_disc_source}
	\left\{
	\begin{array}{rcll}
a(\bu_h,\bv_h) + b(\bv_h,p_h)&=& (\boldsymbol{f},\bv_h)&\forall \bv_h\in\boldsymbol{V}_h,\\
b(\bu_h,q_h)&=&0&\forall q_h\in\mathcal{P}_h.
\end{array}
	\right.
\end{equation}

The choices of $\boldsymbol{V}_h$ and $\mathcal{P}_h$ make that \eqref{def:oseen_system_weak_disc_source} is well posed since the following discrete inf-sup condition holds
\begin{equation}
\label{ec:inf-sup_disc}
\displaystyle\sup_{\btau_h\in\boldsymbol{V}_h}\frac{b(\btau_h,q_h)}{\|\btau_h\|_{1,\O}}\geq\bar{\beta}\|q_h\|_{0,\O}\quad\forall q\in\mathcal{P}_h.
\end{equation}
implying  that $\bT_h$ is well defined due to the Babu\^{s}ka-Brezzi theory.  Moreover, 
we have the following estimates for the velocity and pressure \cite[Lemma 5.13]{John2016}
\begin{equation*}\label{eq:estimatefuente_velocity_discrete}
\ds\nu\|\nabla\widehat{\bu}_h\|_{0,\O}\leq C_p\|\boldsymbol{f}\|_{0,\O},
\end{equation*}
whereas for the pressure we have
\begin{equation*}
\label{eq:estimatefuente_pressure_discrete}
\ds \|\widehat{p}_h\|_{0,\O}^2\leq \frac{1}{\bar{\beta}}\left( C_{p}\|\boldsymbol{f}\|_{0,\O}+\nu^{1/2}
\|\nabla\widehat{\bu}_h\|_{0,\O}\left(\nu^{1/2}+\dfrac{C_p\|\boldsymbol{\beta}\|_{\infty,\O}}{\nu^{1/2}}\right)\right),
\end{equation*}
where $\bar{\beta}>0$ is the constant given by  the discrete inf-sup condition \eqref{ec:inf-sup_disc}.

As in the continuous case, $(\lambda_h,(\bu_h,p_h))$ solves Problem \eqref{def:oseen_system_weak_disc} if and only if $(\kappa_h,\bu_h)$ is an eigenpair of  $\bT_h$.
The discrete counterpart of \eqref{eq:eigenA} is defined from \eqref{def:oseen_system_weak_disc} as: Find $\lambda_h\in\mathbb{C}$ and $(\boldsymbol{0},0)\neq(\bu_h,p)\in\mathcal{X}_h$ such that
	\begin{equation}
		\label{eq:eigenA_disc}
		A((\bu_h,p_h);(\bv,q))=\lambda_h (\bu_h,\bv_h)_{0,\O}\quad\forall (\bv,q)\in\mathcal{X}_h,
	\end{equation}
	where $A((\cdot,\cdot);(\cdot,\cdot))$ is the same  as \eqref{eq:eigenA}.

The dual discrete eigenvalue problem reads as follows: Find $\lambda^*\in\mathbb{C}$ and the pair $(\boldsymbol{0},0)\neq(\bu_h^*,p_h^*)\in\mathcal{X}_h$ such that  
\begin{equation}\label{def:oseen_system_weak_dual_eigen_disc}
	\left\{
	\begin{array}{rcll}
a(\bv_h,\bu_h^*)-b(p_h^*,\bv_h)&=&\lambda_h^*(\bv_h,\widehat{\bu}_h^*)&\forall \bv_h\in \boldsymbol{V}_h,\\
-b(q_h, \bu_h^*)&=&0&\forall q\in\mathcal{P}_h,
\end{array}
	\right.
\end{equation}
and  let us introduce the discrete version of \eqref{eq:operador_adjunto_solucion_u} which we define by \begin{equation*}\label{eq:operador_solucion_u_h_adjunto}
	\bT^*_h:\L^2(\O,\mathbb{C})^d\rightarrow \mathcal{P}_h,\qquad \boldsymbol{f}\mapsto \bT_h^*\boldsymbol{f}:=\widehat{\bu}_h^*, 
\end{equation*}
where the pair $(\widehat{\bu}_h^*,\widehat{p}_h^*)\in\mathcal{X}_h$ is the solution of the following adjoint discrete source problem
\begin{equation}\label{def:oseen_system_weak_dual_source_disc}
	\left\{
	\begin{array}{rcll}
a(\bv_h,\widehat{\bu}_h^*)-b(\widehat{p}_h^*,\bv_h)&=&(\boldsymbol{f},\bv_h)&\forall \bv_h\in \boldsymbol{V}_h,\\
-b(q_h, \widehat{\bu}_h^*)&=&0&\forall q_h\in \mathcal{P}_h,
\end{array}
	\right.
\end{equation}

Now, due to the compactness of $\bT$, we are able to prove that $\bT_h$ converge to $\bT$ as $h$ goes to zero in norm. This is contained in the following result.
\begin{lemma}
\label{lmm:conv1}
Let $\boldsymbol{f}\in \L^2(\Omega,\mathbb{C})$ be such that $\widehat{\bu}:=\boldsymbol{T}\boldsymbol{f}$ and $\widehat{\bu}_h:=\bT_h\boldsymbol{f}$. Then, there exists a positive constant $C$, independent of $h$, such that 
\begin{equation*}
\|(\bT-\bT_h)\boldsymbol{f}\|_{1,\Omega}\leq  C_{\nu,\boldsymbol{\beta},C_p}h^s\|\boldsymbol{f}\|_{0,\O},
\end{equation*}
where  $C_{\nu,\boldsymbol{\beta},C_p}:=\max\left\{1+\|\boldsymbol{\beta}\|_{\infty,\O}\min\left\{\dfrac{C_p}{\nu},\dfrac{1}{\nu^{1/2}}\right\},\dfrac{1}{\nu}\right\}$.
Moreover, we have the improved estimate
\begin{equation*}
	\|(\bT-\bT_h)\boldsymbol{f}\|_{0,\Omega}\leq  C_{\nu,\boldsymbol{\beta},C_p}h^{s+1}\|\boldsymbol{f}\|_{0,\O},
\end{equation*}
where the constant $C_{\nu,\boldsymbol{\beta},C_p}$ is the same defined above.
\end{lemma}
\begin{proof}
Let $\boldsymbol{f}\in\L^2(\O)^d$ be such that $\widehat{\bu}:=\bT\boldsymbol{f}$ and $\widehat{\bu}_h.=\bT_h\boldsymbol{f}$. Then, the following C\'ea estimate holds
\begin{equation*}
\|(\bT-\bT_h)\boldsymbol{f}\|_{1,\Omega}=\|\widehat{\bu}-\widehat{\bu}_h\|_{1,\O}.
\end{equation*}
Now, the proof of the lemma follows from Theorem 5.14, Theorem 5.15, Theorem C.13, Corollary 5.16 of \cite{John2016} and Theorem \ref{th:regularidadfuente}. Since $\boldsymbol{\beta}$ is divergence free, the improved estimate follows similar to the $\H^1$ estimate, but with the aid of the additional regularity from Theorem \ref{th:regularidadfuente} together with \cite[Theorem 6.32, Theorem 6.34, Corollary 6.33]{John2016}. 
\end{proof}
Also for the adjoint problem, we have the following convergence result. Since the proof is essentially identical to Lemma \ref{lmm:conv1} we skip the steps of the proof.
\begin{lemma}
\label{eq:adjoint_diff}
There exists a constant $C>0$, independent of $h$, such that
\begin{equation*}
\|(\bT^*-\bT_h^*)\boldsymbol{f}\|_{1,\O}\leq C_{\nu,\boldsymbol{\beta},C_p} h^{s*}\|\boldsymbol{f}\|_{0,\O}.
\end{equation*}
Moreover, we have the improved estimate
	\begin{equation*}
		\|(\bT^*-\bT_h^*)\boldsymbol{f}\|_{0,\Omega}\leq  C_{\nu,\boldsymbol{\beta},C_p}h^{s^*+1}\|\boldsymbol{f}\|_{0,\O},
	\end{equation*}
	where the constant $C_{\nu,\boldsymbol{\beta},C_p}$ is the same as in Lemma \ref{lmm:conv1}.

\end{lemma}

The key  consequence of the previous results is that we are in position to apply the well established theory of  \cite{MR0203473}  to conclude that  our numerical methods does not introduce spurious eigenvalues. This is stated in the following theorem.
\begin{theorem}
	\label{thm:spurious_free}
	Let $V\subset\mathbb{C}$ be an open set containing $\sp(\bT)$. Then, there exists $h_0>0$ such that $\sp(\bT_h)\subset V$ for all $h<h_0$.
\end{theorem}
\subsection{Error estimates}
\label{sec:conv}
The goal of this section is deriving error estimates for the eigenfunctions and eigenvalues. We first recall the definition of spectral projectors. Let $\mu$ be a nonzero isolated eigenvalue of $\bT$ with algebraic multiplicity $m$ and let $\Gamma$
be a disk of the complex plane centered in $\mu$, such that $\mu$ is the only eigenvalue of $\bT$ lying in $\Gamma$ and $\partial\Gamma\cap\sp(\bT)=\emptyset$. With these considerations at hand, we define the spectral projections of $\boldsymbol{E}$ and $\boldsymbol{E}^*$, associated to $\bT$ and $\bT^*$, respectively, as follows:
\begin{enumerate}
\item The spectral projector of $\bT$ associated to $\mu$ is $\displaystyle \boldsymbol{E}:=\frac{1}{2\pi i}\int_{\partial\Gamma} (z\boldsymbol{I}-\bT)^{-1}\,dz;$
\item The spectral projector of $\bT^*$ associated to $\bar{\mu}$ is $\displaystyle \boldsymbol{E}^*:=\frac{1}{2\pi i}\int_{\partial\Gamma} (z\boldsymbol{I}-\bT^*)^{-1}\,dz,$
\end{enumerate}
where $I$ represents the identity operator. Let us remark that $E$ and $E^*$ are the projections onto the generalized eigenvector $R(\boldsymbol{E})$ and $R(\boldsymbol{E}^*)$, respectively. 
A consequence of Lemma \ref{lmm:conv1} is that there exist $m$ eigenvalues, which lie in $\Gamma$, namely $\mu_h^{(1)},\ldots,\mu_h^{(m)}$, repeated according their respective multiplicities, that converge to $\mu$ as $h$ goes to zero. With this result at hand, we introduce the following spectral projection
\begin{equation*}
\boldsymbol{E}_h:=\frac{1}{2\pi i}\int_{\partial\Gamma} (z\boldsymbol{I}-\bT_h)^{-1}\,dz,
\end{equation*}
which is a projection onto the discrete invariant subspace $R(\boldsymbol{E}_h)$ of $\bT$, spanned by the generalized eigenvector of $\bT_h$ corresponding to 
 $\mu_h^{(1)},\ldots,\mu_h^{(m)}$.
Now we recall the definition of the \textit{gap} $\hdel$ between two closed
subspaces $\mathfrak{X}$ and $\mathfrak{Y}$ of $\L^2(\O)^d$:
$$
\hdel(\mathfrak{X},\mathfrak{Y})
:=\max\big\{\delta(\mathfrak{X},\mathfrak{Y}),\delta(\mathfrak{Y},\mathfrak{X})\big\}, \text{ where } \delta(\mathfrak{X},\mathfrak{Y})
:=\sup_{\underset{\left\|\boldsymbol{x}\right\|_{0,\O}=1}{\boldsymbol{x}\in\mathfrak{X}}}
\left(\inf_{\boldsymbol{y}\in\mathfrak{Y}}\left\|\boldsymbol{x}-\boldsymbol{y}\right\|_{0,\O}\right).
$$
We end this section proving error estimates for the eigenfunctions and eigenvalues.
\begin{theorem}
\label{thm:errors1}
The following estimates hold
\begin{equation*}
\hdel(R(\boldsymbol{E}),R(\boldsymbol{E}_h))\leq C_{\nu,\boldsymbol{\beta},C_p} h\quad\text{and}\quad
|\mu-\mu_h|\leq C_{\nu,\boldsymbol{\beta},C_p}\widehat{C}^{s+s^*},
\end{equation*}
where the constant $C_{\nu,\boldsymbol{\beta},C_p}$ is the same as in Lemma \ref{lmm:conv1} and $$\widehat{C}:=\max\left\{2C_{\nu,\boldsymbol{\beta},C_p},\dfrac{1}{\bar{\beta}}\max\left\{\widetilde{C}^2,1+\dfrac{1}{\bar{\beta}}+\dfrac{1}{\beta_h\nu^{1/2}}\right\}\right\},$$
with $\widetilde{C}:=\left(1+\|\boldsymbol{\beta}\|_{\infty,\O}\min\left\{\dfrac{C_p}{\nu},\dfrac{1}{\nu^{1/2}}\right\}\right)$ and $\bar{\beta}$  is the constant given by the discrete inf-sup condition \eqref{ec:inf-sup_disc}. 
\end{theorem}
\begin{proof}
The proof of the gap between the eigenspaces is a direct consequence of the convergence in norm between $\bT$ and $\bT_h$ as $h$ goes to zero.
We focus on the double order of convergence for the eigenvalues. Let $\{\bu_k\}_{k=1}^m$ be such that $\bT \bu_k=\mu \bu_k$, for $k=1,\ldots,m$. A dual basis for $R(\boldsymbol{E}^*)$ is $\{\bu_k^*\}_{k=1}^m$. This basis satisfies $A((\bu_k,p);(\bu_l^*,p^*))=\delta_{k.l},$
where $\delta_{k.l}$ represents the Kronecker delta.
On the other hand, the following identity holds
\begin{equation*}
|\mu-\widehat{\mu}_h|\lesssim \frac{1}{m}\sum_{k=1}^m|\langle(\bT-\bT_h)\bu_k,\bu_k^* \rangle|+\|(\bT-\bT_h)|_{R(\boldsymbol{E})} \| \|(\bT^*-\bT_h^*)|_{R(\boldsymbol{E}^*)}\|.
\end{equation*}
For the first term on the right-hand side we note that
\begin{multline*}
\langle(\bT-\bT_h)\bu_k,\bu_k^* \rangle=A((\bT-\bT_h)\bu_k,p-p_h);(\bu_k^*,p^*))\\=
A((\bT-\bT_h)\bu_k,p-p_h);(\bu_k^*-\bu_{k,h}^*,p^*-p_h^*))\\\leq \|(\bT-\bT_h)\bu_k\|_{1,\O}\|\bu_k^*-\bu_{k,h}^*\|_{1,\O}+\|p-p_h\|_{0,\O}\|\bu_k^*-\bu_{k,h}^*\|_{1,\O}\\+\|(\bT-\bT_h)\bu_k\|_{1,\O}\|p^*-p_h^*\|_{0,\O}.
\end{multline*}
Then, Theorem \ref{thm:errors1} follows from the above estimates, the approximation properties of discrete spaces, in addition to Lemmas \ref{lmm:conv1} and \ref{eq:adjoint_diff}.
\end{proof}
\subsection{Reduced solution operators}
With the above result at hand, we are allowed  to  define the following operators:
for any $\widehat{\bF}\in \L^2(\O,\mathbb{C})^d$ we introduce the   linear and compact operator $\widehat{\bT}$ defined by 
$
 \widehat{\bT}:\L^2(\Omega,\mathbb{C})^d\rightarrow\L^2(\O,\mathbb{C})^d,\,\,\,\widehat{\bF}\mapsto\widehat{\bT}\widehat{\bF}:=\widetilde{\bu},
$
where the pair $(\widetilde{\bu},\widetilde{p})$  is the solution of \eqref{def:oseen_system_weak_source}.
Also, we introduce  $\widehat{\bT}_h$  as the discrete linear counterpart of $\widehat{\bT}$, defined by
$
 \widehat{\bT}_h:\L^2(\Omega,\mathbb{C})^d\rightarrow\mathbf{H}_h,\,\,\,\widehat{\bF}\mapsto\widehat{\bT}_h\widehat{\bF}:=\widetilde{\bu}_h,
$
where the pair $(\widetilde{\bu}_h,\widetilde{p}_h)$  is the solution of \eqref{def:oseen_system_weak_disc_source}. On the other hand, we introduce the respective adjoint reduced operators defined by 
$
 \widehat{\bT}^*:\L^2(\Omega,\mathbb{C})^d\rightarrow\L^2(\O,\mathbb{C})^d,\,\,\,\widehat{\bF}\mapsto\widehat{\bT}\widehat{\bF}:=\widetilde{\bu}^*,
$
for the continuous counterpart where the pair $(\widetilde{\bu}^*,\widetilde{p}^*)$ solves problem \eqref{def:oseen_system_weak_dual_source}, whereas the discrete version of  $\widehat{\bT}_h^*$
is defined by 
$
 \widehat{\bT}_h^*:\L^2(\Omega,\mathbb{C})^d\rightarrow\L^2(\O,\mathbb{C})^d,\,\,\,\widehat{\bF}\mapsto\widehat{\bT}_h^*\widehat{\bF}:=\widetilde{\bu}_h^*$, and the pair $(\widetilde{\bu}_h^*,\widetilde{p}_h^*)$ is the solution of \eqref{def:oseen_system_weak_dual_source_disc}.
As a consequence of Lemmas \ref{lmm:conv1} and \ref{eq:adjoint_diff}, the following is ensured convergence in operators' norm, i.e:
\begin{equation}
\label{eq:mejororden1}
\begin{split}
\|(\widehat{\bT}-\widehat{\bT}_h)\widehat{\bF}\|_{0,\O}&\leq C_{\nu,\boldsymbol{\beta},C_p}h^{1+s}\|\widehat{\bF}\|_{0,\O},\\
\|(\widehat{\bT}^*-\widehat{\bT}_h)^*\widehat{\bF}\|_{0,\O}&\leq C_{\nu,\boldsymbol{\beta},C_p}h^{1+s*}\|\widehat{\bF}\|_{0,\O}.
\end{split}
\end{equation} 

Hence we guarantee a spectral convergence result for  $\widehat{\bT}_h$ and $\widehat{\bT}$ as $h\rightarrow 0$. Now  we are in a position to guarantee the following estimate.
\begin{lemma}
Let $\bu_h$ be an eigenfunction of $\bT_h$ associated with the eigenvalue $\mu_{k,h}$, with $\|\bu_h\|_{0,\O}$=1. Then, there exists an eigenfunction $\bu$ of $\bT$ associated with $\mu$ such that, for all $s>0$, there exists $C_{\nu,\boldsymbol{\beta},C_p}$ as in Lemma \ref{lmm:conv1} such that
$$\|\bu-\bu_h\|_{0,\O}\leq C_{\nu,\boldsymbol{\beta},C_p} h^{1+s}.$$
\end{lemma}
\begin{proof}
Thanks to \eqref{eq:mejororden1}, \cite[Theorem 7.1]{MR1115235}  yields spectral convergence of $\widehat{\bT}_h$ to $\widehat{\bT}$ as $h\rightarrow 0$. Specifically, due to the connection between the eigenfunctions of $\bT$ and $\bT_h$ and the eigenfunctions of $\widehat{\bT}$  and $\widehat{\bT}_h$, respectively, we have that  $\bu_h\in\boldsymbol{E}_h$, and there exits $\bu\in \boldsymbol{E}$ such that
$$\|\bu-\bu_h\|_{0,\O}\leq C\sup_{\widehat{\bF}\in \widehat{\boldsymbol{E}}:\|\widehat{\bF}\|_{0,\O}=1}\|(\widehat{\bT}-\widehat{\bT}_h)\widehat{\bF}\|_{0,\O}.$$
Conversely, thanks to Lemma \ref{lmm:conv1}, for any function $\widehat{\bF}\in \widehat{\boldsymbol{E}}$, if  $\bF\in \boldsymbol{E}$ is such that $\widehat{\bF}=\bF$, then
\begin{equation*}
\|(\widehat{\bT}-\widehat{\bT}_h)\widehat{\bF}\|_{0,\O}=\|(\bT-\bT_h)\bF\|_{0,\O}\leq C_{\nu,\boldsymbol{\beta},C_p}h^{1+s}\|\bF\|_{0,\O}.
\end{equation*}
which together with the above estimate allow us to conclude the proof.
\end{proof}

Furthermore, in the context of the adjoint problem, we have the following convergence result.
\begin{lemma}
Let $\bu_h^*$ be an eigenfunction of $\bT_h^*$ associated with the eigenvalue $\mu_{k,h}^*$, with $\|\bu_h^*\|_{0,\O}$=1. Then, there exists an eigenfunction $\bu^*$ of $\bT^*$ associated with $\mu^*$ such that, for all $s^*>0$, there exists $C>0$ such that
\begin{equation*}\|\bu^*-\bu_h^*\|_{0,\O}\leq C_{\nu,\boldsymbol{\beta},C_p} h^{1+s^*}.\end{equation*}
\end{lemma}
\begin{remark}
\label{summaryapriori}
We conclude this section by summarizing the a priori error estimates obtained. If $(\lambda_k , \bu_k )$ is a solution of \eqref{def:oseen_system_weak_dual_eigen},
then there exists $(\lambda_{k,h}, \bu_{k,h})$ satisfying \eqref{def:oseen_system_weak_disc} such that :
$$
\begin{aligned}
	\|\bu-\bu_{k,h}\|_{1,\O}+\|p-p_h\|_{0}&\leq \widehat{C} h^s,\\
\|\bu-\bu_{k,h}\|_{0,\O}&\leq \nu,\boldsymbol{\beta},C_p h^{1+s},\\
\|\bu^*-\bu_{k,h}^*\|_{1,\O}+\|p^*-p_h^*\|_{0}&\leq \widehat{C} h^{s^*},\\
\|\bu^*-\bu_{k,h}^*\|_{0,\O}&\leq C_{\nu,\boldsymbol{\beta},C_p}h^{1+s^*},\\
|\lambda-\lambda_{k,h}|&\leq C_{\nu,\boldsymbol{\beta},C_p\widehat{C}} h^{s+{s^*}},
\end{aligned}
$$
where $C_{\nu,\boldsymbol{\beta},C_p}$, $\widehat{C}$ are as in Theorem \ref{thm:errors1}.
\end{remark}

\section{A posteriori analysis}
\label{sec:apost}
The aim of this section is to introduce a suitable
residual-based error estimator for the Oseen
eigenvalue problem which is fully computable,
in the sense that  it depends only on quantities available from the FEM solution. Then, we will show its equivalence with the error.
Moreover, on the forthcoming analysis we will focus only on eigenvalues with simple multiplicity. With  this purpose, we introduce the following definitions and notations. For any element $T\in \CT_h$, we denote by $\CE_{T}$ the set of faces/edges of $T$
and 
$$\CE_h:=\bigcup_{T\in\CT_h}\CE_{T}.$$
We decompose $\CE_h=\CE_{\O}\cup\CE_{\partial\O}$,
where  $\CE_{\partial\O}:=\{\ell\in \CE_h:\ell\subset \partial\O\}$
and $\CE_{\O}:=\CE\backslash\CE_{\partial \O}$.
For each inner face/edge $\ell\in \CE_{\O}$ and for any  sufficiently smooth  function
$\bv$, we define the jump of its normal derivative on $\ell$ by
$$\left[\!\!\left[ \dfrac{\partial \bv}{\partial{ \boldsymbol{n}}}\right]\!\!\right]_\ell:=\nabla (\bv|_{T})  \cdot \boldsymbol{n}_{T}+\nabla ( \bv|_{T'}) \cdot \boldsymbol{n}_{T'} ,$$
where $T$ and $T'$ are  the two elements in $\CT_{h}$  sharing the
face/edge $\ell$ and $\boldsymbol{n}_{T}$ and $\boldsymbol{n}_{T'}$ are the respective outer unit normal vectors.

\subsection{Local and global indicators}
In what follows we introduce a suitable residual-based error estimator for the Oseen  eigenvalue problem. To do this task, for an element $T\in\CT_h$ we introduce the following local error indicator:
\begin{multline*}
\eta_T^2:=h_T^2\|\lambda_h\bu_h+\nu\Delta\bu_h-(\boldsymbol{\beta}\cdot\nabla)\bu_h)-\nabla p_h\|_{0,T}^2\\
+\|\div\bu_h\|_{0,T}^2+\frac{h_e}{2}\sum_{e\in\CT_h}\|\jump{(\nu\nabla \bu_h-p_h\mathbb{I})\cdot\boldsymbol{n}}\|_{0,e}^2,
\end{multline*}
and the global estimator for the primal problem
\begin{equation*}
\eta:=\left(\sum_{T\in\CT_h}\eta_T^2 \right)^{1/2}.
\end{equation*}
On the other hand, we define the local indicator for the dual problem as follows
\begin{multline*}
\eta_{T}^{*,2}:=h_T^2\|\lambda_h\bu_h + \nu\Delta\bu^*_h+\div(\bu_h^*\otimes\boldsymbol{\beta})+\nabla p_h^*\|_{0,T}^2\\
+\|\div\bu_h^*\|_{0,T}^2+\frac{h_e}{2}\sum_{e\in\CT_h}\|\jump{(\nu\nabla \bu_h^*+p_h^*\mathbb{I})\cdot\boldsymbol{n}}\|_{0,e}^2,
\end{multline*}
and the global estimator for the dual problem is 
\begin{equation*}
\eta^*:=\left(\sum_{T\in\CT_h}\eta_T^2 \right)^{1/2}.
\end{equation*}
It is  important to remark that it is possible to define a global estimator as the sum of the contributions given by the primal and dual problems as follows
\begin{equation*}
\theta:=\left(\eta^2+(\eta^*)^2\right)^{1/2}.
\end{equation*}
Now the aim is to prove that the proposed estimator is reliable and efficient. In the following, for the sake of simplicity, sometimes we will denote the errors between the continuous and discrete eigenfunctions as
$$
\eu:=\bu - \bu_{h}, \quad \ep:=p - p_h, \quad \eu^*:=\bu^* - \bu_{h}^*, \quad \ep^*:= p^* - p_h^*.
$$

\subsection{Reliability}

Let us begin with the reliability analysis of the proposed estimators. The task here is to prove that the error is upper bounded by the estimator, together with 
the high order terms. This is stated in the following result.
\begin{theorem}[Reliability]\label{th:reliability}
	The following statements hold
	\begin{enumerate}
		\item Let $(\lambda,(\bu,p))\in\mathbb{C}\times\mathcal{X}$ be a solution of the spectral problem \eqref{def:oseen_system_weak} and let $(\lambda_h,(\bu_h,p_h))\in\mathbb{C}\times\mathcal{X}_h$ be the finite element approximation of $(\lambda,(\bu,p))$ given as the solution of \eqref{def:oseen_system_weak_disc}. Then, for every $h_0\geq h$ there holds
		\begin{equation}\label{eq:reliabilityprimal}
			\|\bu-\bu_h\|_{1,\O}+\|p-p_h\|_{0,\O}\leq C(\eta+|\lambda-\lambda_h|+\lambda\|\bu-\bu_h\|_{0,\O}).
		\end{equation}
		
		\item Let $(\lambda^*,(\bu^*,p^*))\in\mathbb{C}\times\mathcal{X}'$ be a solution of the spectral problem problem \eqref{def:oseen_system_weak_dual_eigen} and let $(\lambda_h^*,(\bu_h^*,p_h^*))\in\mathbb{C}\times\mathcal{X}_h$ be the finite element approximation of $(\lambda^*,(\bu^*,p^*))$ given as the solution of \eqref{def:oseen_system_weak_dual_eigen_disc}. Then, for every $h_0\geq h$ there holds
		\begin{equation}\label{eq:reliabilitydual}
			\|\bu^*-\bu_h^*\|_{1,\O}+\|p^*-p_h^*\|_{0,\O}\leq C(\eta^*+|\lambda^*-\lambda_h^*|+\lambda^*\|\bu^*-\bu_h^*\|_{0,\O}),
		\end{equation}
	\end{enumerate}
	where in each estimate the constant $C>0$ depends on $\nu$, but is independent of the mesh size and the  discrete solutions.
\end{theorem}

\begin{proof} It is enough to prove the reliability for the primal formulation. The dual reliability follows by similar arguments. Given $(\bv,q)\in\mathcal{X}_h$, we subtract the continuous formulation (cf. \eqref{eq:eigenA} ) and its discrete counterpart (cf. \eqref{eq:eigenA_disc}) in order to obtain the error equation
	\begin{equation}
		\label{lem:reliability-eq001}
		A((\texttt{e}_{\bu},\texttt{e}_{p});(\bv,q))=(\lambda\bu - \lambda_h\bu_h,\bv)_{0,\O}\quad\forall (\bv,q)\in\mathcal{X}_h.
	\end{equation}
	
Thanks to  Lemma \ref{lem:inf-sup-A}, we have that for $(\bu,p)\in\mathcal{X}$, there exists $(\bw,r)\in\mathcal{X}$, with
	\begin{equation}
		\label{lem:reliability-eq002}
		\Vert \bw\Vert_{1,\O}+\Vert r\Vert_{0,\O} \leq C,
	\end{equation}
	such that
	\begin{equation}
		\label{lem:reliability-eq003}
		\Vert \texttt{e}_{\bu}\Vert_{1,\O} + \Vert \texttt{e}_p \Vert_{0,\O} \leq A((\texttt{e}_{\bu},\texttt{e}_p);(\bw,r)).
	\end{equation}
	Let us define $\bw_I\in \boldsymbol{V}_h$ as the Cl\'ement interpolant (see \cite[Chapter 2]{koelink2006partial}) and $r_I\in \mathcal{P}_h$ as the usual $\L^2$-projection of $r$. Then, from \eqref{lem:reliability-eq001} and \eqref{lem:reliability-eq003} we obtain
	 \begin{multline}
		\label{lem:reliability-eq004}
		\Vert \texttt{e}_{\bu}\Vert_{1,\O} + \Vert \texttt{e}_p \Vert_{0,\O} \leq\lambda(\bu,\bw)_{0,\O}-A((\bu_h,p_h);(\bw,r)) \\
		= (\lambda\bu-\lambda_{h}\bu_h,\bw)_{0,\O}+\lambda_h(\bu_h,\bw-\bw_I)_{0,\O} - A((\bu_{h},p_h);(\bw-\bw_I,r-r_I))  \\
		= \underbrace{-A((\bu_{h},p_h);(\bw - \bw_I,r - r_I)) + (\lambda_h\bu_h,\bw-\bw_I)_{0,\O}}_{\Lambda_1}
		\\+\underbrace{ (\lambda\bu - \lambda_h\bu_h,\bw)_{0,\O}}_{\Lambda_2}.
	\end{multline}
	The task is to control $\Lambda_1$ and $\Lambda_2$. For $\Lambda_1$, we use integration by parts to obtain
	\begin{multline}
		\label{lem:reliability-eq006}
		\Lambda_1=\sum_{T\in\CT_h}\left\{\int_{T} \left(\nu\Delta\bu_h-(\boldsymbol{\beta}\cdot\nabla)\bu_h-\nabla p_h+\lambda_h\bu_h\right)\cdot (\bw-\bw_I) \right.\\
		\left.+ \int_{T}(r-r_I)\div\bu_h\right\} 	+ \frac{1}{2}\sum_{e\in\CT_h}\int_{e}\jump{\left(\nu\nabla \bu_{h}- p_h\mathbb{I}\right)\cdot\bn}\cdot(\bw-\bw_I).
	\end{multline}
	Applying the Cauchy-Schwarz inequality to \eqref{lem:reliability-eq006}, followed by Cl\'ement interpolation properties, $\L^2$ projection properties, \eqref{lem:reliability-eq002} and the estimator definition, we obtain
	$$
	\Lambda_1 \leq C\eta.
	$$
	For $\Lambda_2$ we follow \cite[Theorem 3.1]{MR2473688} and obtain the estimate
	$$
	\Lambda_2\leq |\lambda-\lambda_h|+\lambda\|\bu-\bu_h\|_{0,\O}.
	$$
	The proof of \eqref{eq:reliabilityprimal} is completed by adding the estimates $\Lambda_1$ and 
	$\Lambda_2$, together with \eqref{lem:reliability-eq004}. The proof of \eqref{eq:reliabilitydual} is obtained in a similar way.
\end{proof}

	Now we are in position to establish the estimator reliability for the eigenvalues, whose proof follows immediately by squaring the estimates from Theorem \ref{th:reliability}, Remark \ref{summaryapriori} and \cite[Remark 2.1]{gedicke2014posteriori} we have the following eigenvalue estimate.

 \begin{proposition}
	The eigenvalue approximation is such that the estimate
\begin{align*}
		&\vert \lambda - \lambda_{h}\vert \leq C_{\nu,\boldsymbol{\beta}}\left(\theta^2+ \vert \lambda - \lambda_{h}\vert^2+ \lambda^2(\Vert \bu - \bu_{h}\Vert_{0,\O}^2+\Vert \bu^* - \bu_{h}^*\Vert_{0,\O}^2)\right),\label{eq:double-order-eigen-primal}
	\end{align*}
	holds, where the constant $C_{\nu,\boldsymbol{\beta}}>0$  is independent of the meshsize and the  discrete solutions.
	\end{proposition}
	\begin{proof}
	First we note that 
	\begin{align*}
	(\l_h-\l)(\bu_h,\bu_h^*)_{0,\O}=A((\bu-\bu_h,p-p_h);(\bu^*-\bu_h^*,p^*-p_h^*))\\-\l(\bu-\bu_h,\bu^*-\bu_h^*)_{0,\O}.
	\end{align*}
	Observe that the term $(\bu_h, \bu_h^*)_{0,\O}$ is needed to be lower bounded (see for instant \cite[Theorem 3.1]{MR3212379}). 
	 Then, there exists $C > 0$ such that $(\bu_h, \bu_h^*)_{0,\O} > C$. Then, taking modulus and applying triangle inequality, we have
	 \begin{align*}
	 |\l-\l_h|\leq |A((\bu-\bu_h,p-p_h),(\bu^*-\bu_h^*,p^*-p_h^*))|+|\l(\bu-\bu_h,\bu^*-\bu_h^*)_{0,\O}|\\
	 \leq C_{\nu,\boldsymbol{\beta}}\left( \|\bu-\bu_h\|_{1,\O}^2+\|\bu^*-\bu_h^*\|_{1,\O}^2+\|p^*-p_h^*\|_{0,\O}^2+\|p-p_h\|_{0,\O}^2\right.\\
	 \left.+ \lambda(\|\bu-\bu_h\|_{0,\O}^2+\|\bu^*-\bu_h^*\|_{0,\O}^2) \right)\\
	 \leq C_{\nu,\boldsymbol{\beta}}\left(\eta^2 +(\eta^*)^2+ \vert \lambda - \lambda_{h}\vert^2+ \lambda^2(\Vert \bu - \bu_{h}\Vert_{0,\O}^2+\Vert \bu^* - \bu_{h}^*\Vert_{0,\O}^2)\right).
	 \end{align*}
	 This concludes the proof.
	\end{proof}
		It is important to note that, thanks to Remark \ref{summaryapriori}, the errors $\|\bu-\bu_h\|_{0,\O}$ (resp.  $\|\bu^*-\bu_h^*\|_{0,\O}$), $\vert \lambda - \lambda_{h}\vert^2$ and $\lambda^2\Vert \bu - \bu_{h}\Vert_{0,\O}^2$ (resp. $\lambda^2\Vert \bu^* - \bu_{h}^*\Vert_{0,\O}^2$ )  are high-order terms in the estimations of the above results.

\subsection{Efficiency}
\label{sub:bubbles}
We begin by introducing the bubble functions for two dimensional elements. Given $T\in\mathcal{T}_h$ and $e\in\mathcal{E}(T)$, we let $\psi_T$ and $\psi_e$ be the usual triangle-bubble and edge-bubble functions, respectively. See \cite{MR3059294} for further details and properties  about these functions.

\subsection{Upper bound}
In order to simplify the presentation of the material, let us define $\textbf{R}_{1,T}:=\Delta\bu_h-(\boldsymbol{\beta}\cdot\nabla)\bu_h-\nabla p_h+\lambda_h\bu_h$. Also, we define $\bv_T:=\psi_T \textbf{R}_{1,T}$, where $\psi_T$ is the bubble that satisfies  the classical properties of bubbles. Now we compute a bound for the term $\|\textbf{R}_{1,T}\|_{0,T}$. To do this task, let us recall that 
the continuous problem satisfies $-\nu\Delta\bu+(\boldsymbol{\beta}\cdot\nabla)\bu+\nabla p-\lambda\bu=\boldsymbol{0}$. Hence
\begin{multline*}
\|\textbf{R}_{1,T}\|_{0,T}^2\leq\int_T(\nu\Delta\bu_h-(\boldsymbol{\beta}\cdot\nabla)\bu_h-\nabla p_h+\lambda_h\bu_h)\cdot\bv_T\\
=\int_T(-\nu\Delta\texttt{e}_{\bu}+(\boldsymbol{\beta}\cdot\nabla)\texttt{e}_{\bu}+\nabla\texttt{e}_p+\lambda_h\bu_h-\lambda\bu)\cdot\bv_T\\
=\underbrace{\int_T(-\nu\Delta\texttt{e}_{\bu}+(\boldsymbol{\beta}\cdot\nabla)\texttt{e}_{\bu}+\nabla\texttt{e}_p)\cdot\bv_T}_{\textrm{T}_1}+\underbrace{\int_T(\lambda_h\bu_h-\lambda\bu)\cdot\bv_T}_{\textrm{T}_2}.
\end{multline*}
Now our task is to estimate the terms $\textrm{T}_1$ and $\textrm{T}_2$. Hence we have
\begin{multline*}
\textrm{T}_1=\nu\int_T\nabla\texttt{e}_{\bu}:\nabla\bv_T+\int_T(\boldsymbol{\beta}\cdot\nabla)\texttt{e}_{\bu}\bv_T-\int_{T}\texttt{e}_p\div\bv_T\\
\leq\max\{\nu,\|\boldsymbol{\beta}\|_{\infty,T},1\}(\|\nabla\texttt{e}_{\bu}\|_{0,T}+\|\texttt{e}_p\|_{0,T})\|\nabla\bv_T\|_{0,T}\\
\leq\max\{\nu,\|\boldsymbol{\beta}\|_{\infty,T},1\}(\|\nabla\texttt{e}_{\bu}\|_{0,T}+\|\texttt{e}_p\|_{0,T})h_T^{-1}\|\textbf{R}_{1,T}\|_{0,T}.
\end{multline*}
Now for $\textrm{T}_2$ is enough to follow the arguments on \cite[Theorem 3.2]{MR2473688} to obtain
\begin{equation*}
\textrm{T}_2\lesssim h_T^2(|\lambda-\lambda_h|+\lambda\|\texttt{e}_{\bu}\|_{0,T})\|\textbf{R}_{1,T}\|_{0,T}.
\end{equation*}
Then
\begin{multline}\label{eq:R1T}
h_T\|\textbf{R}_{1,T}\|_{0,T}\leq C_{\nu,\boldsymbol{\beta}}\left(\|\nabla(\bu-\bu_h)\|_{0,T}+\|p-p_h\|_{0,T}\right.\\
\left.+h_T(|\lambda-\lambda_h|+\lambda\|\bu-\bu_h\|_{0,T})\right).
\end{multline}
The control of the term $\|\div\bu_h\|_{0,T}^2$ is direct:
\begin{equation*}
\label{eq:control_div}
\|\div\bu_h\|_{0,T}=\|\div\texttt{e}_{\bu}\|_{0,T}\leq \sqrt{n}\|\nabla\texttt{e}_{\bu}\|_{0,T},
\end{equation*}
where we have used the incompressibility condition $\div\bu=0$ in $\O$.

The following step is to estimate  the boundary term of the estimator $\eta$. Given $e\in\CE_{h}$, let us define $\textbf{R}_{e}:=\jump{\left(\nu\nabla \bu_{h}- p_h\mathbb{I}\right)\bn}$. Using the extension operator $L$ defined by $L : \mathcal{C}(e)\rightarrow \mathcal{C}(T )$ with $\mathcal{C}$ and $\mathcal{C}$ being the spaces of continuous functions defined on $e$ and $T$ , respectively, and using the properties of  $\psi_e$ the edge-bubble function,  we have that
\begin{multline*}
||\textbf{R}_{e}||_{0,e}^2=\int_e \psi_e L(\textbf{R}_{e})\cdot \jump{\left(\nu\nabla \bu_{h}- p_h\mathbb{I}\right)\cdot\bn}\\
=\sum_{T\in\omega_e}\int_{\partial e}\psi_e  L(\textbf{R}_{e})\cdot \jump{\left(\nu\nabla \bu_{h}- p_h\mathbb{I}\right)\cdot\bn},
\end{multline*}
where $\omega_e:=\{ T'\in CT_h: e\in \CE_{T'}\}$.
Now, using $\jump{\left(\nu\nabla \bu- p\mathbb{I}\right)\cdot\bn}=0$, and integrating by parts, we get 
\begin{multline*}
||\textbf{R}_{e}||_{0,e}^2=\sum_{T\in\omega_e}\left(\int_T\nu\nabla \texttt{e}_{\bu}:\nabla ( \psi_e L(\textbf{R}_{e}))+\int_T\texttt{e}_{p}\div(\psi_e L(\textbf{R}_{e}))\right.\\
\left.+\int_T\boldsymbol{\beta}\cdot \nabla\texttt{e}_{\bu}\cdot \psi_e L(\textbf{R}_{e})+\int_T(\lambda \bu-\lambda_h\bu_h)\cdot \psi_e L(\textbf{R}_{e})+\int_T\textbf{R}_{1,T}\cdot \psi_e L(\textbf{R}_{e})\right)\\
\leq \sum_{T\in\omega_e}C_{\nu,\boldsymbol{\beta}}\left(\|\nabla\texttt{e}_{\bu}\|_{0,T}+\|\texttt{e}_{p}\|_{0,T}+h_T\|\textbf{R}_{1,T}\|_{0,T}+h_T(|\l-\l_h|\right.\\
\left.+\|\texttt{e}_{\bu}\|_{0,T}\right)h_e^{-1/2}||\textbf{R}_{e}||_{0,e},
\end{multline*}
where we used that $h_e\leq h_T$. Combining the above result with \eqref{eq:R1T} we have that 
\begin{equation*}
h_e^{1/2}||\textbf{R}_{e}||_{0,e}\leq \sum_{T\in\omega_e}C_{\nu,\boldsymbol{\beta}}\left(\|\nabla\texttt{e}_{\bu}\|_{0,T}+\|\texttt{e}_{p}\|_{0,T}+h_T(|\l-\l_h|+\|\texttt{e}_{\bu}\|_{0,T}\right).
\end{equation*}
In summary we have proof that 
$$\eta_T^2\leq  \sum_{T\in\omega_e}C_{\nu,\boldsymbol{\beta}}\left(\|\nabla\texttt{e}_{\bu}\|_{0,T}^2+\|\texttt{e}_{p}\|_{0,T}^2+h_T^2(|\l-\l_h|^2+\|\texttt{e}_{\bu}\|_{0,T}^2\right).$$
Now, we are in a position to establish the efficiency $\eta$, which is stated in the following result.
\begin{lemma}(Efficiency) The following estimate holds 
$$\eta\leq C_{\nu,\boldsymbol{\beta}}(\|\bu-\bu_h\|_{1,\O}^2+\|p-p_h\|_{0,\O}^2+h.o.t),$$
where the constant $C_{\nu,\boldsymbol{\beta}}>0$ is independent of meshsize, $\lambda$, and the discrete solution.
\end{lemma}
Finally, the efficiency for the $\eta^*$ estimator is analogous to that shown for $\eta$, so the proof is omitted. 
\section{Numerical experiments}
\label{sec:numerics}
In this section we carry out several numerical experiments to visualize the robustness and performance of the proposed schemes. The discrete eigenvalue problem have been implemented using FEniCS \cite{logg2012automated}. After computing the eigenvalues,  the rates of convergence  are calculated by using a least-square fitting. In this sense, if $\lambda_h$ is a discrete complex eigenvalue, then the rate of convergence $\alpha$ is calculated by extrapolation and the least square  fitting
$$
\lambda_{h}\approx \lambda_{\text{extr}} + Ch^{\alpha},
$$
where $\lambda_{\text{extr}}$ is the extrapolated eigenvalue given by the fitting. For convenience in handling the two- and three-dimensional plots, the representation of eigenfunctions is done using Taylor-Hood elements.

 In what follows, we denote the mesh resolution by $N$, which is connected to the mesh-size $h$ trough the relation $h\sim N^{-1}$. We also denote the number of degrees of freedom by $\texttt{dof}$, namely $\texttt{dof}=\dim(\boldsymbol{V}_h)+\dim(\mathcal{P}_h)$. The relation between $\texttt{dof}$ and the mesh size is given by $h\sim\texttt{dof}^{-1/n}$, with $n\in\{2,3\}$.
 

Let us define $\err(\lambda_i)$ as the error on the $i$-th eigenvalue, with
$$
\err(\lambda_i):=\frac{\vert \lambda_{i,h}-\lambda_{i}\vert}{|\lambda_{i}|},
$$
where $\lambda_i$ is the extrapolated value. Similarly, the effectivity indexes with respect to $\eta$, $\eta_*$ and $\theta$ and the eigenvalue $\lambda_{i,h}$ is defined, respectively, by
$$
\eff(\lambda_i):=\frac{\err(\lambda_i)}{\eta^2}, \qquad \eff_*(\lambda_i):=\frac{\err(\lambda_i)}{(\eta^*)^2},\qquad \eff_\theta(\lambda_i):=\frac{\err(\lambda_i)}{\theta^2}.
$$

In order to apply the adaptive finite element method, we shall generate a sequence of nested conforming triangulations using the loop
\begin{center}
	\Large \textrm{solve $\rightarrow$ estimate $\rightarrow$ mark $\rightarrow$ refine,} 
\end{center}
based on \cite{verfuhrt1996}:
\begin{enumerate}
	\item Set an initial mesh $\CT_{h}$.
	\item Solve \eqref{def:oseen_system_weak_disc} (resp. \eqref{def:oseen_system_weak_dual_eigen_disc}) in the actual mesh to obtain $\lambda_{h}$ and $(\bu_{h},p_h)$ (resp. $(\bu_{h}^*,p_h^*)$ ). 	
	\item Compute $\eta_T$ (resp. $\eta_T^*$) for each $T\in\CT_{h}$ using the eigenfunctions $(\bu_{h},p_h)$ (resp. $(\bu_{h}^*,p_h^*)$ ). 
	\item Use blue-green marking strategy to refine each $T'\in \CT_{h}$ whose indicator $\zeta_{T'}$ satisfies
	$$
	\zeta_{T'}\geq 0.5\max\{\zeta_{T}\,:\,T\in\CT_{h} \},
	$$
	where $\zeta_{T}\in\{\eta_{T},\eta_{T}^*,\theta_T\}$.
	\item Set $\CT_{h}$ as the actual mesh and go to step 2.
\end{enumerate}

The refinement algorithm is the one implemented by FEniCS through the command \texttt{refine}, which implements Plaza and Carey's algorithms for 2D and 3D geometries. The algorithms use local refinement of simplicial grids based on the skeleton. 

%

For the study of the estimators, we comprise each contribution of the global residual terms as
$$
\begin{aligned}
	&\mathbf{R}:=\sum_{T\in\CT_h}h_T^2\|\lambda_h\bu_h+\nu\Delta\bu_h-(\boldsymbol{\beta}\cdot\nabla)\bu_h)-\nabla p_h\|_{0,T}^2.\\
	&\mathbf{D}:=\sum_{T\in\CT_h}\|\div\bu_h\|_{0,T}^2.\\
	&\mathbf{J}:=\frac{h_e}{2}\sum_{e\in\CT_h}\|\jump{(\nu\nabla \bu_h-p_h\mathbb{I})\cdot\boldsymbol{n}}\|_{0,e}^2.
\end{aligned}
$$
for the case of $\eta$, while each contribution of $\eta^*$ are given by
$$
\begin{aligned}
	&\mathbf{R}^*:=\sum_{T\in\CT_h}h_T^2\|\lambda_h\bu_h + \nu\Delta\bu^*_h+\div(\bu_h^*\otimes\boldsymbol{\beta})+\nabla p_h^*\|_{0,T}^2.\\
	&\mathbf{D}^*:=\sum_{T\in\CT_h}\|\div\bu_h^*\|_{0,T}^2.\\
	&\mathbf{J}^*:=\frac{h_e}{2}\sum_{e\in\CT_h}\|\jump{(\nu\nabla \bu_h^*+p_h^*\mathbb{I})\cdot\boldsymbol{n}}\|_{0,e}^2.
\end{aligned}
$$ 

\subsection{A priori numerical test}
In this section we deal with numerical results obtained on two and three dimensional geometries. Convex and non-convex geometries are considered to confirm the efficiency of the proposed estimators.

\subsubsection{A 2D square}
\label{subsec:num-exp-2D-square}
Let us begin with the two-dimensional square domain $\O:=(-1,1)^2$, with $\boldsymbol{\beta}=(1,0)^{\texttt{t}}$. A total of $8$ uniform refinements are performed on each iteration and the convergence and estimator efficiency is studied.

Tables \ref{table-square2D-primal-formulation} depicts the convergence history for the four lowest computed eigenvalues in the primal. To study the performance, we have computed the eigenvalues using the higher order Taylor-Hood elements $\mathbb{P}_3-\mathbb{P}_2$. In all cases presented, a convergence rate $\mathcal{O}(h^{2(k+1)})$ is observed, where $k=0$ for mini-elements, and $k\geq 1$ for Taylor-hood elements. A slight decrease in the fourth eigenvalue for the $\mathbb{P}_3-\mathbb{P}_2$ family is observed due to machine precision. The error curves for mini-element and the lowest-order Taylor-Hood family are presented in Figure \ref{fig:square2D-error-curves}.

Altough our a posteriori error analysis is performed for the mini-element family, in this test we also study the performance of the estimator when Taylor-Hood elements are used. It is important to remark that in this case, because of the regularity requirements for the eigenfunctions, the high order terms are not negligible. However, we present in Table \ref{table-square2D-first-eigenvalue} the results of the computations for each estimator component, and we observe that in both cases, the estimator is bounded above and below when uniform refinenments are performed. Also, we note that the dual estimator tends to be four orders of magnitude smaller when using Taylor-Hood elements.
\begin{table}[t!]
	\centering 
	{\footnotesize
		\begin{center}
			\caption{Example \ref{subsec:num-exp-2D-square}. Convergence behavior of the first four lowest computed eigenvalues for the primal formulation on the square domain with homogeneous boundary conditions and the field $\boldsymbol{\beta}=(1,0)$. }
			\begin{tabular}{c c c c |c| c}
				\hline
				\hline
				 $N=20$             &  $N=30$         &   $N=40$         & $N=50$ & Order & $\lambda_{\text{extr}}$ \\ 
				\midrule
				\multicolumn{6}{c}{Mini-Element $\mathbb{P}_{1,b}-\mathbb{P}_1$}\\
				\midrule
				13.7800  &    13.6826  &    13.6498  &    13.6350  & 2.15 &    13.6107  \\
				23.6532  &    23.3545  &    23.2539  &    23.2083  & 2.15 &    23.1340  \\
				23.9263  &    23.6386  &    23.5420  &    23.4983  & 2.16 &    23.4276  \\
				33.5447  &    32.8303  &    32.5908  &    32.4827  & 2.16 &    32.3069  \\
				
				\midrule
				\multicolumn{6}{c}{Taylor-Hood $\mathbb{P}_2-\mathbb{P}_1$}\\
				\midrule	
				
				13.6100  &    13.6097  &    13.6096  &    13.6096  & 4.06 &    13.6096  \\
				23.1319  &    23.1302  &    23.1299  &    23.1298  & 4.06 &    23.1297  \\
				23.4251  &    23.4234  &    23.4231  &    23.4230  & 4.06 &    23.4230  \\
				32.3055  &    32.2996  &    32.2986  &    32.2983  & 4.05 &    32.2981  \\

				\midrule
				\multicolumn{6}{c}{Taylor-Hood $\mathbb{P}_3-\mathbb{P}_2$}\\
				\midrule
				
				13.6096  &    13.6096  &    13.6096  &    13.6096  & 6.46 &    13.6096  \\
				23.1298  &    23.1297  &    23.1297  &    23.1297  & 6.03 &    23.1297  \\
				23.4230  &    23.4230  &    23.4230  &    23.4230  & 6.01 &    23.4230  \\
				32.2982  &    32.2981  &    32.2981  &    32.2981  & 5.78 &    32.2981  \\
				
				\hline
				\hline             
			\end{tabular}
	\end{center}}
	\smallskip
	
	\label{table-square2D-primal-formulation}
\end{table}

\begin{table}[t!]
	\setlength{\tabcolsep}{3.5pt}
	\centering 
	\caption{Example \ref{subsec:num-exp-2D-square}. Lowest computed eigenvalue error history for different finite element families of $\mathbb{P}_2-\mathbb{P}_1$ and $\mathbb{P}_{1,b}-\mathbb{P}_1$ for $\bu_h$ and $p_h$, respectively, on the unit square domain $\O=(-1,1)^2$. Here, the convective velocity is set to be $\boldsymbol{\beta}=(1,0)^{\texttt{t}}$. }
	{\footnotesize\begin{tabular}{rccccccc}
			\hline\hline
			dof   &   $h$  & $\err(\lambda_1)$  &   $\eta^2$   &   $\eta_*^2$  &   $\eff(\lambda_1)$  &  $\eff_*(\lambda_1)$  &   $r(\lambda) $ \\
			\hline 
			\multicolumn{8}{c}{Taylor-Hood} \\
			\hline
			 1004 & 0.283  & $4.5667e-04  $   & $1.2523e+00 $  & $5.9123e-02 $  & $3.6467e-04 $ & $7.7241e-03 $ & $0.00 $ \\
			 3804 & 0.141  & $3.0346e-05  $   & $7.8499e-02 $  & $4.1780e-03 $  & $3.8658e-04 $ & $7.2633e-03 $ & $3.91 $ \\
			 8404 & 0.094  & $6.0671e-06  $   & $1.5723e-02 $  & $8.6079e-04 $  & $3.8588e-04 $ & $7.0482e-03 $ & $3.97 $ \\
			 14804 & 0.071  & $1.9173e-06  $   & $5.0067e-03 $  & $2.7796e-04 $  & $3.8295e-04 $ & $6.8978e-03 $ & $4.00 $ \\
			 23004 & 0.057  & $7.8043e-07  $   & $2.0546e-03 $  & $1.1521e-04 $  & $3.7984e-04 $ & $6.7739e-03 $ & $4.03 $ \\
			 33004 & 0.047  & $3.7266e-07  $   & $9.9061e-04 $  & $5.5994e-05 $  & $3.7620e-04 $ & $6.6554e-03 $ & $4.05 $ \\
			 44804 & 0.040  & $1.9832e-07  $   & $5.3409e-04 $  & $3.0390e-05 $  & $3.7133e-04 $ & $6.5260e-03 $ & $4.09 $ \\
			 58404 & 0.035  & $1.1394e-07  $   & $3.1260e-04 $  & $1.7886e-05 $  & $3.6449e-04 $ & $6.3703e-03 $ & $4.15 $ \\
			\hline
			\multicolumn{8}{c}{Mini-element}\\
			\hline
			764 & 0.283  & $5.4785e-02  $   & $1.8830e+01 $  & $1.1453e+02 $  & $2.9095e-03 $ & $4.7836e-04 $ & $0.00 $ \\
			2924 & 0.141  & $1.2369e-02  $   & $3.9700e+00 $  & $2.0629e+01 $  & $3.1156e-03 $ & $5.9960e-04 $ & $2.15 $ \\
			6484 & 0.094  & $5.2093e-03  $   & $1.6165e+00 $  & $7.4313e+00 $  & $3.2227e-03 $ & $7.0099e-04 $ & $2.13 $ \\
			11444 & 0.071  & $2.8004e-03  $   & $8.5963e-01 $  & $3.5890e+00 $  & $3.2576e-03 $ & $7.8027e-04 $ & $2.16 $ \\
			17804 & 0.057  & $1.7118e-03  $   & $5.2940e-01 $  & $2.0478e+00 $  & $3.2334e-03 $ & $8.3589e-04 $ & $2.21 $ \\
			25564 & 0.047  & $1.1299e-03  $   & $3.5756e-01 $  & $1.3006e+00 $  & $3.1601e-03 $ & $8.6876e-04 $ & $2.28 $ \\
			34724 & 0.040  & $7.8316e-04  $   & $2.5725e-01 $  & $8.8967e-01 $  & $3.0443e-03 $ & $8.8029e-04 $ & $2.38 $ \\
			45284 & 0.035  & $5.6008e-04  $   & $1.9377e-01 $  & $6.4246e-01 $  & $2.8904e-03 $ & $8.7177e-04 $ & $2.51 $ \\
			\hline
			\hline
	\end{tabular}}
	\smallskip
	
	\label{table-square2D-first-eigenvalue}
	\end{table}
	
	 \begin{figure}[t!]
		\centering
		\includegraphics[scale=0.45]{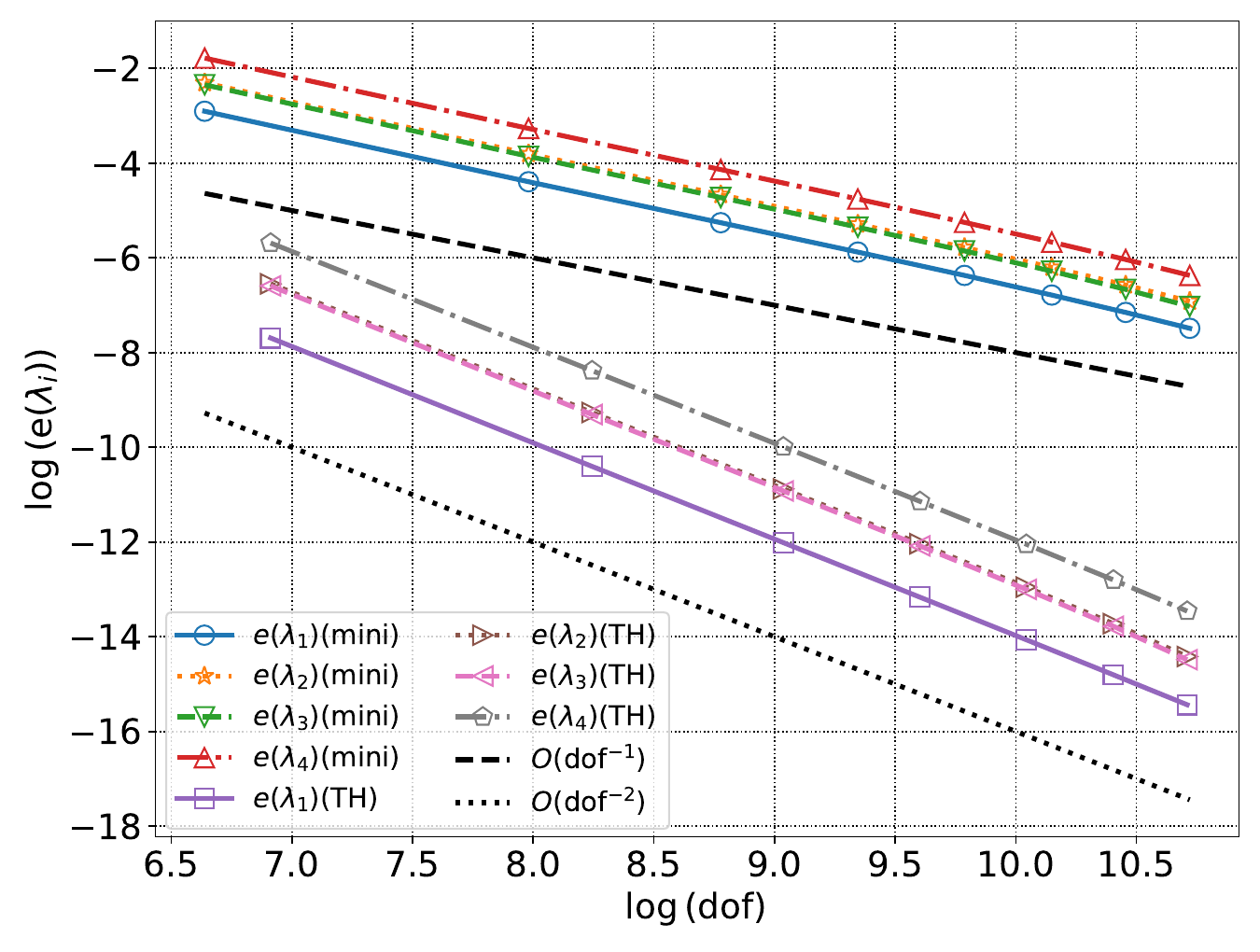}
		\caption{Example \ref{subsec:num-exp-2D-square}. Error curves for the real part of the first five lowest eigenvalues in the unit square domain. The convective velocity coefficient is set to be $\boldsymbol{\beta}=(1,0)^{\texttt{t}}$. }
		\label{fig:square2D-error-curves}
	\end{figure}
	
%
	

\subsubsection{Convergence to Stokes problem}\label{subsec:oseentostokes-test}
In this experiment we analyze the spectrum of \eqref{def:oseen_system_weak_disc} when $\Vert\boldsymbol{\beta}\Vert_{\infty,\Omega}\rightarrow 0$ in order to observe experimentally the convergence to a Stokes eigenvalue problem. We consider the same square domain as the last experiment, and the mini-element family is used for simplicity. Similar results are achieved when using Taylor-Hood elements.

A series of values of $\boldsymbol{\beta}=(2^{-i},0)^{\texttt{t}}$, for $i=0,...,15$ are considered, and we compute the first nine eigenvalues. For maximum accuracy, we set the mesh level $N=150$, implying that $\texttt{dof}=158404$. The results of this calculation are portrayed in Figure \ref{fig:oseen-to-stokes}, where we have plotted the exact spectrum of the Stokes eigenvalue problem. Here, we have considered only the primal spectrum because it behaves similar to that of the dual problem. The convergence in the limit is clearly visible.
	\begin{figure}[t!]
	\centering
	\includegraphics[scale=0.45]{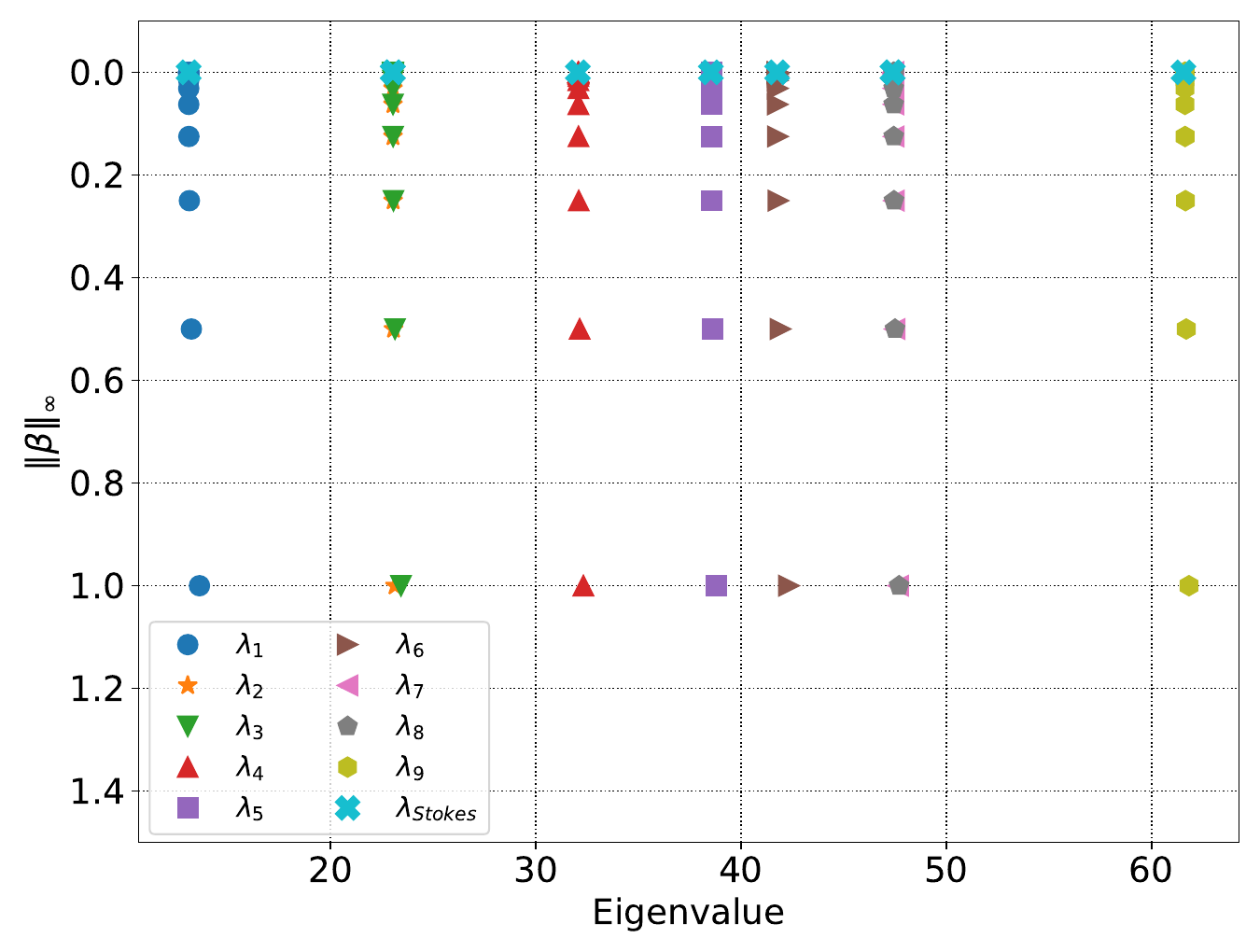}
	\caption{Example \ref{subsec:oseentostokes-test}. Comparison of the Stokes eigenvalues with the spectrum of \eqref{def:oseen_system_weak_disc} for several convective velocities such that $\Vert\boldsymbol{\beta}\Vert_{\infty,\Omega}\rightarrow 0$. }
	\label{fig:oseen-to-stokes}
\end{figure}

\subsubsection{Convergence on 3D geometries}\label{subsec:apriori-3D}
This test aims to study the perfomance of the method when considering three-dimensional geometries. A unit cube domain $\Omega_c=(0,1)^3$ and a unit radius sphere with center on the origin
$$\Omega_s:=\{(x,y,z)\in\mathbb{R}^3\;:\; x^2+y^2+z^1\leq1\},$$ 
 are considered.

A computation with several mesh levels in $\Omega_c$ is presented in Table \ref{table-cube3D-dual-formulation}. Here, we observe that a convergence rate of $\mathcal{O}(\texttt{dof}^{-0.66})\approx \mathcal{O}(h^2)$ is observed. Similar results were observed when the domain $\Omega_s$ is considered. For this case, it is important to mention that, because of the variational crime of triangulating the sphere using tetrahedrons, the best rate that we expect is $\mathcal{O}(h^2)$, which is the one observed on the table. Finally, we depict in Figure \ref{fig:3D_apriori_cube_and_sphere-velocity} the velocity fields for the cube and the sphere, accompanied with pressure surface contour plots, presented in Figure \ref{fig:3D_apriori_cube_and_sphere-pressure}.

\begin{table}[t!]
	\centering 
	{\footnotesize
		\begin{center}
			\caption{Example \ref{subsec:apriori-3D}. Convergence behavior of the first four lowest computed eigenvalues for the primal and adjoint formulation on the cube domain $\Omega_c$ with homogeneous boundary conditions and the field $\boldsymbol{\beta}=(0,0,1)^{\texttt{t}}$. Here, the mini-element family $\mathbb{P}_{1,b}-\mathbb{P}_1$ for velocity and pressure is used for the discretization.}
			\begin{tabular}{c c c c |c| c}
				\toprule
				$N=5$             &  $N=10$         &   $N=15$         & $N=20$ & Order & $\lambda_{\text{extr}}$ \\ 
				\midrule
				\multicolumn{6}{c}{Primal formulation}\\
				\midrule
				 84.6107  &    67.4096  &    64.6456  &    63.6936  & 2.30 &    62.7468  \\
				89.2399  &    68.4680  &    65.1065  &    64.0006  & 2.30 &    62.8363  \\
				89.3685  &    68.6426  &    65.2630  &    64.1266  & 2.28 &    62.9327  \\
				137.3480  &   102.5965  &    96.7467  &    94.6599  & 2.21 &    92.4487  \\
				\hline
				\hline             
			\end{tabular}
	\end{center}}
	\smallskip
	\label{table-cube3D-dual-formulation}
\end{table}

\begin{table}[t!]
	\centering 
	{\footnotesize
		\begin{center}
			\caption{Example \ref{subsec:apriori-3D}. Comparison on the convergence behavior of the first four lowest computed eigenvalues between Taylor-Hood and Mini-element families on the sphere domain $\Omega_s$ with homogeneous boundary conditions and the field $\boldsymbol{\beta}=(0,0,1)^{\texttt{t}}$.}
			\begin{tabular}{c c c c |c| c}
				\toprule
				$N=20$             &  $N=25$         &   $N=30$         & $N=35$ & Order & $\lambda_{\text{extr}}$ \\ 
				\midrule
				\multicolumn{6}{c}{Mini-Element $\mathbb{P}_{1,b}-\mathbb{P}_1$}\\
				\midrule
				21.5361  &    21.1476  &    20.9608  &    20.8696  & 2.03 &    20.6469  \\
				21.7385  &    21.3731  &    21.2010  &    21.1139  & 2.05 &    20.9101  \\
				21.7610  &    21.3783  &    21.2027  &    21.1151  & 2.14 &    20.9208  \\
				36.5358  &    35.2685  &    34.6768  &    34.3526  & 1.98 &    33.5997  \\
				\midrule
				\multicolumn{6}{c}{Taylor-Hood $\mathbb{P}_{2}-\mathbb{P}_1$}\\
				\midrule
				20.8096  &    20.7285  &    20.7044  &    20.6870  & 2.13 &    20.6489  \\
				21.0637  &    20.9839  &    20.9618  &    20.9436  & 2.13 &    20.9067  \\
				21.0660  &    20.9856  &    20.9627  &    20.9438  & 2.05 &    20.9034  \\
				33.8969  &    33.7529  &    33.7150  &    33.6825  & 2.21 &    33.6215  \\
				\hline
				\hline             
			\end{tabular}
	\end{center}}
	\smallskip
	\label{table-sphere3D-formulation}
\end{table}

\begin{figure}[!h]
	\centering
	\begin{minipage}{0.49\linewidth}
		\includegraphics[scale=0.065,trim=40cm 0cm 40cm 0cm,clip]{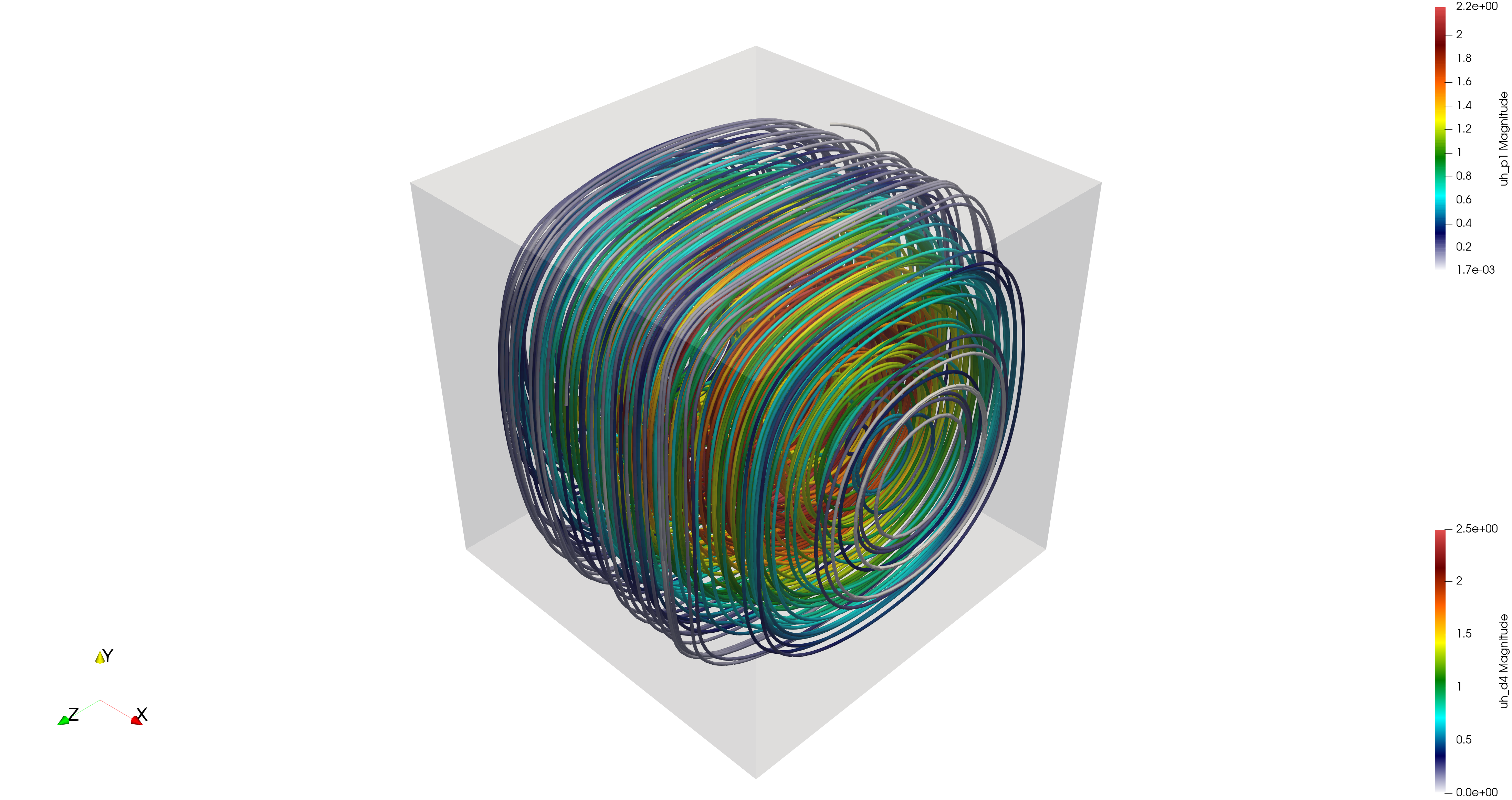}
	\end{minipage}
	\begin{minipage}{0.49\linewidth}
		\includegraphics[scale=0.060,trim=40cm 0cm 40cm 0cm,clip]{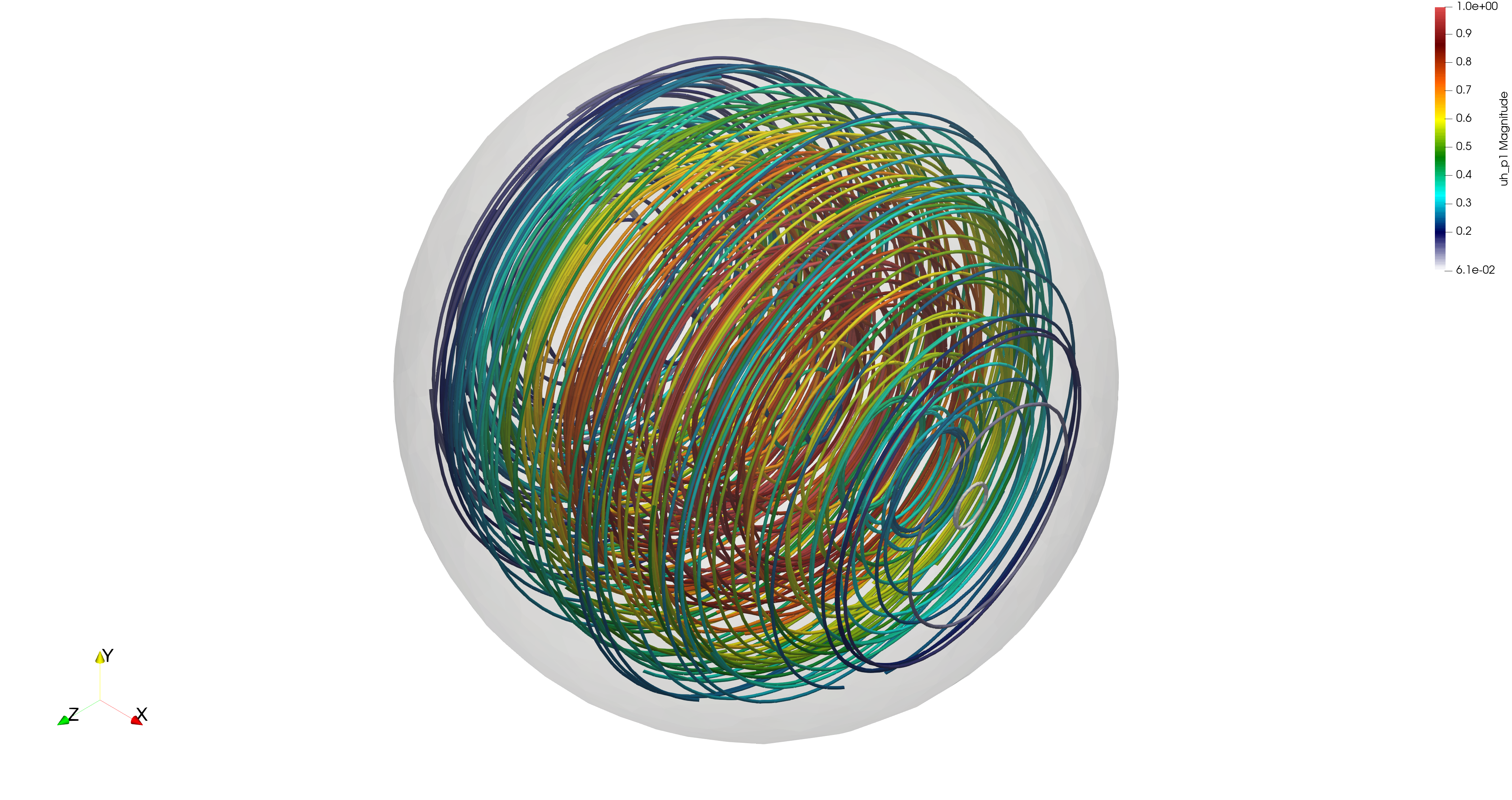}
	\end{minipage}\\
	\begin{minipage}{0.49\linewidth}
		\includegraphics[scale=0.065,trim=40cm 0cm 40cm 0cm,clip]{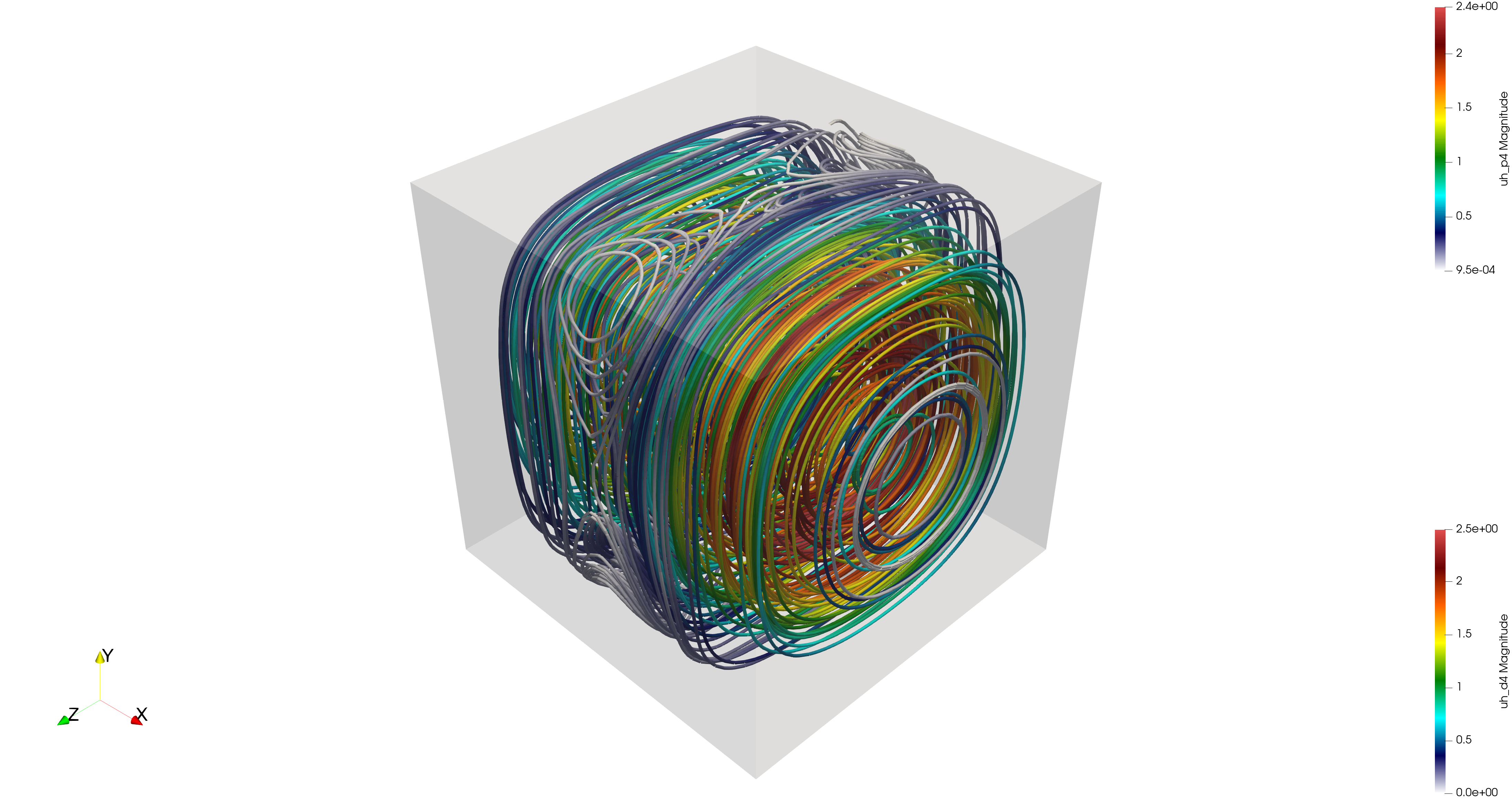}
	\end{minipage}
	\begin{minipage}{0.49\linewidth}\vspace*{0.4cm}
		\includegraphics[scale=0.060,trim=40cm 0cm 40cm 0cm,clip]{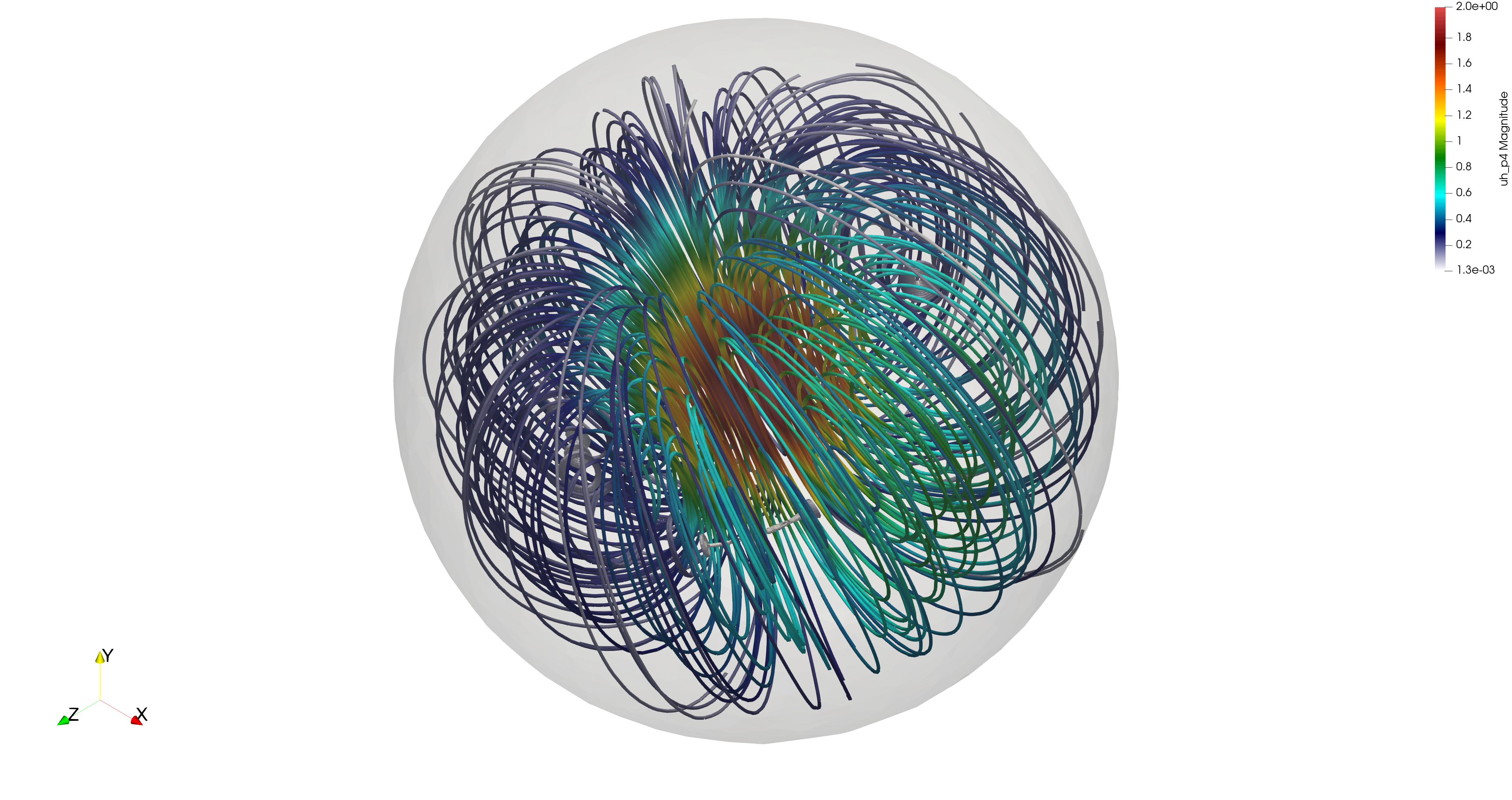}\\
	\end{minipage}
	\caption{Example \ref{subsec:apriori-3D}. Comparison between the first and four lowest velocities streamlines eigenfunctions on the domains $\Omega_c$ and $\Omega_s$. Top: $\bu_{1,h}$ on $\Omega_c$ and $\Omega_s$, respectively. Bottom: $\bu_{4,h}$ on $\Omega_c$ and $\Omega_s$, respectively.}
	\label{fig:3D_apriori_cube_and_sphere-velocity}
\end{figure}

\begin{figure}[!h]
	\centering
	\begin{minipage}{0.49\linewidth}
		\includegraphics[scale=0.065,trim=40cm 0cm 40cm 0cm,clip]{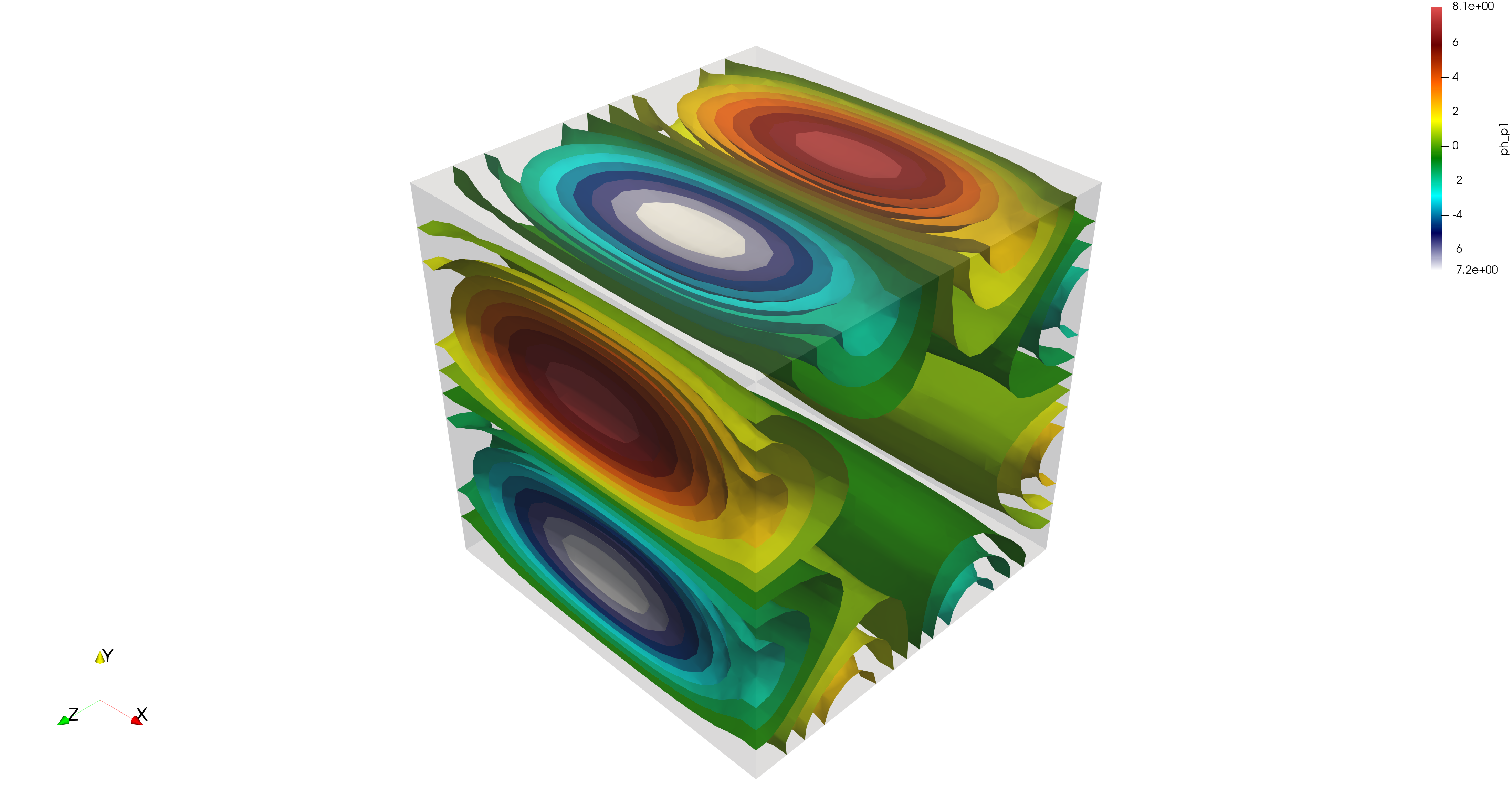}
	\end{minipage}
	\begin{minipage}{0.49\linewidth}
		\includegraphics[scale=0.060,trim=40cm 0cm 40cm 0cm,clip]{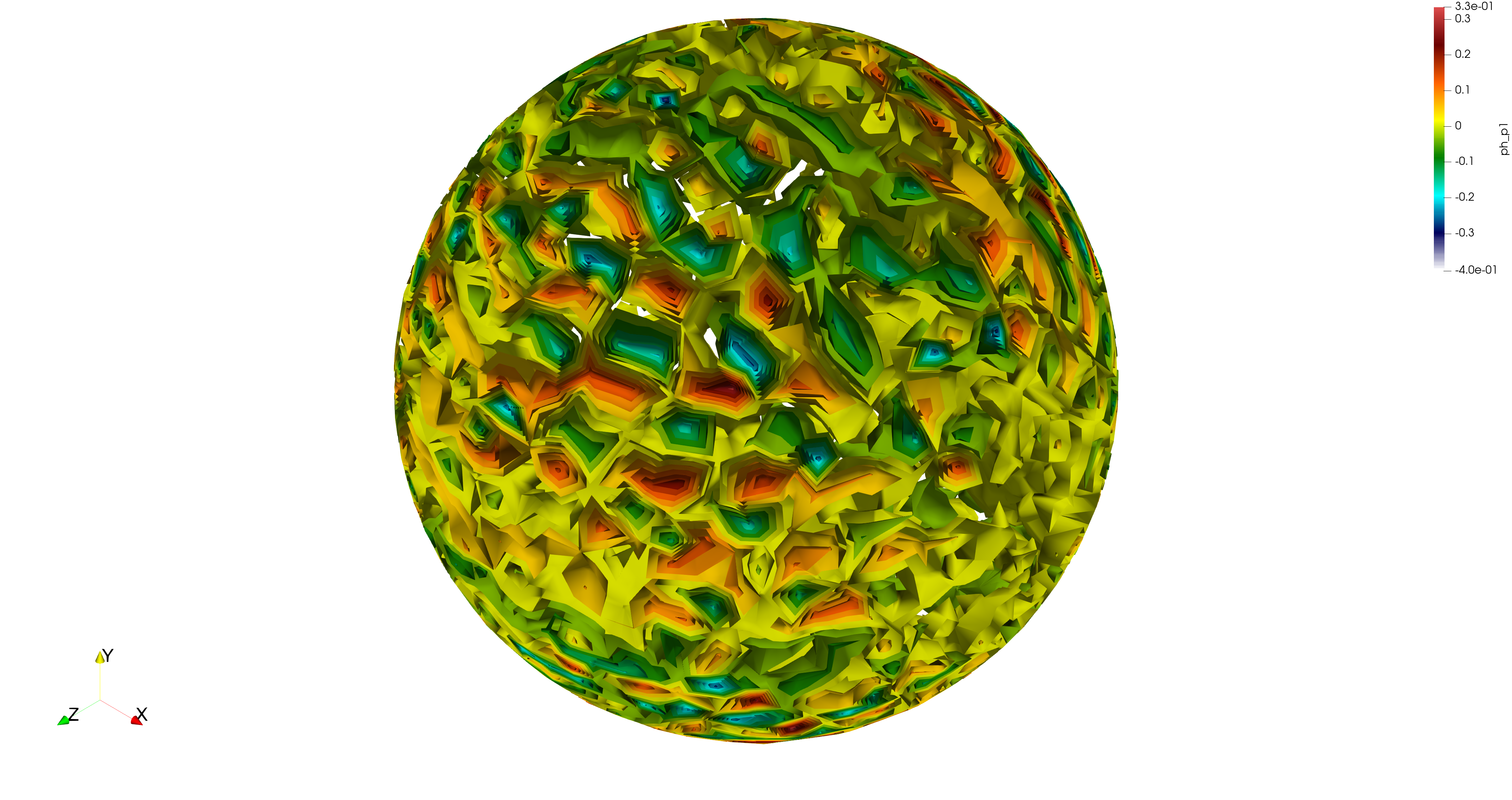}
	\end{minipage}\\
	\begin{minipage}{0.49\linewidth}
		\includegraphics[scale=0.065,trim=40cm 0cm 40cm 0cm,clip]{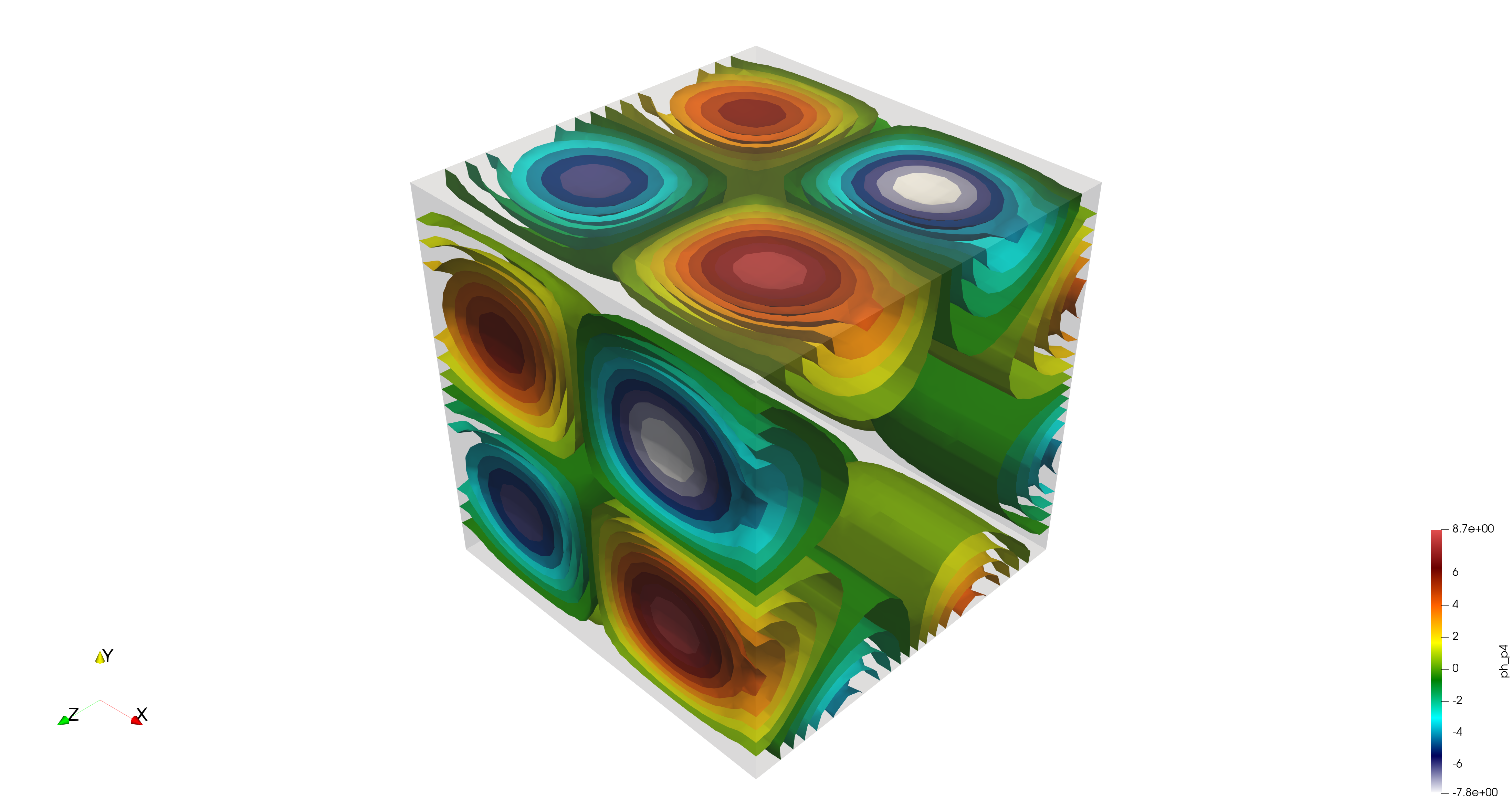}
	\end{minipage}
	\begin{minipage}{0.49\linewidth}\vspace*{0.4cm}
		\includegraphics[scale=0.060,trim=40cm 0cm 40cm 0cm,clip]{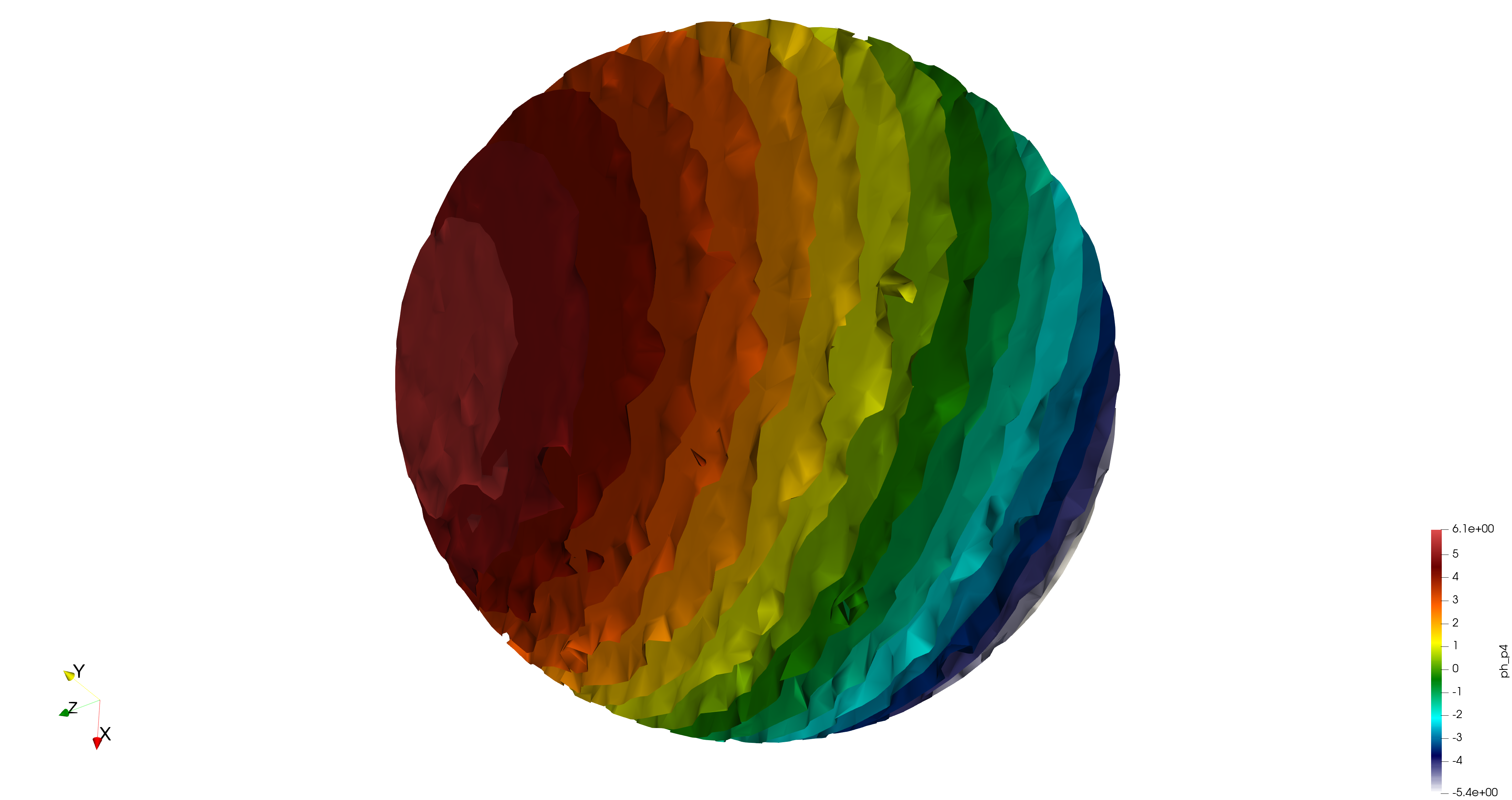}\\
	\end{minipage}
	\caption{Example \ref{subsec:apriori-3D}. Comparison between the first and fourth lowest pressure eigenfunctions contour plot on the domains $\Omega_c$ and $\Omega_s$. Top: $p_{1,h}$ on $\Omega_c$ and $\Omega_s$, respectively. Bottom: $p_{4,h}$ on $\Omega_c$ and $\Omega_s$, respectively.}
	\label{fig:3D_apriori_cube_and_sphere-pressure}
\end{figure}

\subsection{A posteriori test}
This section is devoted to study the performance of the estimator when domains with singularities are considered in two and three dimensions. All the test are implemented using the mini-element family in order to observe the recovery of the optimal rate of convergence on each case.

\subsubsection{A 2D L-shaped domain}\label{subsec:lshape2D}
This experiment aims to study the adaptive algorithm in a 2D domain with a reentrant corner. The domain under consideration is the usual L-shaped domain $\Omega:=(-1,1)^2\backslash(-1,0)^2$. In this particular domain, the regularity of the first eigenmode decreases, hence we expect a rate of  $\mathcal{O}(h^{s})$, with $s\geq 1.32$ if uniform refinement are used. The extrapolated eigenvalues taken as the exact solution for the primal and dual problem is given by 
$$\lambda_{1}=32.963150646072528$$
Below we present the performance of our estimator on this domain in the recovery of the optimal rate of convergence $\mathcal{O}(h^2)$. The error, volume and jumps contributions are presented in Tables \ref{table-lshape2D-first-eigenvalue-primal} -- \ref{table-lshape2D-first-eigenvalue-dual}. Here, we note that for 15 iterations, the primal estimator tends to mark more elements than the dual formulation. This behavior is expected because of the shift in the velocity eigenmode, depicted in Figure \ref{fig:lshape2D-uh_ph}, where we also observe that the pressure is singular near $(x,y)=(0,0)$. An example of the estimators performance is depicted in Figure \ref{fig:lshape2D-meshes}, where refinements near the singularity are evident. We also note that $\mathbf{R}$ and $\mathbf{J}$ (resp. $\mathbf{R}^*$ and $\mathbf{J}^*$) are the quantities that contribute the most to the estimator $\eta$ (resp. $\eta^*$). In both cases, the effectivity index remains bounded. As comparison, we present the error curves and effectivity indexes in Figure \ref{fig:lshape2D-error}, where an experimental rate $\mathcal{O}(\texttt{dof}^{-1})$ is observed for both problems. 
\begin{table}[t!]
	\setlength{\tabcolsep}{3.5pt}
	\centering 
	\caption{Example \ref{subsec:lshape2D}. Comparison of the error history of the lowest computed primal eigenvalue with the global residual terms and $\eta^2$, on the two-dimensional L-shaped domain. Here, the convective velocity is set to be $\boldsymbol{\beta}=(1,0)^{\texttt{t}}$. }
	{\footnotesize\begin{tabular}{rccccccc}
			\hline\hline
			dof   &   $\err(\lambda_1)$  &   $\mathbf{R}$   &   $\mathbf{D}$  &  $\mathbf{J}$& $\eta^2$&  $\eff(\lambda_1)$  \\
			\hline 
			\hline
			388 &  $1.0203e+01  $ & $1.1643e+02  $   & $3.8551e+00 $  & $1.9211e+02 $  & $3.1239e+02 $ & $5.3558e-03 $ \\
			778 &  $3.3535e+00  $ & $2.2805e+01  $   & $1.2871e+00 $  & $5.6211e+01 $  & $8.0303e+01 $ & $7.8040e-03 $ \\
			1238 &  $1.8747e+00  $ & $1.0866e+01  $   & $7.4111e-01 $  & $3.2443e+01 $  & $4.4050e+01 $ & $9.2355e-03 $ \\
			2193 &  $1.0324e+00  $ & $5.3005e+00  $   & $3.9367e-01 $  & $1.7487e+01 $  & $2.3181e+01 $ & $1.0952e-02 $ \\
			3293 &  $6.6180e-01  $ & $3.2094e+00  $   & $2.4977e-01 $  & $1.0912e+01 $  & $1.4371e+01 $ & $1.1285e-02 $ \\
			4513 &  $4.9614e-01  $ & $2.1453e+00  $   & $1.8164e-01 $  & $8.0998e+00 $  & $1.0427e+01 $ & $1.2660e-02 $ \\
			7076 &  $3.2336e-01  $ & $1.3290e+00  $   & $1.1818e-01 $  & $5.2945e+00 $  & $6.7417e+00 $ & $1.3080e-02 $ \\
			10512 &  $2.1595e-01  $ & $8.1067e-01  $   & $8.0439e-02 $  & $3.4716e+00 $  & $4.3627e+00 $ & $1.4245e-02 $ \\
			14077 &  $1.4979e-01  $ & $5.6845e-01  $   & $5.6494e-02 $  & $2.4521e+00 $  & $3.0770e+00 $ & $1.4866e-02 $ \\
			19517 &  $1.0983e-01  $ & $4.0271e-01  $   & $4.1821e-02 $  & $1.8174e+00 $  & $2.2619e+00 $ & $1.5598e-02 $ \\
			30191 &  $7.2192e-02  $ & $2.5515e-01  $   & $2.6907e-02 $  & $1.1857e+00 $  & $1.4678e+00 $ & $1.5291e-02 $ \\
			43989 &  $4.6942e-02  $ & $1.6744e-01  $   & $1.8034e-02 $  & $7.9266e-01 $  & $9.7814e-01 $ & $1.6249e-02 $ \\
			61639 &  $3.1572e-02  $ & $1.1828e-01  $   & $1.2386e-02 $  & $5.4936e-01 $  & $6.8003e-01 $ & $1.6454e-02 $ \\
			85992 &  $2.2001e-02  $ & $8.5341e-02  $   & $9.0940e-03 $  & $4.0362e-01 $  & $4.9806e-01 $ & $1.6679e-02 $ \\
			131770 &  $1.3896e-02  $ & $5.3835e-02  $   & $6.0354e-03 $  & $2.6681e-01 $  & $3.2668e-01 $ & $1.6552e-02 $ \\
			\hline
			\hline
	\end{tabular}}
	\smallskip
	\label{table-lshape2D-first-eigenvalue-primal}
\end{table}

\begin{table}[t!]
	\setlength{\tabcolsep}{3.5pt}
	\centering 
	\caption{Example \ref{subsec:lshape2D}. Comparison of the error history of the lowest computed dual eigenvalue with the global residual terms and $(\eta^*)^2$, on the two-dimensional L-shaped domain. Here, the convective velocity is set to be $\boldsymbol{\beta}=(1,0)^{\texttt{t}}$. }
	{\footnotesize\begin{tabular}{rccccccc}
			\hline\hline
			dof   &   $\err_*(\lambda_1)$  &   $\mathbf{R}^*$   &   $\mathbf{D}^*$  &  $\mathbf{J}^*$& $(\eta^*)^2$&  $\eff_*(\lambda_1)$  \\
			\hline 
			\hline
			388 &  $1.0202e+01  $   & $1.7203e+03  $   & $4.0861e+00 $  & $1.8056e+02 $  & $1.9049e+03 $ & $5.3561e-03$ \\
			674 &  $4.6447e+00  $   & $5.2107e+02  $   & $1.7125e+00 $  & $7.2388e+01 $  & $5.9517e+02 $ & $5.6345e-03$ \\
			860 &  $3.6255e+00  $   & $3.3621e+02  $   & $1.2901e+00 $  & $5.5063e+01 $  & $3.9256e+02 $ & $4.7756e-03$ \\
			1220 &  $2.6282e+00  $   & $2.0174e+02  $   & $9.5675e-01 $  & $3.7278e+01 $  & $2.3997e+02 $ & $4.3022e-03$ \\
			1828 &  $1.5533e+00  $   & $1.1351e+02  $   & $5.8244e-01 $  & $2.3551e+01 $  & $1.3765e+02 $ & $4.8080e-03$ \\
			2610 &  $1.1142e+00  $   & $7.0622e+01  $   & $4.3224e-01 $  & $1.6953e+01 $  & $8.8007e+01 $ & $5.6375e-03$ \\
			3600 &  $8.0170e-01  $   & $4.8529e+01  $   & $3.2416e-01 $  & $1.2442e+01 $  & $6.1294e+01 $ & $5.2756e-03$ \\
			5537 &  $5.2732e-01  $   & $2.9537e+01  $   & $1.8703e-01 $  & $7.2933e+00 $  & $3.7018e+01 $ & $5.8338e-03$ \\
			7672 &  $3.4993e-01  $   & $1.8393e+01  $   & $1.2911e-01 $  & $5.0176e+00 $  & $2.3539e+01 $ & $6.3632e-03$ \\
			10257 &  $2.5877e-01  $   & $1.2658e+01  $   & $9.8517e-02 $  & $3.8336e+00 $  & $1.6590e+01 $ & $6.6203e-03$ \\
			14227 &  $1.8628e-01  $   & $9.2889e+00  $   & $7.3963e-02 $  & $2.8197e+00 $  & $1.2183e+01 $ & $5.9259e-03$ \\
			20480 &  $1.2639e-01  $   & $5.9365e+00  $   & $4.7464e-02 $  & $1.7944e+00 $  & $7.7784e+00 $ & $6.0349e-03$ \\
			27693 &  $9.1463e-02  $   & $4.1882e+00  $   & $3.5170e-02 $  & $1.3354e+00 $  & $5.5588e+00 $ & $5.6798e-03$ \\
			38543 &  $6.5957e-02  $   & $2.9215e+00  $   & $2.6242e-02 $  & $1.0068e+00 $  & $3.9545e+00 $ & $5.5637e-03$ \\
			52762 &  $4.8409e-02  $   & $2.1559e+00  $   & $1.9560e-02 $  & $7.4915e-01 $  & $2.9246e+00 $ & $4.7514e-03$ \\
			\hline
			\hline
	\end{tabular}}
	\smallskip
	\label{table-lshape2D-first-eigenvalue-dual}
\end{table}
\begin{figure}[!h]
	\centering
	\begin{minipage}{0.32\linewidth}
		\includegraphics[scale=0.05,trim= 37cm 1cm 37cm 1cm, clip]{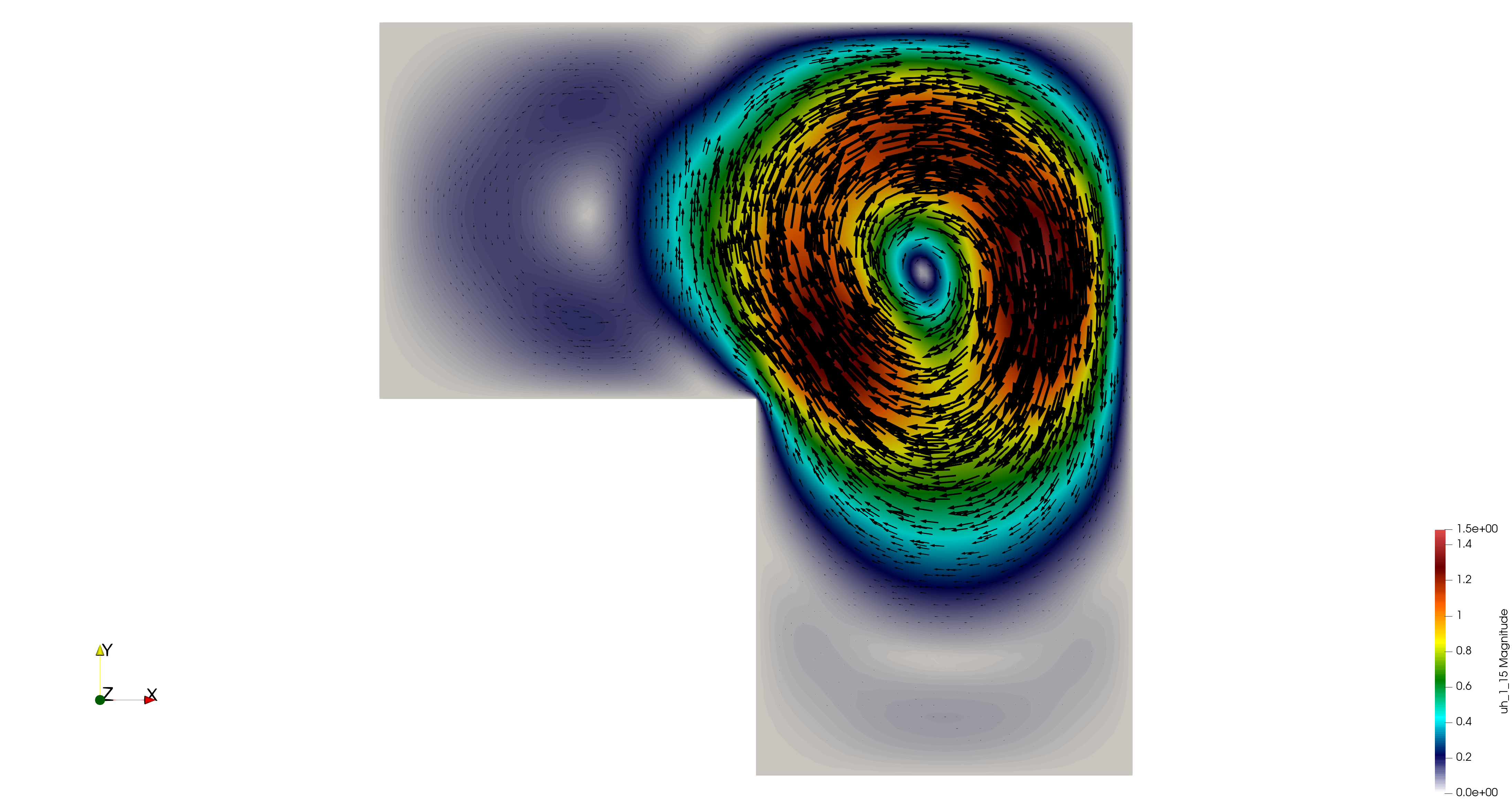}
	\end{minipage}
	\begin{minipage}{0.32\linewidth}
		\includegraphics[scale=0.05,trim= 37cm 1cm 37cm 1cm, clip]{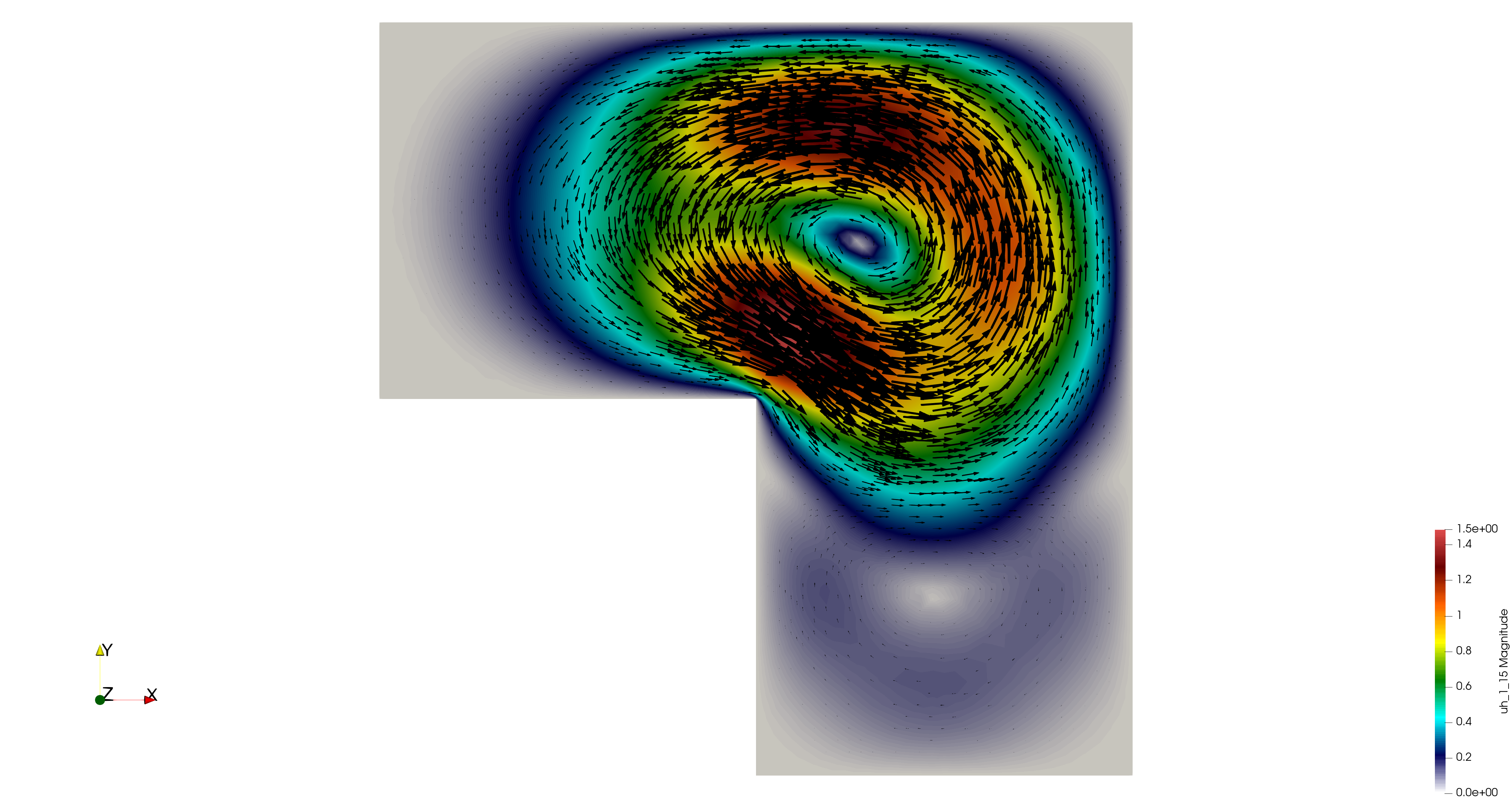}
	\end{minipage}
	\begin{minipage}{0.32\linewidth}
		\includegraphics[scale=0.05,trim= 37cm 1cm 37cm 1cm, clip]{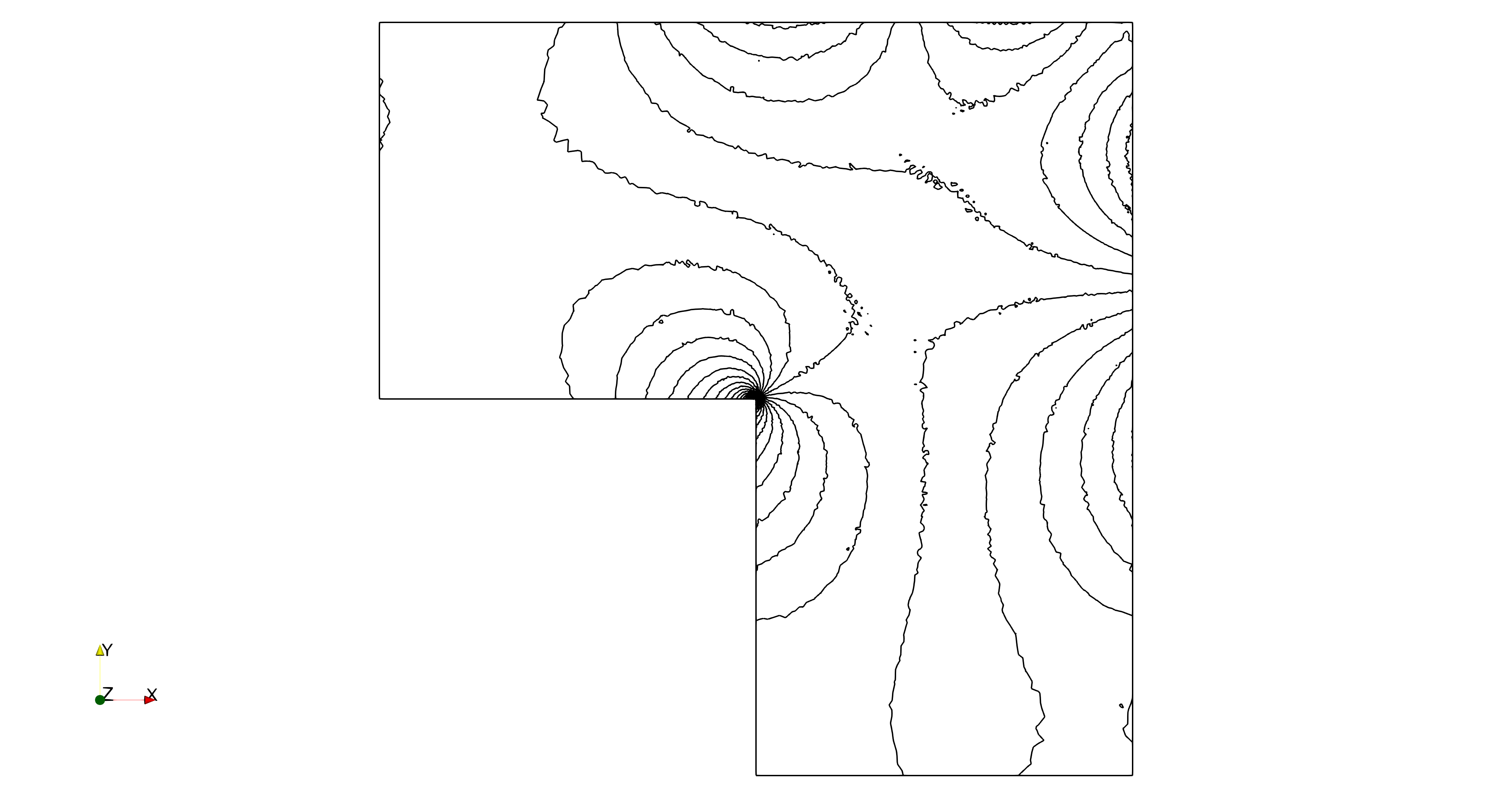}
	\end{minipage}\\
	\caption{Example \ref{subsec:lshape2D}. Velocity fields (left and middle) for the lowest order computed eigenmodes for the primal and dual problems, respectively, together with the corresponding singular pressure contour plot (right).}
	\label{fig:lshape2D-uh_ph}
\end{figure}
\begin{figure}[!h]
	\centering
	\begin{minipage}{0.32\linewidth}
		\includegraphics[scale=0.05,trim= 37cm 1cm 37cm 1cm, clip]{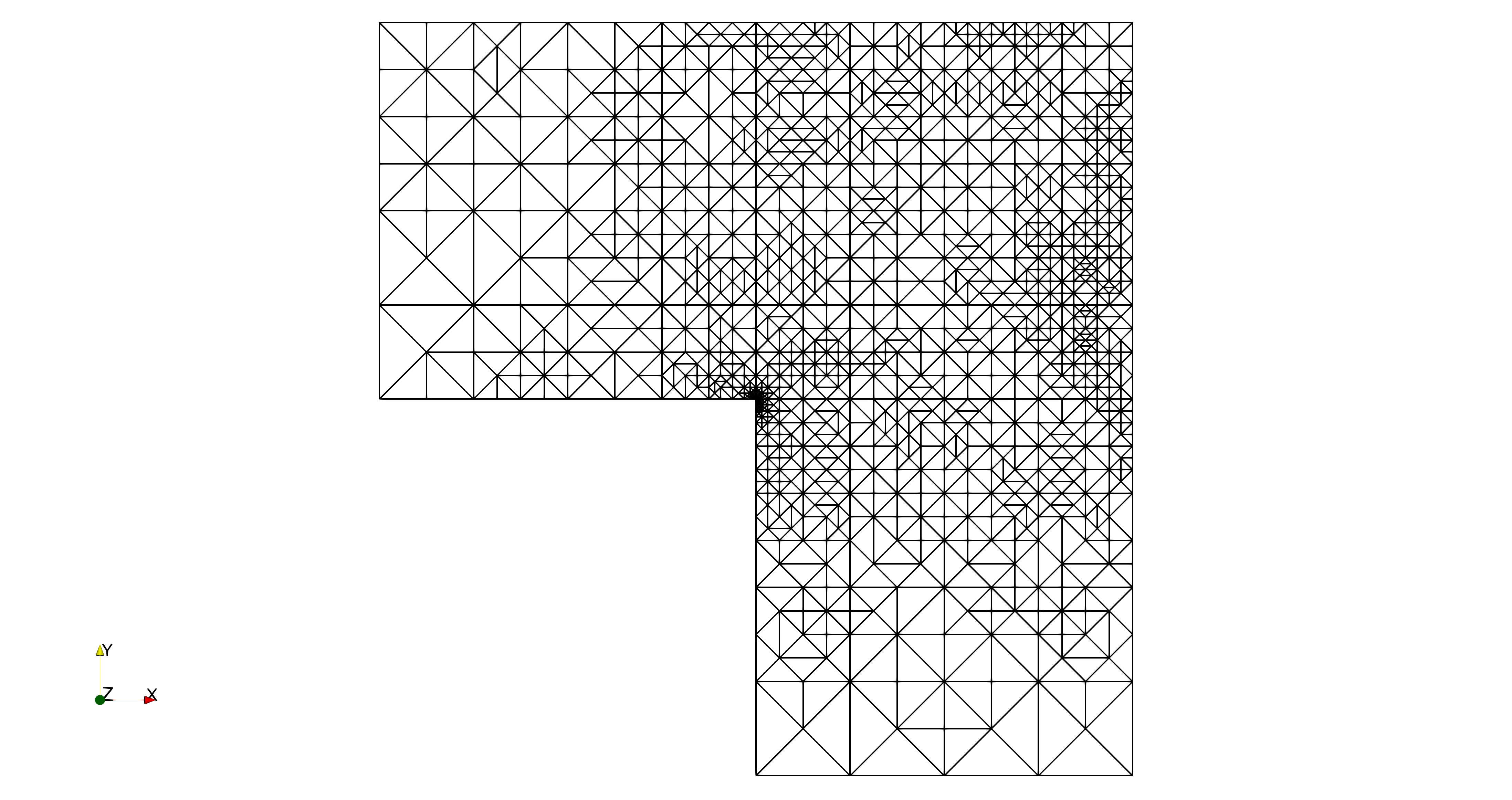}
	\end{minipage}
	\begin{minipage}{0.32\linewidth}
		\includegraphics[scale=0.05,trim= 37cm 1cm 37cm 1cm, clip]{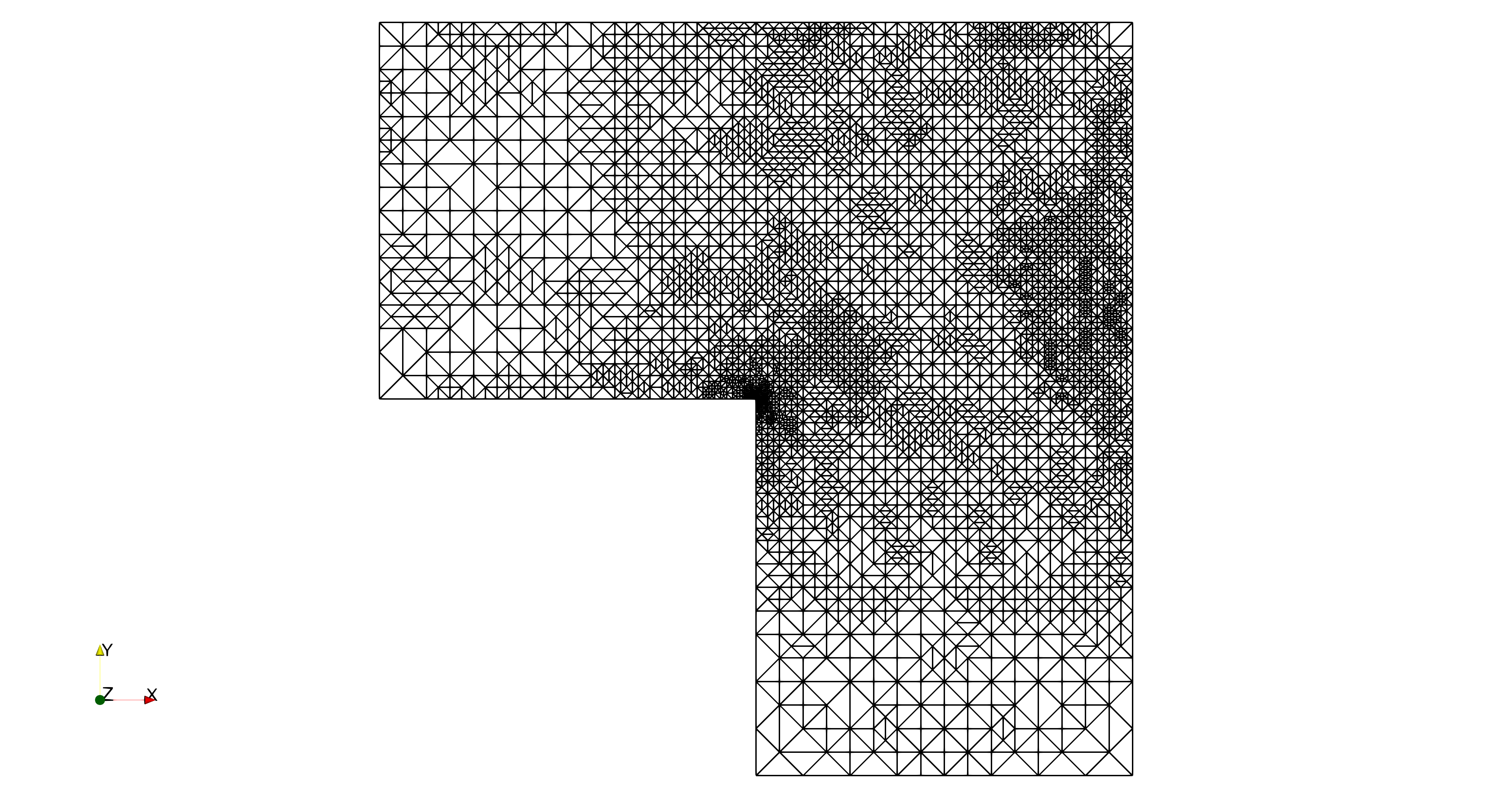}
	\end{minipage}
	\begin{minipage}{0.32\linewidth}
		\includegraphics[scale=0.05,trim= 37cm 1cm 37cm 1cm, clip]{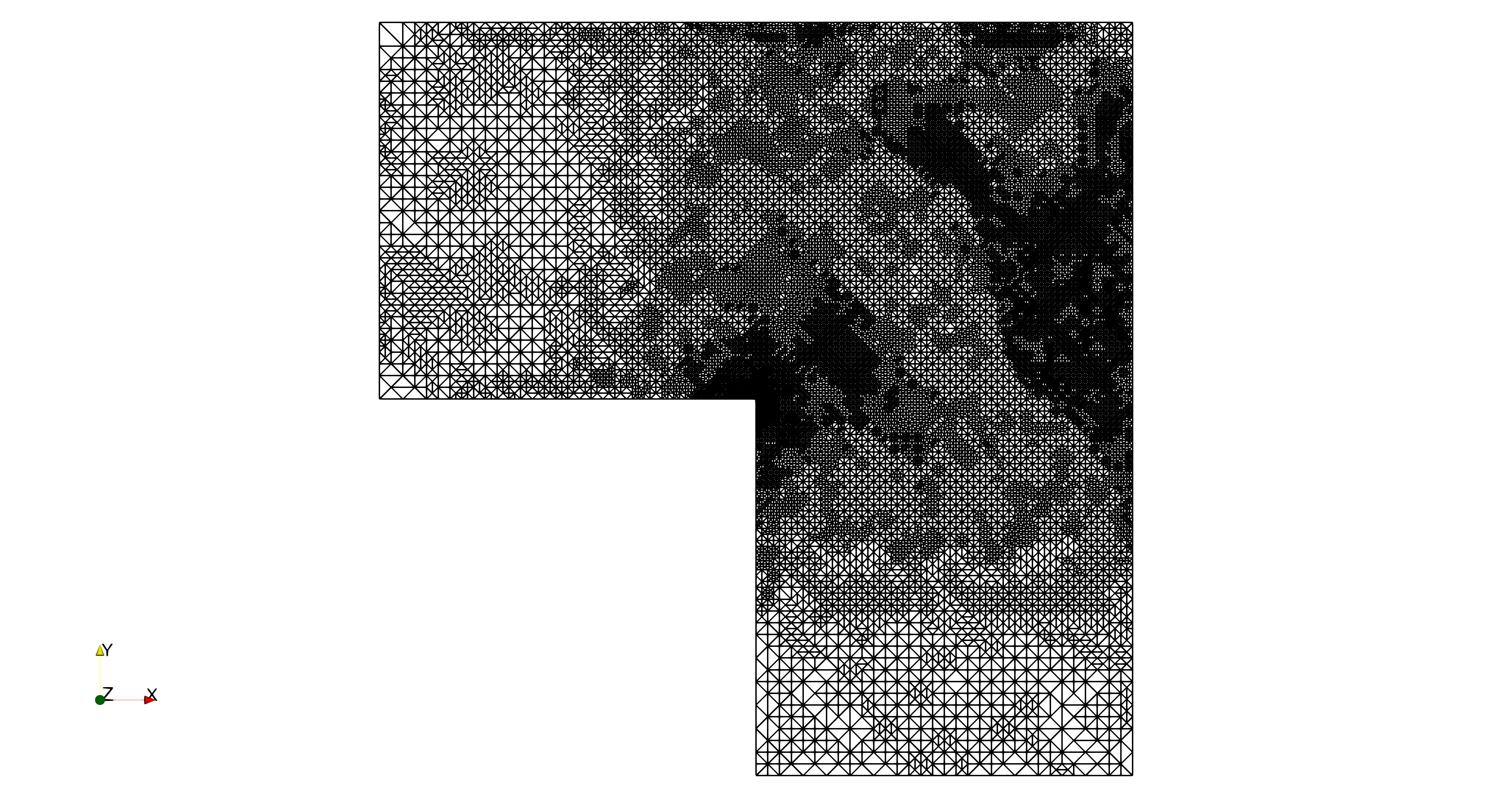}
	\end{minipage}\\
	\begin{minipage}{0.32\linewidth}
		\includegraphics[scale=0.05,trim= 37cm 1cm 37cm 1cm, clip]{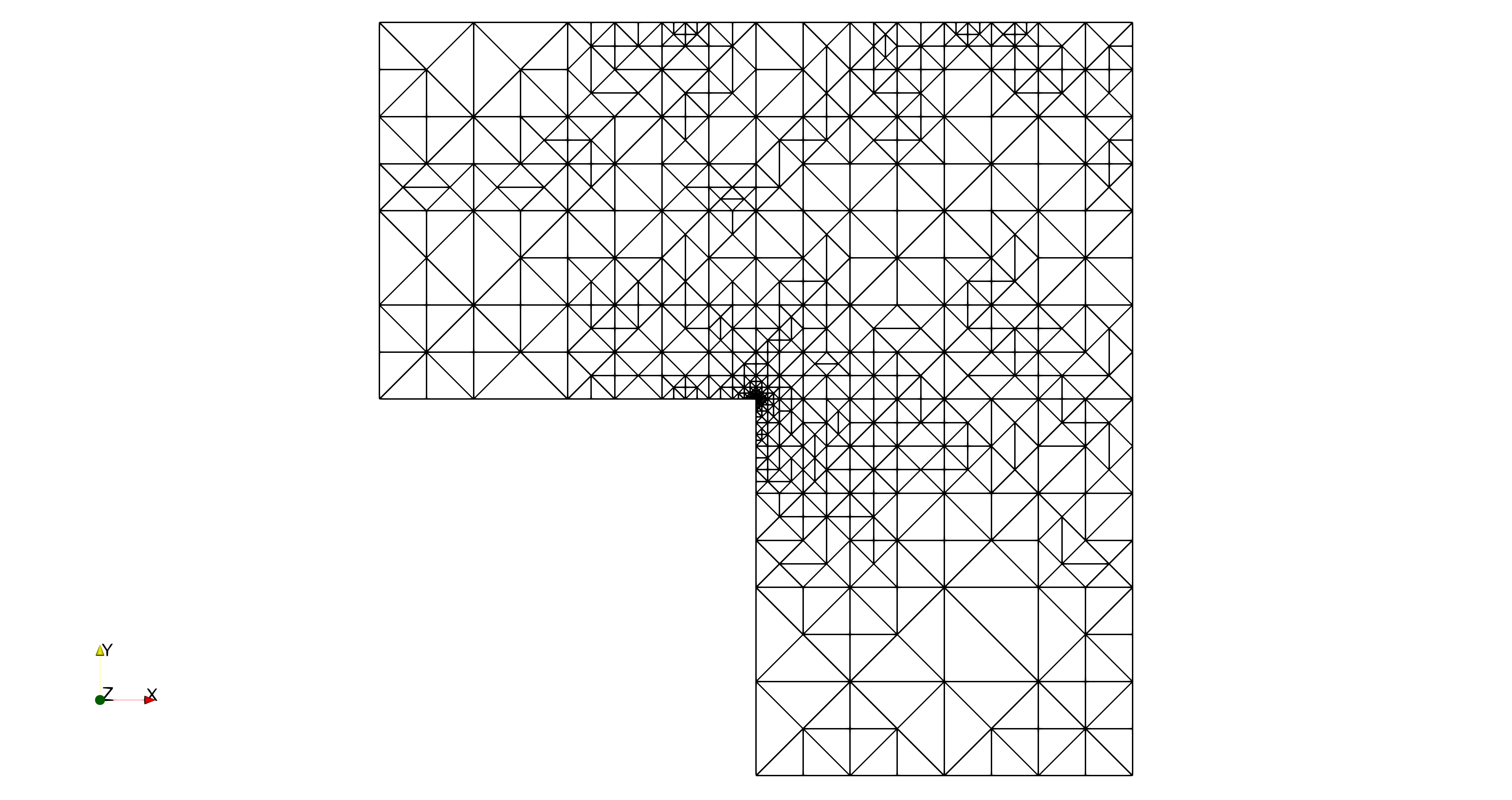}
	\end{minipage}
	\begin{minipage}{0.32\linewidth}
		\includegraphics[scale=0.05,trim= 37cm 1cm 37cm 1cm, clip]{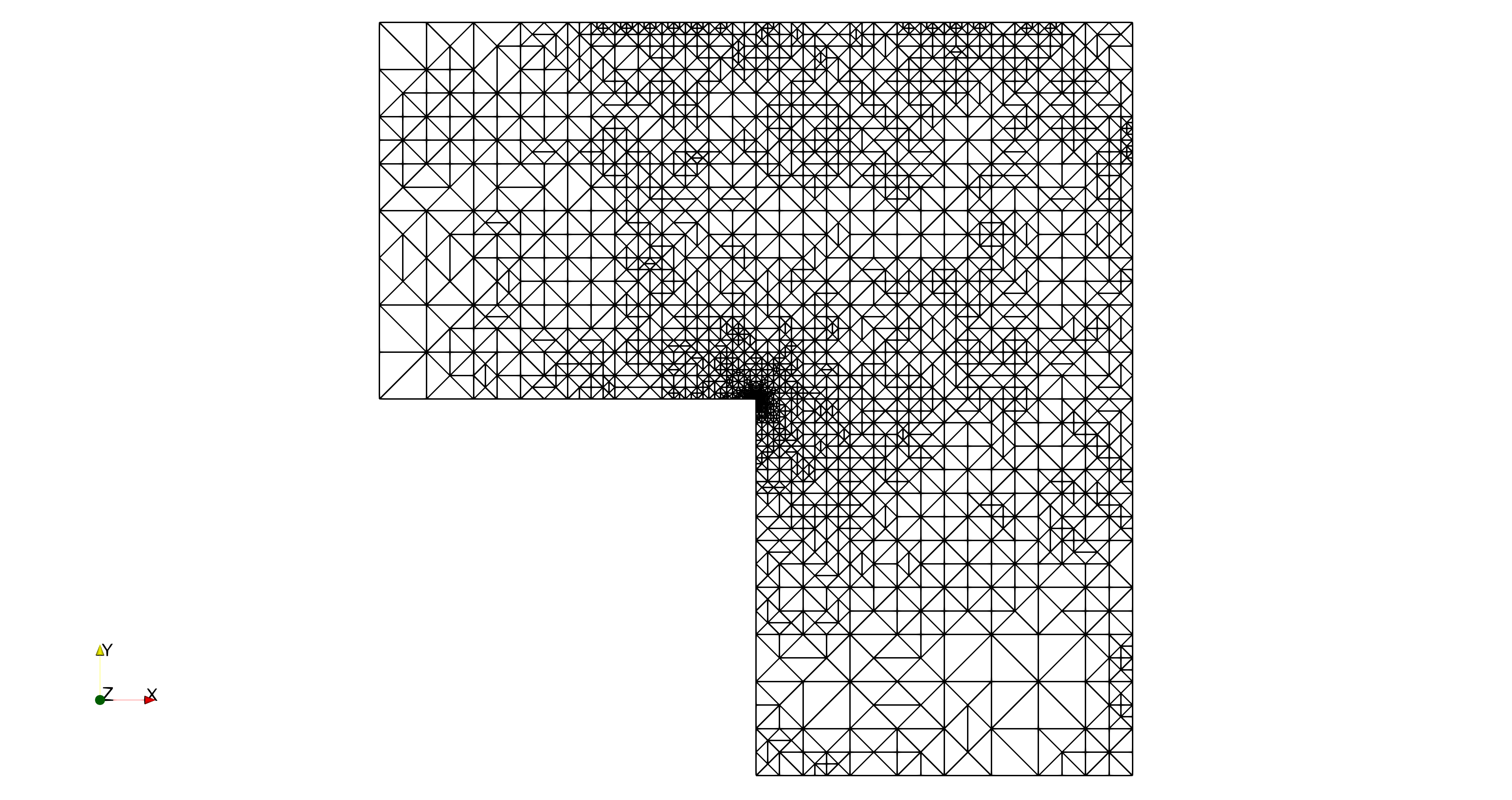}
	\end{minipage}
	\begin{minipage}{0.32\linewidth}
		\includegraphics[scale=0.05,trim= 37cm 1cm 37cm 1cm, clip]{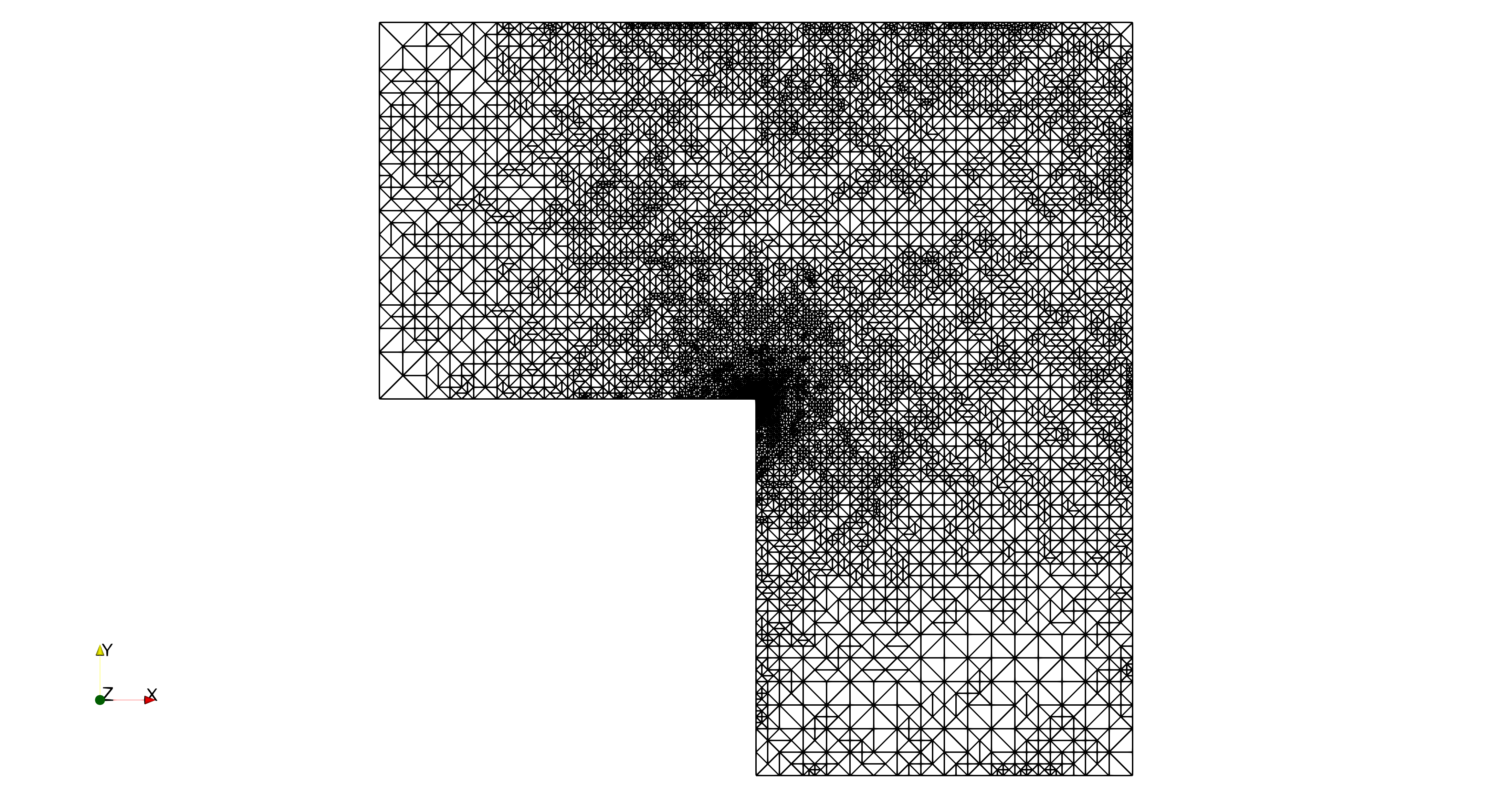}
	\end{minipage}\\
		\begin{minipage}{0.32\linewidth}
		\includegraphics[scale=0.05,trim= 32cm 1cm 32cm 1cm, clip]{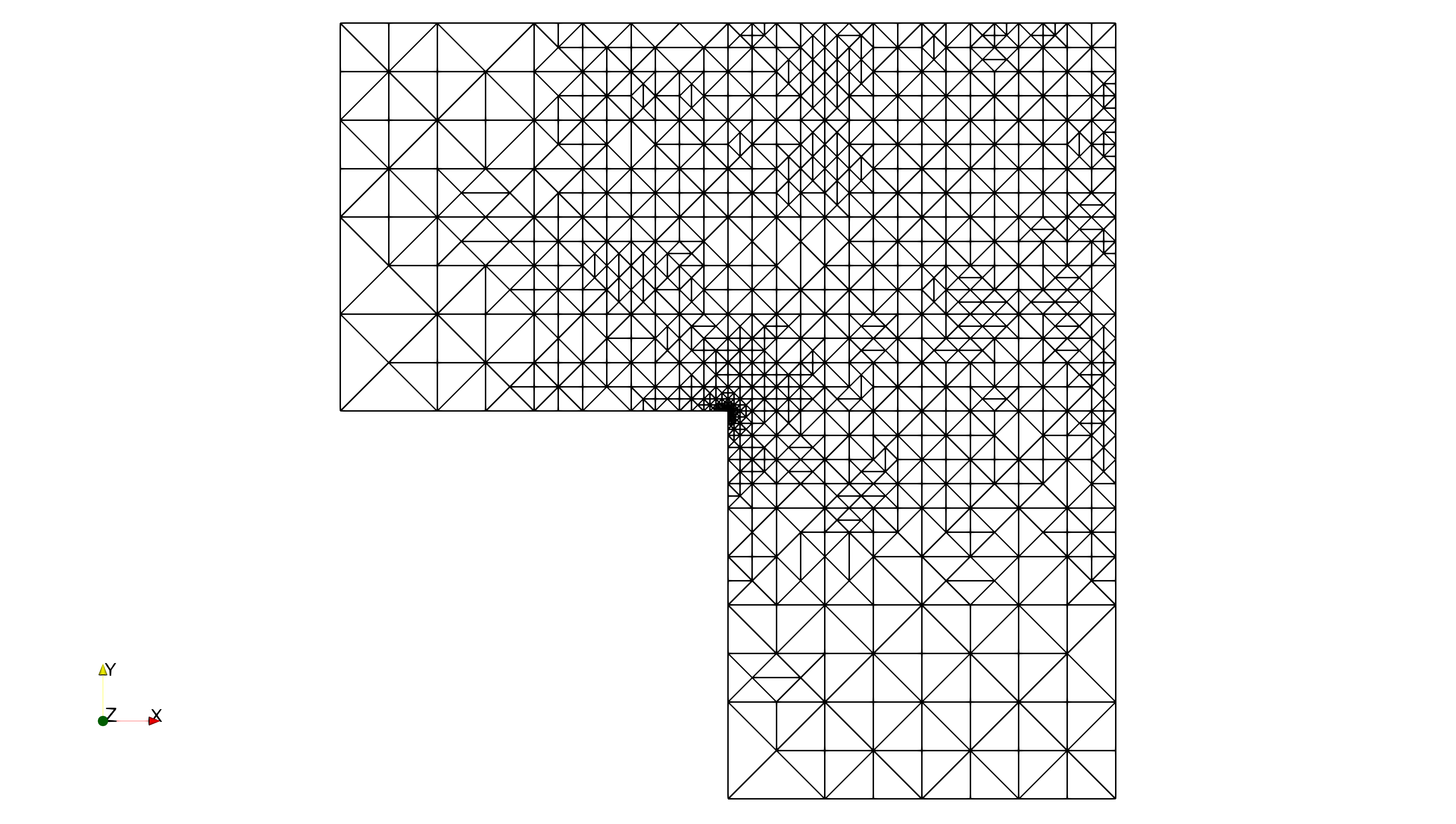}
	\end{minipage}
	\begin{minipage}{0.32\linewidth}
		\includegraphics[scale=0.05,trim= 32cm 1cm 32cm 1cm, clip]{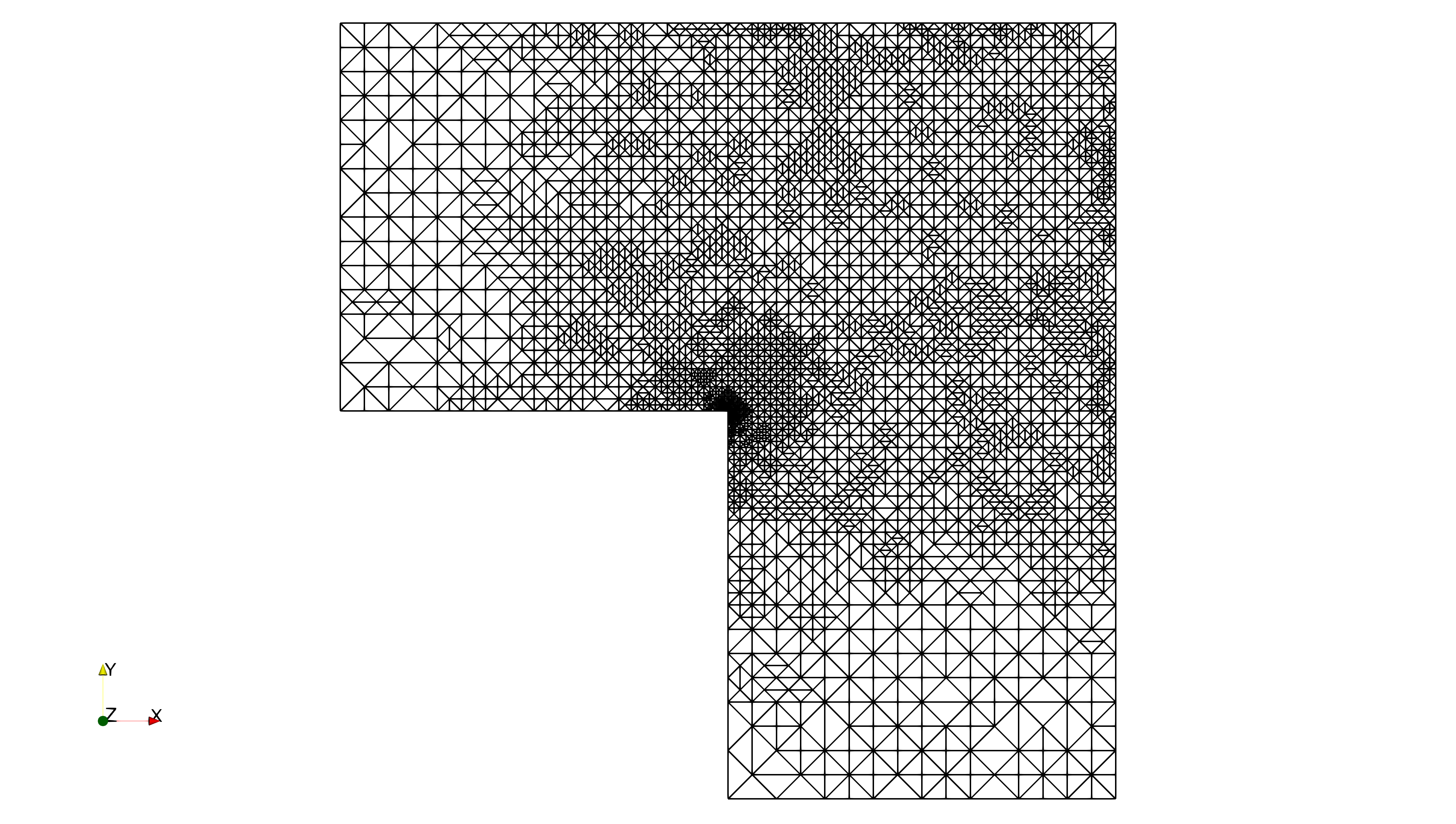}
	\end{minipage}
	\begin{minipage}{0.32\linewidth}
		\includegraphics[scale=0.05,trim= 32cm 1cm 32cm 1cm, clip]{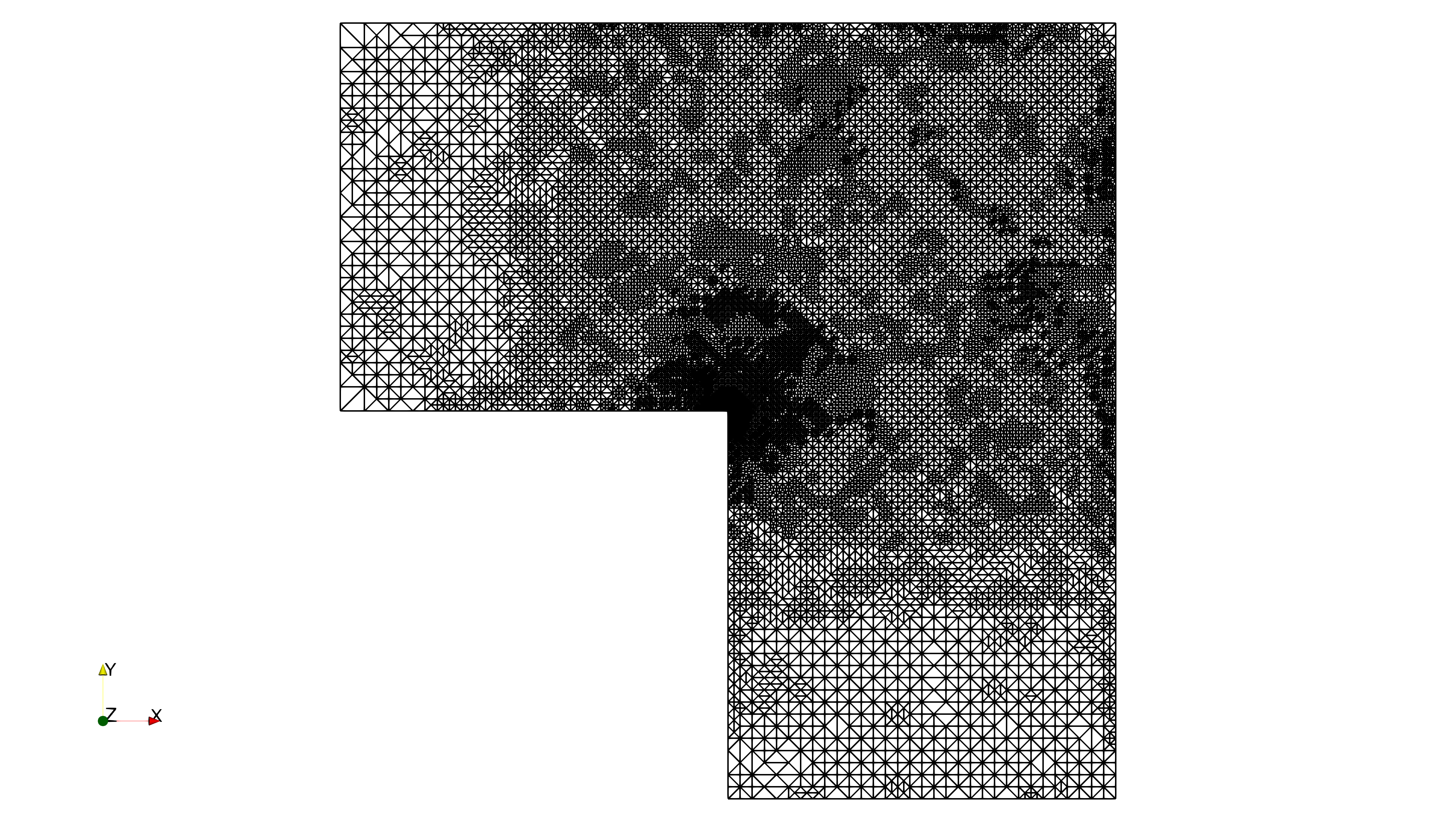}
	\end{minipage}\\
	\caption{Example \ref{subsec:lshape2D}. Comparison of intermediate meshes for the adaptive algorithm on the primal and dual  problem when computing the lowest order eigenvalue. Top: Intermediate meshes for iterations $i=7,11,15$ with $7076, 30191$ and  $131770$ degrees of freedom, respectively, using estimator $\eta$. Middle: Intermediate meshes for iterations $i=7,11,15$ with $3600, 14227$ and $52762$ degrees of freedom, respectively, using estimator $\eta^*$. Bottom: Intermediate meshes for iterations $i=7,11,15$ with $5995, 26074$ and $110671$ degrees of freedom, respectively, using estimator $\theta$.  }
	\label{fig:lshape2D-meshes}
\end{figure}
\begin{figure}[h!]
	\centering
	\includegraphics[scale=0.45]{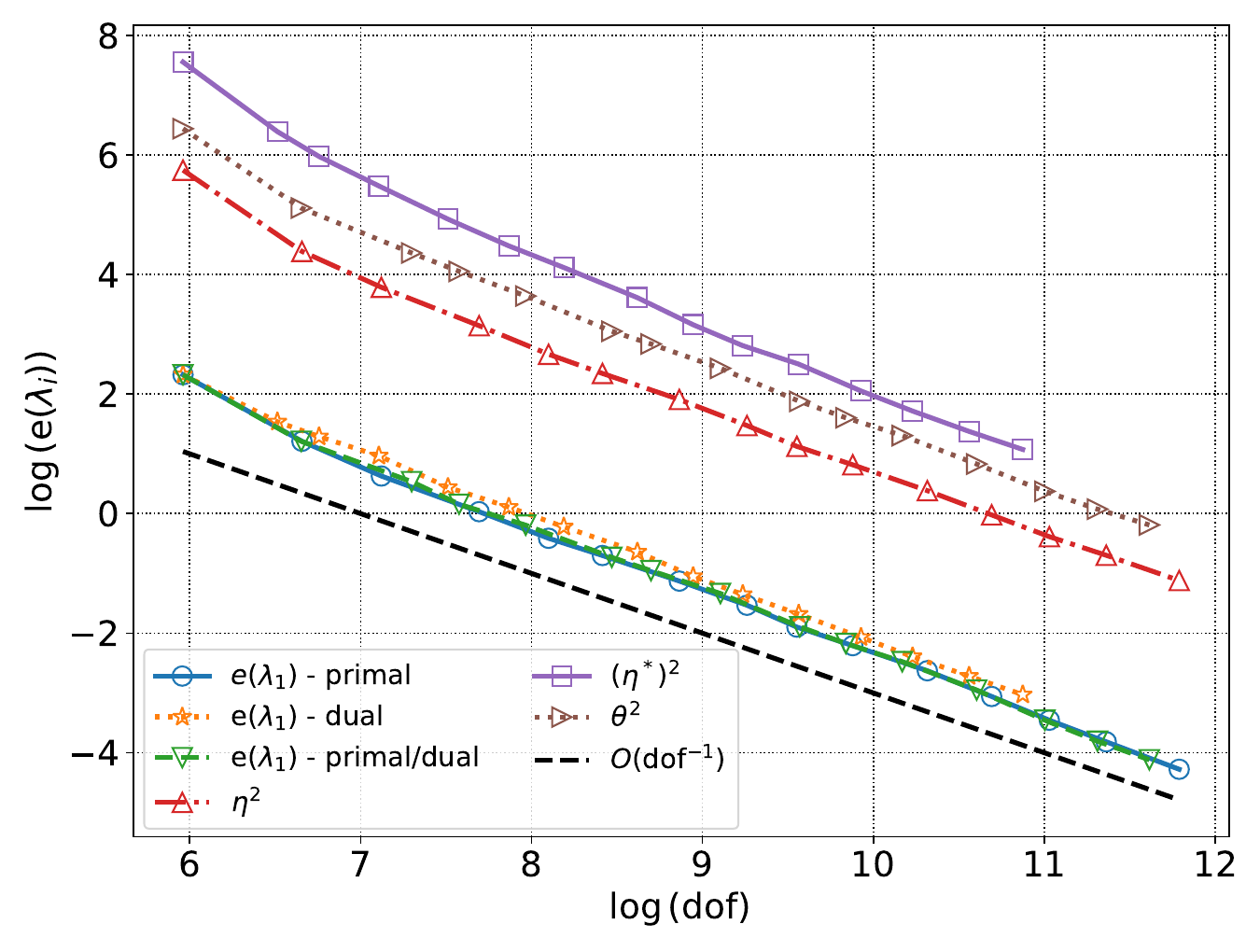}
	\caption{Example \ref{subsec:lshape2D}. Error curves obtained from the adaptive algorithm for the primal and dual problems compared with their corresponding estimators $\eta$ and $\eta_*$, respectively, and the optimal line $\mathcal{O}(\texttt{dof}^{-1})$.}
	\label{fig:lshape2D-error}
\end{figure}

\subsubsection{3D L-shaped domain}\label{subsec:lshape3D}
This final test presents the estimator performance when a three dimensional domain with a dihedral singularity is considered. The domain is an L-shaped domain given by
$$
\Omega:=(-1/2,1/2)\times(0,1)\times(-1/2,1/2)\backslash\big((0,1/2)\times(0,1)\times(0,1/2)\big).
$$
Note that this domain has a singularity along the line $(0,y,0)$, for $y\in[0,1]$, so the convergence with uniform meshes will be, at best, $\mathcal{O}(\texttt{dof}^{-0.44})$. The extrapolated eigenvalues taken as the exact solution for the primal and dual problem is given by 

$$\lambda_{1}=83.0386474910940$$

In Table \ref{table-lshape3D-first-eigenvalue-primal} we observe the estimator performance for 11 iterations. By observing the estimator, it notes that most of the contributions come from the volumetric integrals, followed by the jump terms. On each iterations, both contributions are bigger than the actual error, but the effectivity index remains bounded. Similar behavior is observed in the dual adaptive iterations, presented in Table \ref{table-lshape3D-first-eigenvalue-dual}, where a bigger volumetric contribution is observed. Similar to the two-dimensional L-shaped domain, the dual adaptive refinements tends to mark less elements, as observed in the final iteration. 

As a graphical evidence of the above, we present in Figure \ref{fig:lshape2D-meshes} two iteration steps, including the last one, of the adaptive algorithm for the primal and dual formulation. In both cases, the refinement is prioritized near the singular line. On Figure \ref{fig:lshape3D-error} we observe that the estimators contributions decays as $\mathcal{O}(\texttt{dof}^{-0.66})$, similar to the error curves. Finally, Figure \ref{fig:lshape3D-uh_ph} depicts the velocity field and the singular pressure contour plot for the lowest computed eigenvalue. Note that high pressure gradients are formed near the dihedral singularity.

\begin{table}[t!]
	\setlength{\tabcolsep}{3.5pt}
	\centering 
	\caption{Example \ref{subsec:lshape3D}. Comparison of the error history of the lowest computed primal eigenvalue with the global residual terms and $\eta^2$, on the three-dimensional L-shaped domain. Here, the convective velocity is set to be $\boldsymbol{\beta}=(0,0,1)^{\texttt{t}}$. }
	{\footnotesize\begin{tabular}{rccccccc}
			\hline\hline
			dof   &   $\err(\lambda_1)$  &   $\mathbf{R}$   &   $\mathbf{D}$  &  $\mathbf{J}$& $\eta^2$&  $\eff(\lambda_1)$  \\
			\hline 
			\hline
5073 &  $3.1589e+01  $ & $1.0244e+03  $   & $9.7776e+00 $  & $8.5554e+02 $  & $1.8897e+03 $ & $1.8615e-03 $ \\
16092 &  $1.2533e+01  $ & $3.4931e+02  $   & $2.9520e+00 $  & $3.0394e+02 $  & $6.5621e+02 $ & $2.8581e-03 $ \\
27001 &  $8.9541e+00  $ & $2.2699e+02  $   & $2.1150e+00 $  & $2.1536e+02 $  & $4.4446e+02 $ & $3.1296e-03 $ \\
75499 &  $4.3362e+00  $ & $1.2175e+02  $   & $1.0639e+00 $  & $1.1217e+02 $  & $2.3499e+02 $ & $3.0598e-03 $ \\
122781 &  $2.9899e+00  $ & $8.1738e+01  $   & $7.5005e-01 $  & $7.8883e+01 $  & $1.6137e+02 $ & $3.3887e-03 $ \\
129185 &  $2.8751e+00  $ & $7.7128e+01  $   & $7.2486e-01 $  & $7.6383e+01 $  & $1.5424e+02 $ & $3.4754e-03 $ \\
284817 &  $1.5336e+00  $ & $4.6190e+01  $   & $4.4394e-01 $  & $4.7043e+01 $  & $9.3677e+01 $ & $3.7368e-03 $ \\
309110 &  $1.4246e+00  $ & $4.2328e+01  $   & $4.2070e-01 $  & $4.4609e+01 $  & $8.7357e+01 $ & $3.9128e-03 $ \\
533602 &  $8.1289e-01  $ & $2.8754e+01  $   & $3.0096e-01 $  & $3.1060e+01 $  & $6.0115e+01 $ & $3.9862e-03 $ \\
559552 &  $7.7423e-01  $ & $2.7462e+01  $   & $2.9382e-01 $  & $3.0240e+01 $  & $5.7996e+01 $ & $4.0587e-03 $ \\
1157348 &  $2.3114e-01  $ & $1.5773e+01  $   & $1.7570e-01 $  & $1.8017e+01 $  & $3.3966e+01 $ & $3.8112e-03 $ \\
			\hline
			\hline
	\end{tabular}}
	\smallskip
	\label{table-lshape3D-first-eigenvalue-primal}
\end{table}

\begin{table}[t!]
	\setlength{\tabcolsep}{3.5pt}
	\centering 
	\caption{Example \ref{subsec:lshape3D}. Comparison of the error history of the lowest computed dual eigenvalue with the global residual terms and $(\eta^*)^2$, on the three-dimensional L-shaped domain. Here, the convective velocity is set to be $\boldsymbol{\beta}=(0,0,1)^{\texttt{t}}$. }
	{\footnotesize\begin{tabular}{rccccccc}
			\hline\hline
			dof   &   $\err_*(\lambda_1)$  &   $\mathbf{R}^*$   &   $\mathbf{D}^*$  &  $\mathbf{J}^*$& $(\eta^*)^2$&  $\eff_*(\lambda_1)$  \\
			\hline 
			\hline
			5073 &  $3.1956e+01  $   & $1.7098e+04  $   & $9.4481e+00 $  & $5.9935e+01 $  & $1.7167e+04 $ & $1.8401e-03$ \\
			11103 &  $1.8014e+01  $   & $6.2751e+03  $   & $4.4214e+00 $  & $2.3181e+01 $  & $6.3027e+03 $ & $1.9885e-03$ \\
			28238 &  $9.9352e+00  $   & $3.1636e+03  $   & $2.2270e+00 $  & $8.8048e+00 $  & $3.1746e+03 $ & $2.8205e-03$ \\
			65815 &  $5.5110e+00  $   & $1.7962e+03  $   & $1.1901e+00 $  & $3.6815e+00 $  & $1.8011e+03 $ & $2.4075e-03$ \\
			85009 &  $4.6080e+00  $   & $1.3560e+03  $   & $9.8976e-01 $  & $2.7953e+00 $  & $1.3598e+03 $ & $2.1988e-03$ \\
			159324 &  $2.9813e+00  $   & $8.5567e+02  $   & $6.4967e-01 $  & $1.5080e+00 $  & $8.5782e+02 $ & $3.3516e-03$ \\
			216321 &  $2.4336e+00  $   & $6.4960e+02  $   & $5.2799e-01 $  & $1.1266e+00 $  & $6.5126e+02 $ & $2.3548e-03$ \\
			291129 &  $2.0428e+00  $   & $5.2077e+02  $   & $4.5163e-01 $  & $8.8071e-01 $  & $5.2210e+02 $ & $2.7287e-03$ \\
			411620 &  $1.5973e+00  $   & $3.9972e+02  $   & $3.6153e-01 $  & $6.3506e-01 $  & $4.0072e+02 $ & $2.0286e-03$ \\
			493495 &  $1.3819e+00  $   & $3.3963e+02  $   & $3.2181e-01 $  & $5.2684e-01 $  & $3.4048e+02 $ & $2.2739e-03$ \\
			831334 &  $8.8359e-01  $   & $2.3131e+02  $   & $2.2585e-01 $  & $3.0183e-01 $  & $2.3184e+02 $ & $9.9700e-04$ \\
			\hline
			\hline
	\end{tabular}}
	\smallskip
	\label{table-lshape3D-first-eigenvalue-dual}
\end{table}

\begin{figure}[!h]
	\centering
	\begin{minipage}{0.48\linewidth}\centering
		\includegraphics[scale=0.065,trim= 37cm 1cm 37cm 1cm, clip]{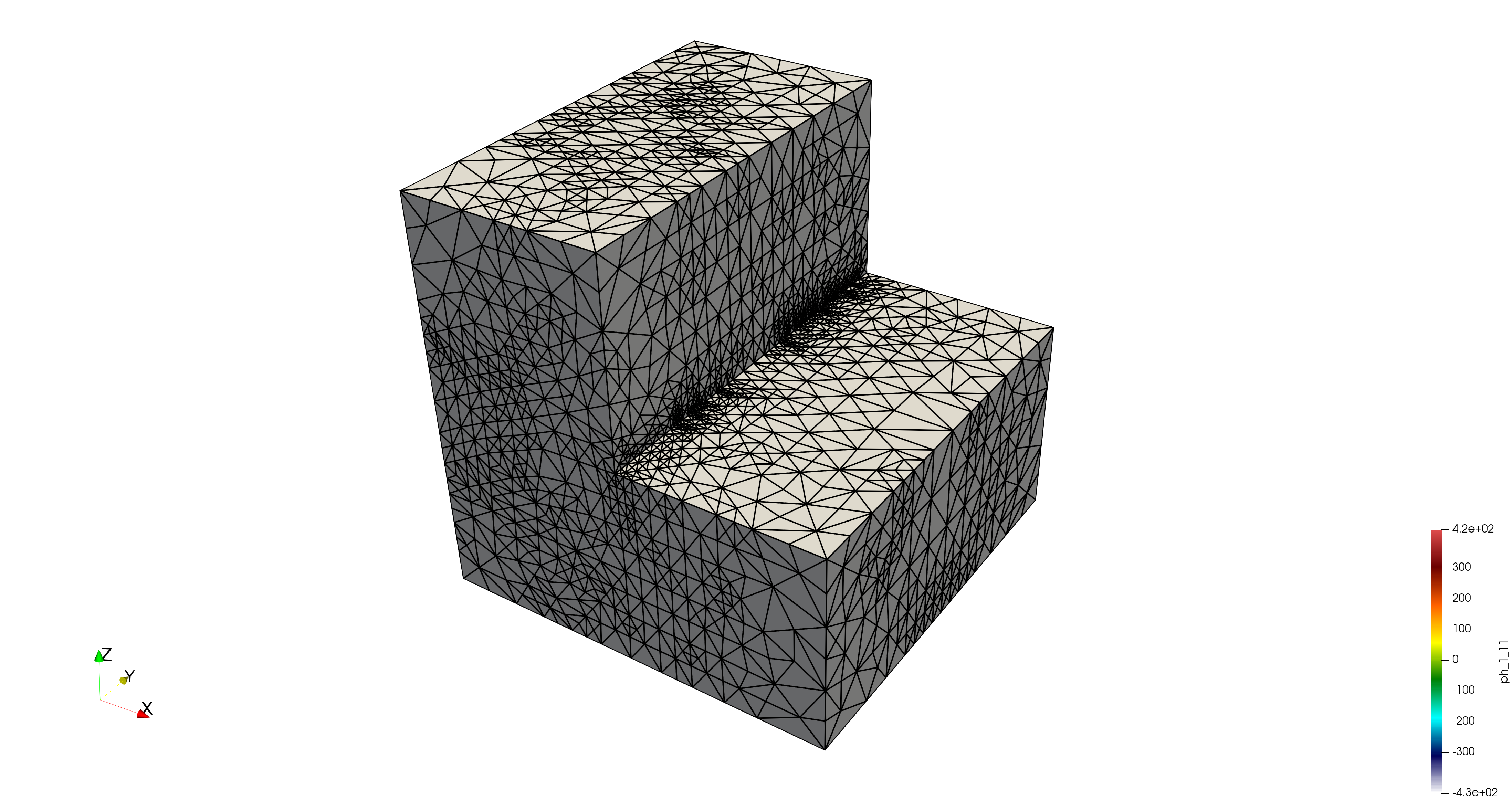}
	\end{minipage}
	\begin{minipage}{0.48\linewidth}\centering
		\includegraphics[scale=0.065,trim= 37cm 1cm 37cm 1cm, clip]{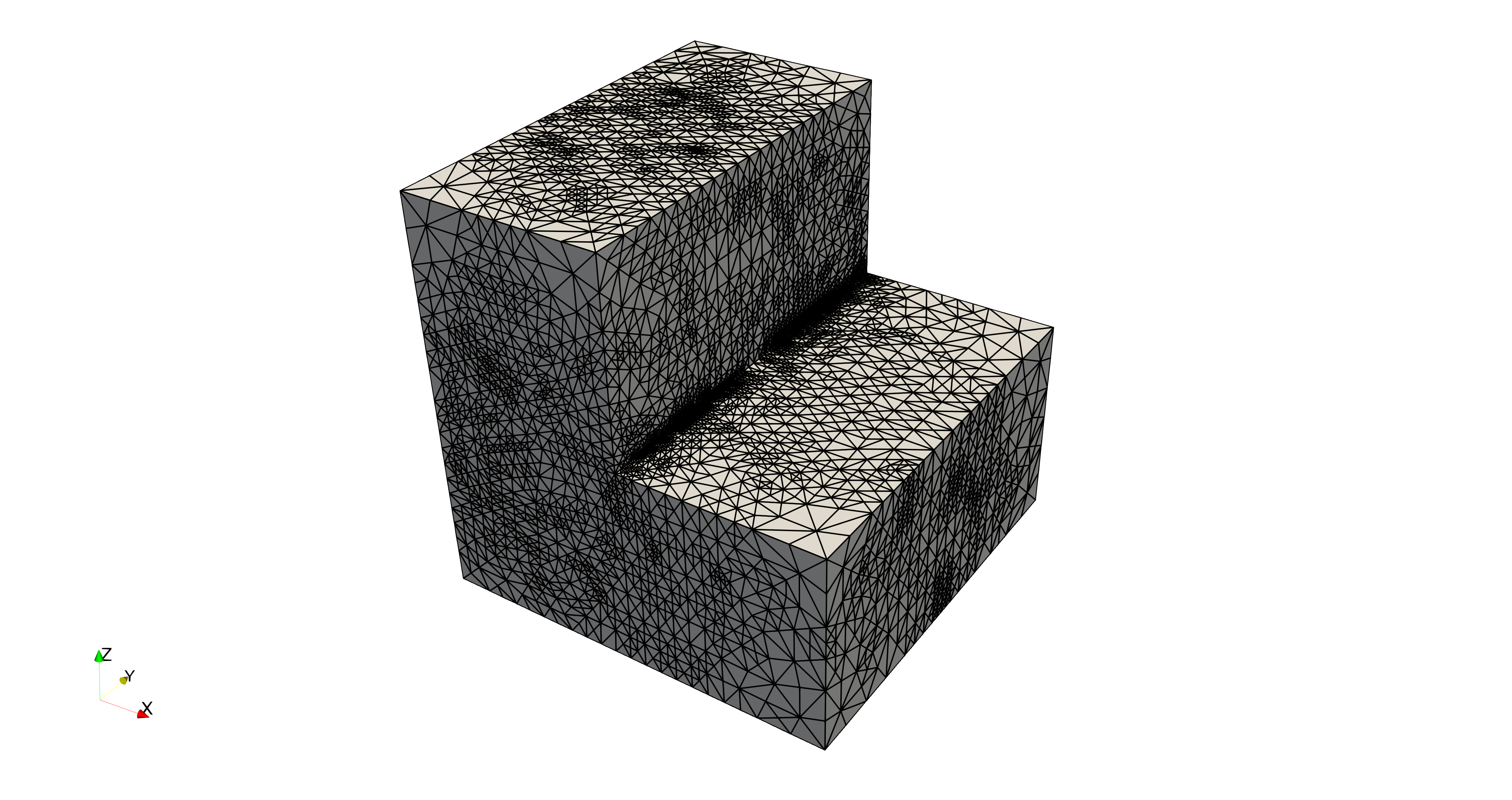}
	\end{minipage}\\
	\begin{minipage}{0.48\linewidth}\centering
		\includegraphics[scale=0.065,trim= 37cm 1cm 37cm 1cm, clip]{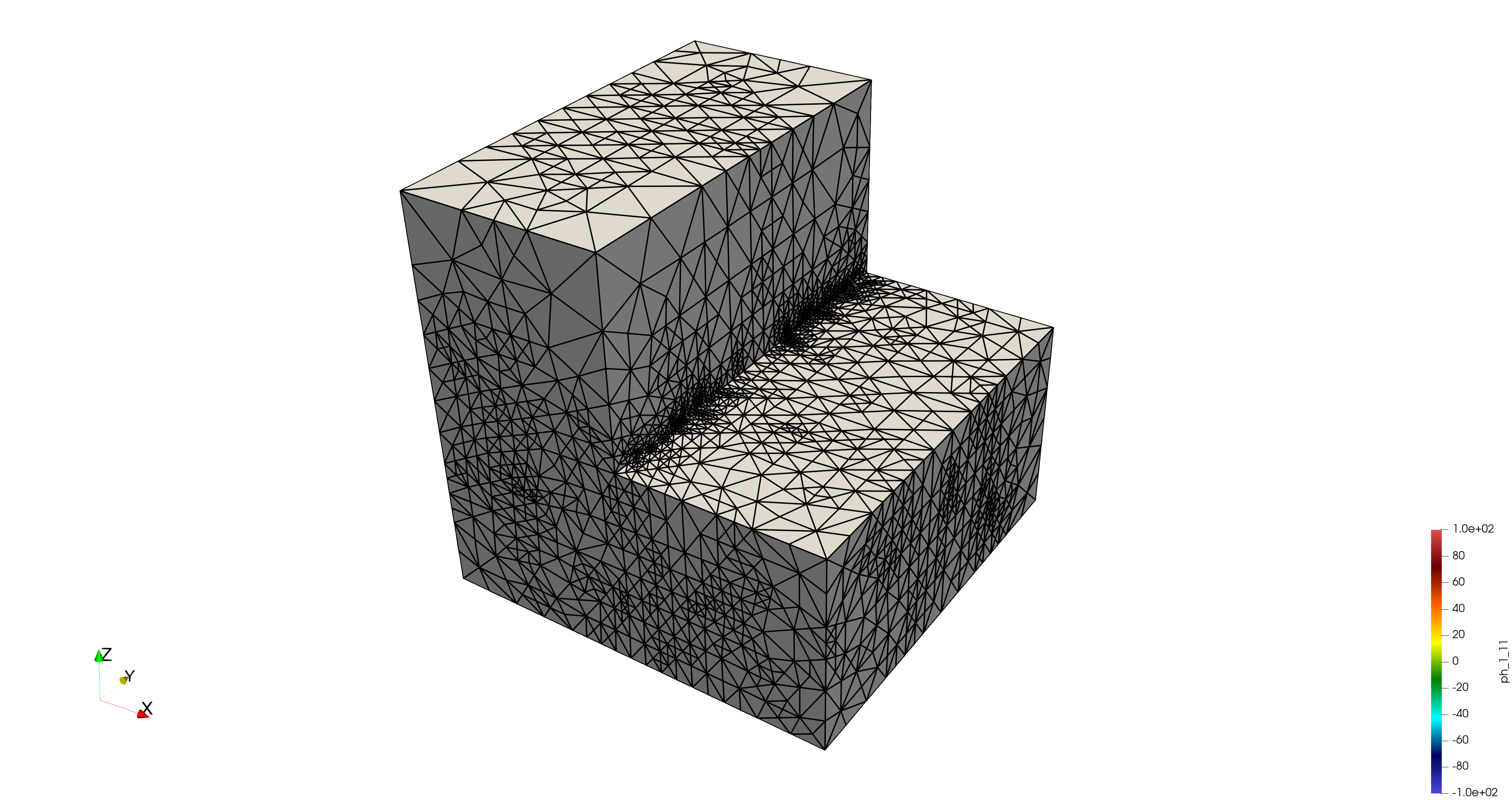}
	\end{minipage}
	\begin{minipage}{0.48\linewidth}\centering
		\includegraphics[scale=0.065,trim= 37cm 1cm 37cm 1cm, clip]{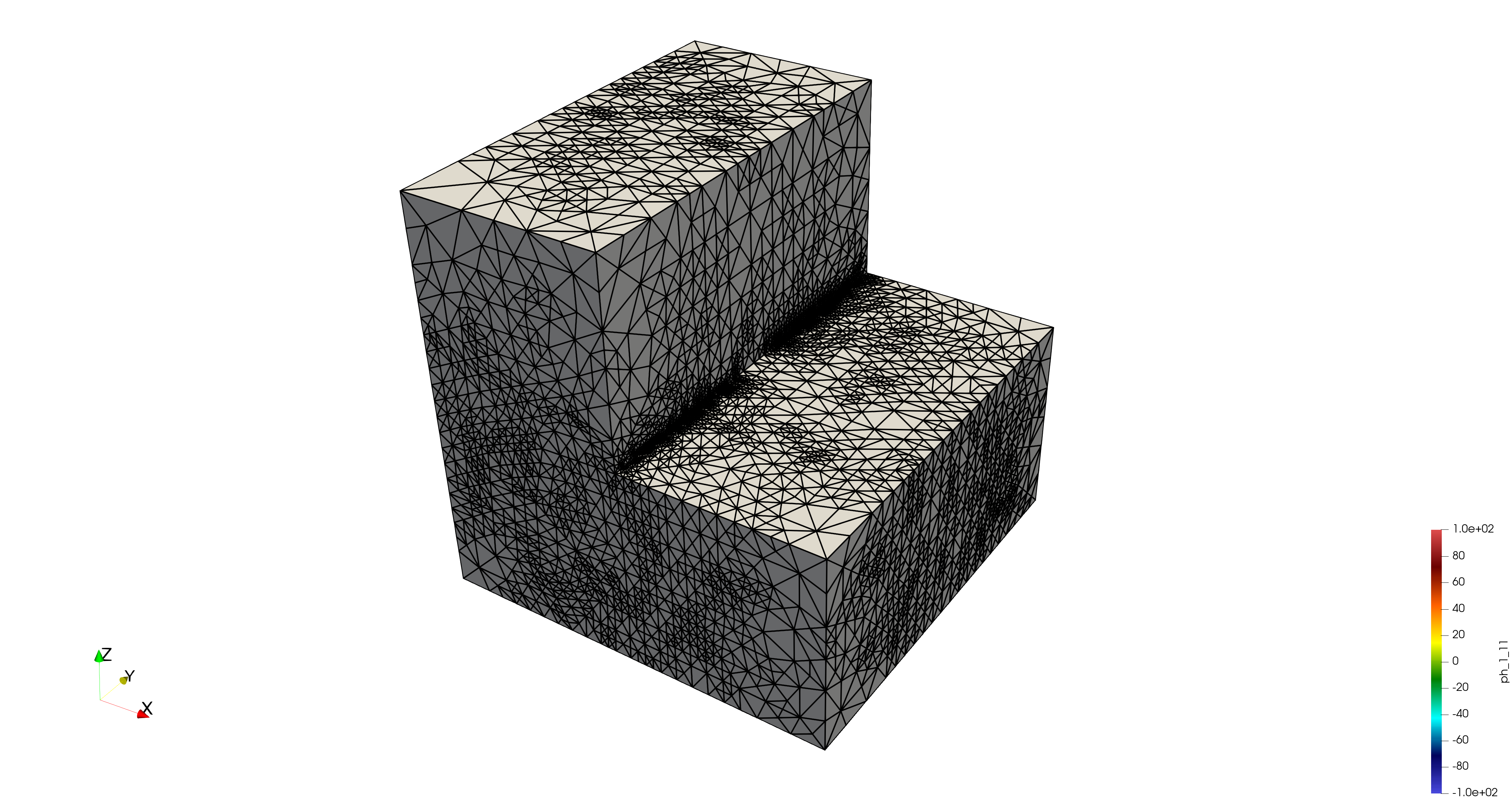}
	\end{minipage}\\
	\begin{minipage}{0.48\linewidth}\centering
		\includegraphics[scale=0.065,trim= 37cm 1cm 37cm 1cm, clip]{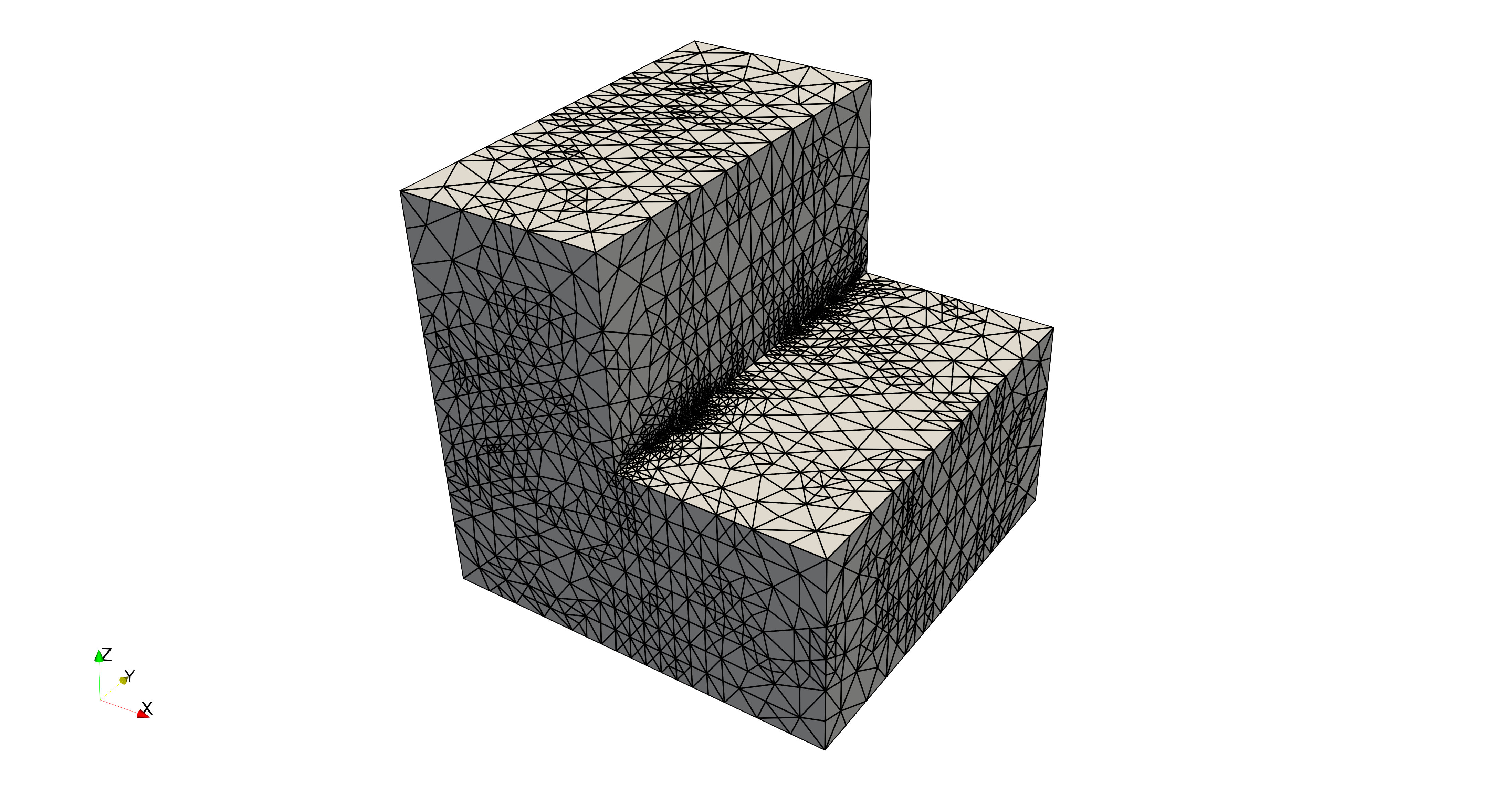}
	\end{minipage}
	\begin{minipage}{0.48\linewidth}\centering
		\includegraphics[scale=0.065,trim= 37cm 1cm 37cm 1cm, clip]{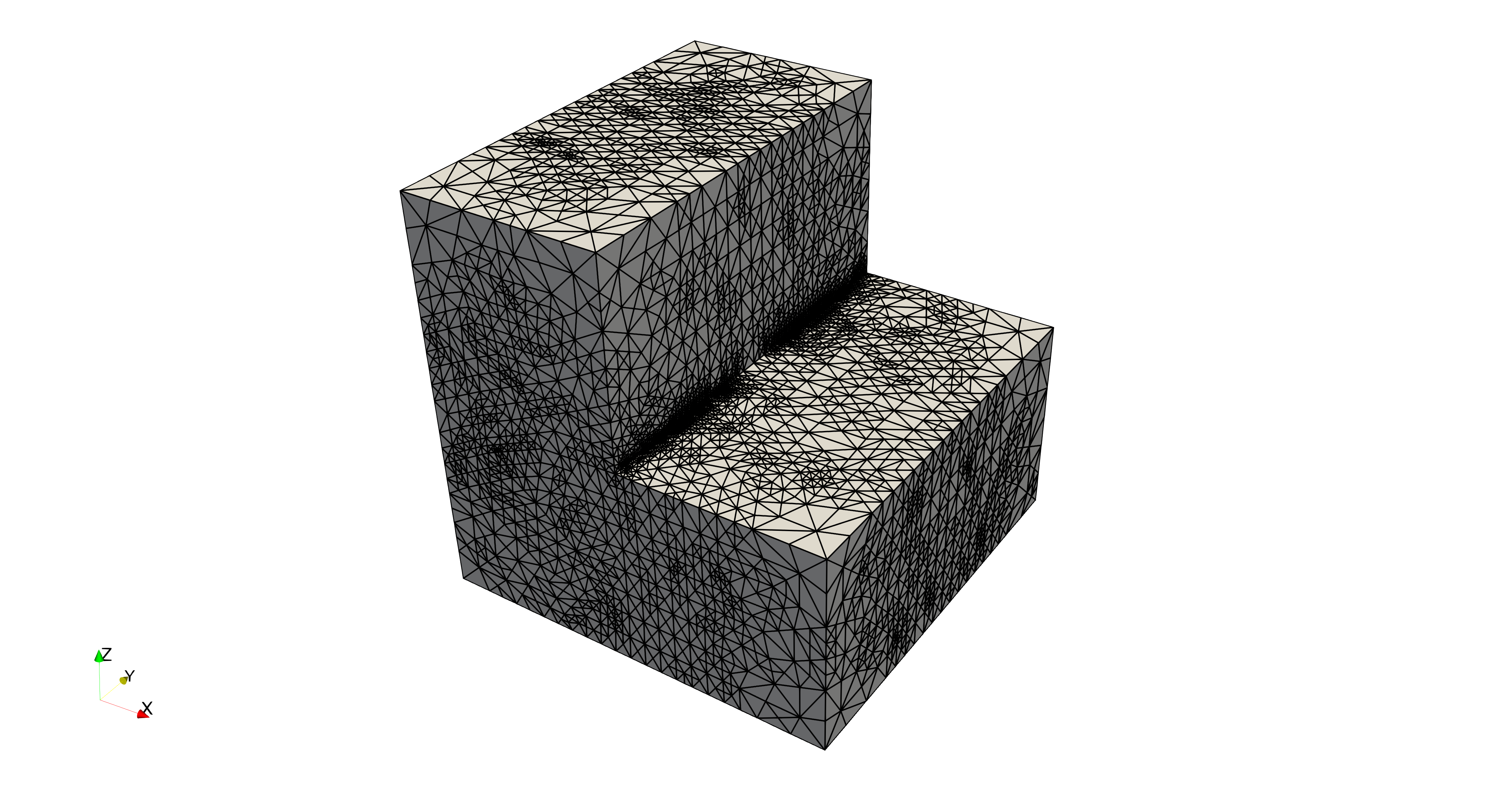}
	\end{minipage}\\
	\caption{Example \ref{subsec:lshape3D}. Comparison of intermediate meshes for the adaptive algorithm on the primal and dual  problem when computing the lowest order eigenvalue in the 3D L-shaped domain. Top: Intermediate meshes for iterations $i=8,11,$ with $309110$ and  $1157348$ degrees of freedom, respectively, using estimator $\eta$. Middle: Intermediate meshes for iterations $i=8,11,$ with $159324$ and $831334$ degrees of freedom, respectively, using estimator $\eta^*$. Bottom: Intermediate meshes for iterations $i=8,11,$ with $374564$ and $921282$ degrees of freedom, respectively, using estimator $\theta$.}
	\label{fig:lshape3D-meshes}
\end{figure}
\begin{figure}[h!]
	\centering
	\includegraphics[scale=0.45]{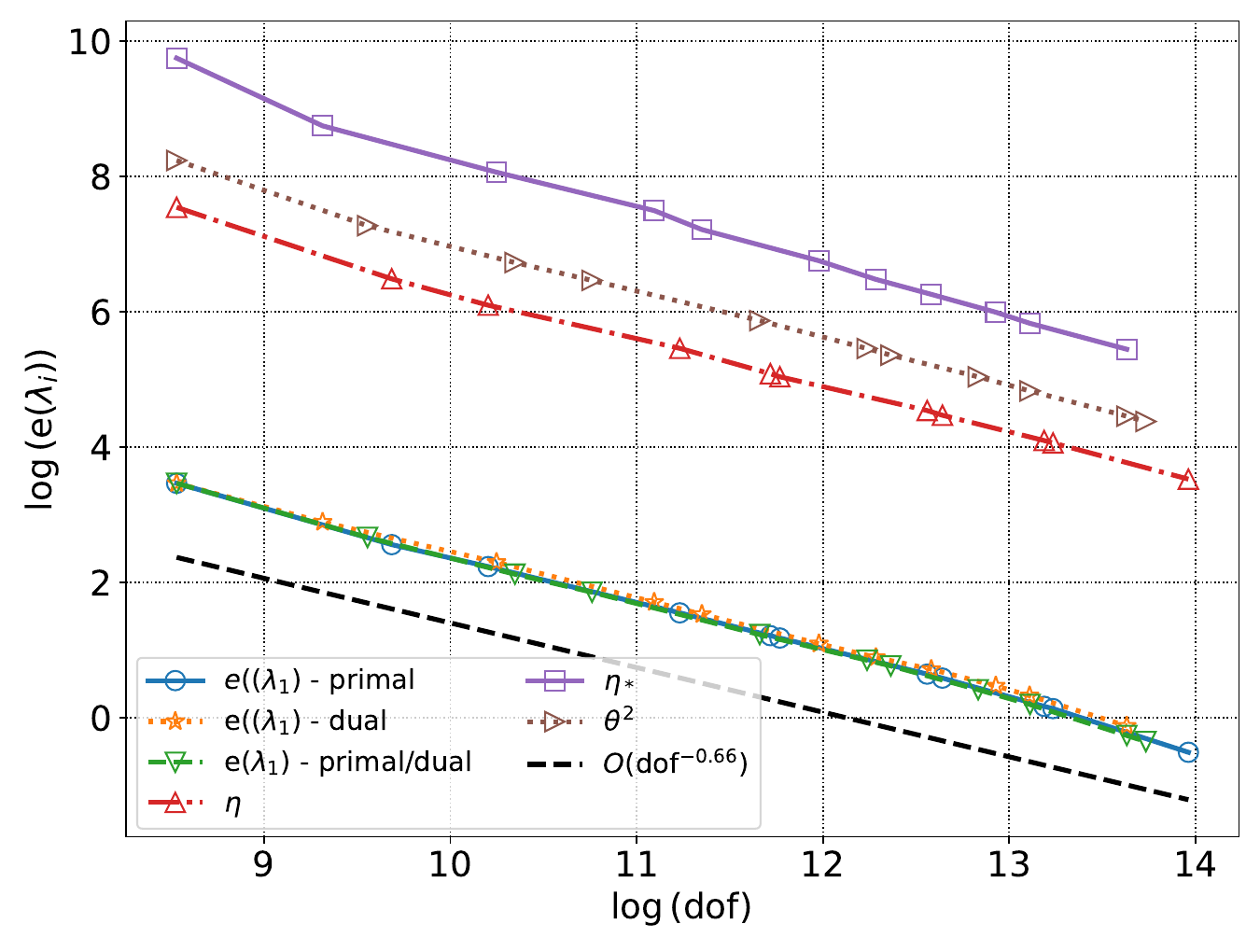}
	\caption{Example \ref{subsec:lshape3D}. Error curves obtained from the adaptive algorithm for the primal and dual problems compared with their corresponding estimators $\eta$ and $\eta_*$, respectively, and the optimal line $\mathcal{O}(\texttt{dof}^{-0.66})$.}
	\label{fig:lshape3D-error}
\end{figure}
\begin{figure}[!h]
	\centering
	\begin{minipage}{0.48\linewidth}\centering
		\includegraphics[scale=0.065,trim= 36cm 1cm 37cm 1cm, clip]{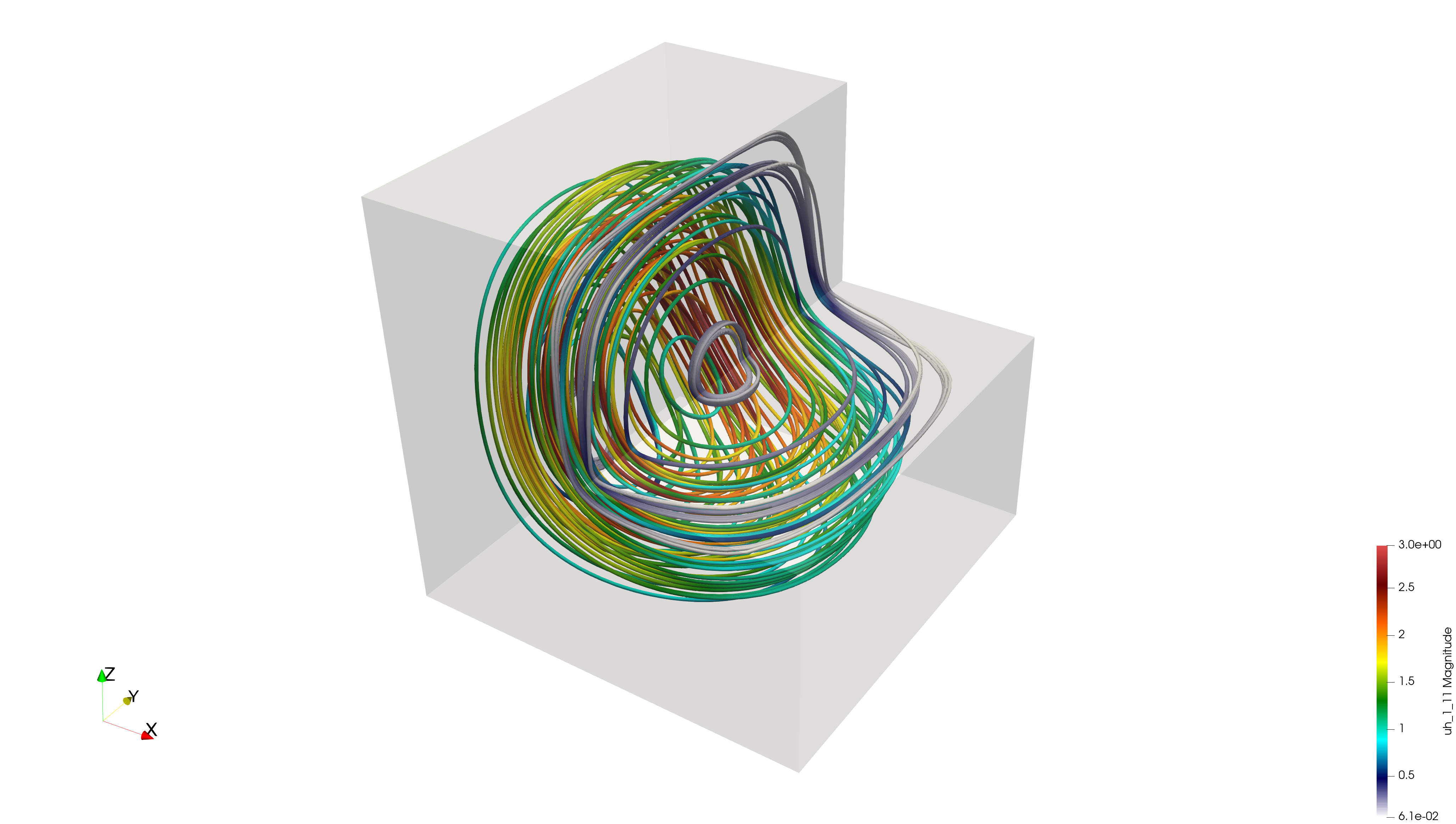}
	\end{minipage}
	\begin{minipage}{0.48\linewidth}\centering
		\includegraphics[scale=0.065,trim= 37cm 1cm 37cm 1cm, clip]{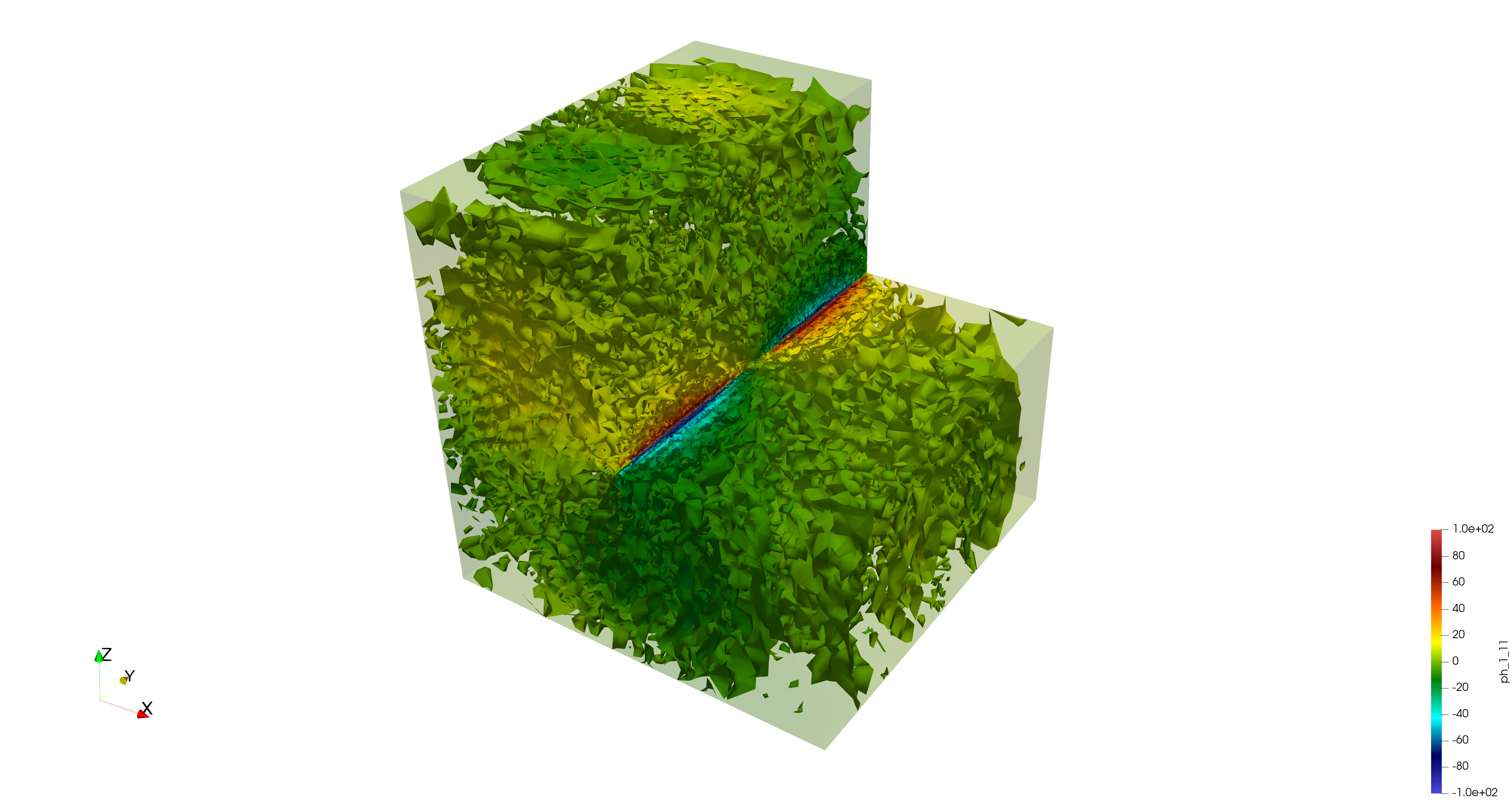}
	\end{minipage}\\
	\caption{Example \ref{subsec:lshape3D}. Velocity field for the lowest order computed eigenmodes for the primal and dual problems (left), together with the corresponding singular pressure contour surface plot (right).}
	\label{fig:lshape3D-uh_ph}
\end{figure}


\bibliographystyle{siamplain}
\bibliography{oseen-eigenvalue}
\end{document}

%% file: ex_shared.tex
\usepackage{amsmath,amssymb,amsfonts,latexsym,stmaryrd}
\usepackage{graphicx,float,epsfig} 
\usepackage{mathrsfs} 
\usepackage{color}
\usepackage{setspace} 
\usepackage{lipsum}
\usepackage{graphicx,float}
\usepackage{color}
\usepackage{epstopdf}
\usepackage{multirow}
\usepackage{svg}
\usepackage{url}
\usepackage{diagbox}
\ifpdf
  \DeclareGraphicsExtensions{.eps,.pdf,.png,.jpg}
\else
  \DeclareGraphicsExtensions{.eps}
\fi
\newcommand\err{\texttt{err}}
\newcommand\eff{\texttt{eff}}

\def\ds{\displaystyle}

\def\O{\Omega}

\def\l{\lambda}

\renewcommand\sp{\mathop{\mathrm{Sp}}\nolimits}

\newcommand{\jump}[1]{\llbracket #1 \rrbracket}

\usepackage{booktabs}
\usepackage{float}

\newcommand\bu{\boldsymbol{u}}
\newcommand\bv{\boldsymbol{v}}
\newcommand\bw{\boldsymbol{w}}

\newcommand\br{\boldsymbol{r}}
\newcommand\bs{\boldsymbol{s}}
\newcommand\bn{\boldsymbol{n}}

\def\hdel{\widehat{\delta}}

\newcommand\bF{\boldsymbol{f}}

\newcommand\bT{\boldsymbol{T}}

\def\CT{{\mathcal T}}

\newcommand\btau{\boldsymbol{\tau}}

\newcommand\eu{\texttt{e}_{\boldsymbol{u}}}
\newcommand\ep{\texttt{e}_{p}}

\renewcommand\H{\mathrm{H}}

\renewcommand\L{\mathrm{L}}

\renewcommand\O{\Omega}

\renewcommand\div{\mathop{\mathrm{div}}\nolimits}

\renewcommand\sp{\mathop{\mathrm{sp}}\nolimits}

\newcommand{\vertiii}[1]{{\left\vert\kern-0.25ex\left\vert\kern-0.25ex\left\vert #1 
    \right\vert\kern-0.25ex\right\vert\kern-0.25ex\right\vert}}

\newsiamremark{remark}{Remark}
\newsiamremark{hypothesis}{Hypothesis}
\crefname{hypothesis}{Hypothesis}{Hypotheses}
\newsiamthm{problem}{Problem}
\newsiamthm{claim}{Claim}

\headers{Finite Element Analysis of the Oseen eigenvalue problem}{Felipe Lepe,  Gonzalo Rivera and Jesus Vellojin}

\title{Finite Element Analysis of the Oseen eigenvalue problem\thanks{Submitted to the editors DATE.
\funding{The first author was partially supported by
	DIUBB through project 2120173 GI/C Universidad del B\'io-B\'io and ANID-Chile through FONDECYT project 11200529 (Chile).
	The second author was supported by Universidad de Los Lagos Regular R02/21  and ANID-Chile through FONDECYT project 1231619 (Chile).  The third author was partially supported by the National Agency for Research and Development, ANID-Chile through project Anillo of
	Computational Mathematics for Desalination Processes ACT210087, FONDECYT Postdoctorado project 3230302, and by project Centro de Modelamiento Matemático (CMM), FB210005, BASAL funds for centers of excellence.
}}}

\author{Felipe Lepe\thanks{GIMNAP-Departamento de Matem\'atica, Universidad del B\'io - B\'io, Casilla 5-C, Concepci\'on, Chile. \email{flepe@ubiobio.cl}.}
\and Gonzalo Rivera\thanks{Departamento de Ciencias Exactas,
	Universidad de Los Lagos, Casilla 933, Osorno, Chile. \email{gonzalo.rivera@ulagos.cl}.}
\and Jesus Vellojin\thanks{GIMNAP-Departamento de Matem\'atica, Universidad del B\'io - B\'io, Casilla 5-C, Concepci\'on, Chile. \email{jvellojin@ubiobio.cl}.}}

\usepackage{amsopn}
